\documentclass{amsart}

\usepackage[utf8]{inputenc}
\usepackage[T1]{fontenc}
\usepackage{lmodern}
\usepackage[english]{babel}
\usepackage{amsmath}
\usepackage{amsthm}
\usepackage{amssymb}
\usepackage{units}
\usepackage{amscd}
\usepackage{geometry}
\usepackage{graphicx}
\usepackage{tikz}
\usepackage{tikz-cd}
\usepackage{pst-plot}
\usepackage{enumitem}
\usetikzlibrary{arrows,automata}
\usepackage{graphicx}
\usepackage{url}
\usepackage{hyperref}

{\newtheorem{theo}{Theorem}[section]}
{\newtheorem{corollaire}[theo]{Corollary}}
{\newtheorem{problem}[theo]{Problem}}
{\theoremstyle{definition} \newtheorem{defin}[theo]{Definition}
							\newtheorem{prop}[theo]{Proposition}
							\newtheorem{lemme}[theo]{Lemma}}
{\theoremstyle{remark} \newtheorem{remarque}[theo]{Remark}
						\newtheorem{exemple}[theo]{Example}}
{ \newtheorem*{DuplicateWhiteheadSimplicial}{Theorem~\ref{WhiteheadSimplicial}}
\newtheorem*{DuplicateWhitehead}{Theorem~\ref{Whitehead}}
\newtheorem*{DuplicateAlmostMiller}{Theorem~\ref{AlmostMiller}}
}

\newcommand\mapsfrom{\mathrel{\reflectbox{\ensuremath{\mapsto}}}}	\newcommand{\Id}{Id}	
\newcommand{\Set}{\text{Set}}
\newcommand{\op}{op}
\newcommand{\sS}{\text{sSet}}

\newcommand{\Hom}{\text{Hom}}

\newcommand{\Top}{\text{Top}}
\newcommand{\Real}[1]{{}||#1||}
\newcommand{\RealP}[1]{{}\Real{#1}_P}
\newcommand{\Sing}{\text{Sing}}

\newcommand{\pr}{pr}
\newcommand{\colim}{\operatornamewithlimits{colim}}
\newcommand{\Ex}{\text{Ex}}
\newcommand{\sd}{\text{sd}}

\newcommand{\Exi}{\text{Ex}^{\infty}}
\newcommand{\lv}{\text{l.v}}
\newcommand{\Map}{\text{Map}}
\newcommand{\Fun}{\text{Fun}}

\newcommand{\G}{\mathcal{G}}
\newcommand{\nd}{\text{n.d.}}

\newcommand{\holim}{\text{holim}}
\newcommand{\Grp}{\text{Grp}}
\newcommand{\ev}{\text{ev}}
\newcommand{\Cont}{\mathcal{C}^0}
\newcommand{\Hol}{\text{Holink}}
\newcommand{\proj}{\text{pr}}
\newcommand{\OSk}{\text{OSk}}
\newcommand{\Sp}{\text{Sp}}
\newcommand{\Z}{\mathbb{Z}}

\newcommand{\An}{\text{An}}
\newcommand{\C}{\mathcal{C}}

\newcommand{\Ei}{E^{\infty}}
\newcommand{\Er}{\text{Er}}
\title{A simplicial approach to stratified homotopy theory}
\author{Sylvain Douteau}
\begin{document}
\begin{abstract}
In this article we consider the homotopy theory of stratified spaces from a simplicial point of view. We first consider a model category of filtered simplicial sets over some fixed poset $P$, and show that it is a simplicial combinatorial model category. We then define a generalization of the homotopy groups for any fibrant filtered simplicial set $X$ : the filtered homotopy groups $s\pi_n(X)$. They are diagrams of groups built from the homotopy groups of the different pieces of $X$. We then show that the weak equivalences are exactly the morphisms that induce isomorphisms on those filtered homotopy groups.

Then, using filtered versions of the topological realisation of a simplicial set and of the simplicial set of singular simplices, we transfer those results to a category whose objects are topological spaces stratified over $P$. In particular, we get a stratified version of Whitehead's theorem. 
Specializing to the case of conically stratified spaces, a wide class of topological stratified spaces, we recover a theorem of Miller saying that to understand the homotopy type of conically stratified spaces, one only has to understand the homotopy type of strata and holinks. We then provide a family of examples of conically stratified spaces and of computations of their filtered homotopy groups. 
\end{abstract}
\maketitle

\tableofcontents

\section*{Introduction}

Stratified spaces appear in many domains of geometry and topology as natural objects to study singular spaces. Stratification theory has its roots in the seminal works of Thom and Whitney on decompositions of algebraic and analytic varieties into smooth pieces called strata (\cite{Thom}, \cite{Whitney}, see also \cite{GoreskyIntro}, \cite{Mather}).
Nowadays, stratifications appear through poset-stratified spaces (\cite{WoolfFundamentalCategory}) - spaces with a continuous map to a poset with the Alexandrov topology, $X\to P$ -  with a particular focus on conically stratified spaces \cite{HigherAlgebra}, \cite{AyalaFrancisTanakaLocalStructure}.

Moreover, in the 1970s, motivated by the theory of characteristic classes of singular spaces, Goresky and MacPherson introduced intersection cohomology in order to restore Poincaré duality in the framework of stratified spaces \cite{IntersectionHomologyI}, \cite{IntersectionHomologyII}. 
Since then, intersection cohomology has become an important tool in complex algebraic geometry and in geometric representation theory as the building block of the category of perverse sheaves \cite {BBD} \cite{Dimca} \cite{KirwanWoolf} .

Since the introduction of intersection homology, other invariants of stratified spaces have been developed. For examples :
\begin{itemize}
\item The Exit Paths Category of a stratified space, and its higher categorical generalizations have been considered by Treumann \cite{Treumann}, Woolf \cite{WoolfFundamentalCategory} and also Lurie \cite{HigherAlgebra}.
\item Intersection homotopy groups have been introduced by Gajer \cite{Gajer}.
\item Singular chain complexes and cochain complexes together with cap and cup products have been developed for intersection homology and cohomology by King \cite{King} and Friedman and McClure \cite{FriedmanMcClure} \cite{Friedman}
\item A rational homotopy theory, in the spirit of Sullivan's theory has been developed by Chataur, Saralegui and Tanré in \cite{MemoireDavid}.
\item Banagl has also introduced rational generalized intersection homology theories \cite{BanaglRationalGeneralizedHomology}.
\end{itemize}

But those objects are not invariants in a classical sense; they are only invariant under \textbf{stratified} homotopy equivalences.  Which begs the question : What is the homotopy category through which those invariants factor? This question was originally formulated about intersection cohomology by Goresky and MacPherson \cite[Problem 4]{Borel}, and is also relevant for invariants of stratified spaces that have been constructed since. Furthermore, once a homotopy category for stratified spaces is known, one can start to investigate which functors defined on stratified spaces are representable in a homotopy-theoretical sense. The first example of interest - for which an answer is not known yet - is intersection cohomology. Goresky and MacPherson asked if it was representable, and an affirmative answer could help define generalized intersection cohomology theories through some theory of stratified spectra see \cite[Problems 1 and 11]{Borel}. 

But a homotopy category of stratified spaces also opens up the possibility to study more general functors from the category of stratified spaces.
For example, let $G$ be a compact Lie group, and $C$ a $G$-CW complex. The quotient $Q=C/G$ comes equipped with a natural stratification over $P_{G}$, the poset of conjugacy classes of closed subgroup of $G$, defined by 
\begin{align*}
Q&\to P_G\\
x&\mapsto [\text{Stab}_x]
\end{align*}
where $[\text{Stab}_x]$ is the conjugacy class of stabilizers of preimages of $x$ in $C$. Now, given $R$, some $CW$ complex stratified over $P_G$, and any stratified cellular map $f\colon R\to Q$, the pullback $f^{-1}(C)$ is a $G$-CW complex whose quotient is $R$. This means that there is a functor associating to any $CW$ complex stratified over $P_G$, $Q$, the set of isomorphism classes of $G-CW$ complex with quotient $Q$,
\begin{align*}
\text{$P_G$-stratified CW complexes}&\to \Set^{\op}\\
Q&\mapsto \{C\ |\ C/G\simeq Q\}/{\sim}
\end{align*}
This functor has a well-known "classical" counterpart, the (sub)functor mapping a (trivially stratified) $CW$-complex, $Q$, to the set of principal $G$-bundle with basis $Q$. In the homotopy category, it is represented by $BG$, the classifying space of $G$. One can ask the same question about the stratified version of this functor, and although the answer is unknown yet, the homotopy category of $P_G$-stratified spaces provides a setting to work on this problem.

Relatedly, Ayala, Francis and Tanaka defined conically smooth stratified spaces, and a notion of constructible smooth bundle between them \cite[Definition 3.6.1]{AyalaFrancisTanakaLocalStructure}. Such a constructible bundle is a map $f\colon X\to Y$ restricting to a (smooth) fiber bundle on each strata, and on each link. In \cite[Section 6.3]{AyalaFrancisRozenblyumStratifiedHomotopyHypothesis}, the authors prove that constructible bundle over some fixed conically smooth stratified space $Y$ are classified in an $\infty$-categorical sense by some $\infty$-category. This suggests in particular that constructible bundles with fixed fibers for each strata (and link) should be classified by some stratified object in a homotopy category of stratified spaces. One could also compare this notion with the stratified fiber bundles introduced by Baues and Ferrario in \cite{BauesFerrario}.


In this article, following the Kan-Quillen strategy, we propose a model for the homotopy category of stratified spaces : the combinatorial model category of filtered simplicial sets over a poset $P$. Our goal is to provide a categorical context in which the aforementioned problems can be studied. To do so, we work with the presheaf category of simplicial sets stratified over some fixed poset $P$. Since we work over a fixed poset, we use the adjective \textbf{filtered}, instead of stratified. Equipped with the definition of filtered homotopy equivalences - equivalently, the definition of filtered cylinders - we construct a model structure on $\sS_P$ using Cisinski's theory. It is the Cisinski model structure having the smallest possible class of weak equivalences containing filtered homotopy equivalences (see Proposition \ref{AdmissibleHornHomotopyEquivalence}). We then prove \textit{a posteriori} that weak equivalences are detected by some diagrams of groups (see Definition \ref{FilteredHomotopyGroupsSimplicial} and Theorem \ref{WhiteheadSimplicial}). This contrast with the approach in Henriques' unpublished work \cite{Henriques}, where the author starts from the data of a category of diagrams of simplicial sets, and uses it to define a class of weak equivalences on $\sS_P$.
Thanks to this model structure, we were able to prove some topological results on stratified spaces, see Theorem \ref{Whitehead} and Theorem \ref{AlmostMiller}. 

\subsection*{Outline of the paper}
In section \ref{SectionCMFsSetP}, we define the presheaf category of filtered simplicial sets. Then, starting with a notion of filtered homotopy equivalences coming from the usual construction of a filtered cylinder object, we apply Cisinski's work \cite[Théorème 1.3.22]{Cisinski} to produce a cofibrantly generated model category in which cofibrations are monomorphisms (Theorem \ref{ApplyCisinski}).

But Theorem \ref{ApplyCisinski} does not give us a full characterization of the fibrations. In particular, it is not immediate that fibrations are defined "à la Kan" by lifting properties against admissible horn inclusions - as would be expected. In section \ref{SectionCharacterizingFibration}, we prove that this characterization holds for $\sS_P$. We do so by constructing filtered analogues to the subdivision and extensions functors - the filtered subdivision, $\sd_P$ (Definition \ref{FilteredSubdivision}), and its adjoint $\Ex_P$ (Definition \ref{DefinitionEx}) - and proving that they satisify suitable properties. The characterization of the model category $\sS_P$ (Theorem \ref{TheoremCharacterizationsSetP}) is then an application of the more general result, Theorem \ref{TheoAppendixA}, proven in Appendix \ref{AppendixA}

In section \ref{SimplicialStructure}, we show that the model category is simplicial and that we actually have a combinatorial model category, see Theorem \ref{SimplicialModelCategory}. Furthermore, for $X$ a fibrant filtered simplicial set, using the simplicial category structure, we define $s\pi_n(X)$ the filtered homotopy groups of $X$. For $n\geq 0$, $s\pi_n(X)$ is a diagram of groups (or sets, if $n=0$) built from the $n$-th homotopy groups of the filtered pieces of $X$ such as the strata and the links. When $X$ is trivially filtered, we recover the usual homotopy groups of $X$ but in general the structure is much richer (See Remark \ref{RemarqueFilteredHomotopyGroupsSimplicialSets}). In particular, we get that the filtered homotopy groups completely characterize the filtered homotopy equivalences between fibrant filtered simplicial sets. More precisely, we have the following result :

\begin{DuplicateWhiteheadSimplicial}
Let $f\colon X\to Y$ be a filtered map between fibrant filtered simplicial sets. It is a filtered homotopy equivalence if and only if the induced maps $ s\pi_n(f)\colon  s\pi_n(X,\phi)\to  s\pi_n(Y,f\circ\phi)$ and $ s\pi_0(f)\colon  s\pi_0(X)\to  s\pi_0(Y)$ are isomorphisms for all $n\geq 1$ and all pointing $\phi\colon V\to X$.
\end{DuplicateWhiteheadSimplicial}

In section \ref{FilteredTopologicalSpaces}, we continue to adapt the Kan-Quillen strategy, and introduce a pair of adjoint functors, $\Real{-}_P\colon \sS_P\to \Top_P$, the filtered realisation, and $\Sing_P\colon \Top_P\to\sS_P$, the filtered simplicial set of singular simplices. Here $\sS_P$ is the category of simplicial sets filtered over some poset $P$, and $\Top_P$ is the category of topological spaces filtered over $P$. The object of the latter category are in fact stratified spaces, and the adjective filtered reflects the fact that the morphisms of this category preserve the filtrations.

Contrary to the classical case, for a stratified space $A$, the filtered simplicial set $\Sing_P(A)$ is not fibrant in general (see Example \ref{ExampleNonFibrantSpaces}). For this reason, we restrict our attention to spaces such that $\Sing_P(A)$ is fibrant. Proposition \ref{ConStratFibrant} guarantees that such spaces exist and that, in fact, most of the classical examples of stratified spaces (such as topological pseudomanifolds) are of this kind.
In addition, we extend the definition of filtered homotopy groups to filtered spaces (Definition \ref{TopFilteredHomotopyGroup}).
 We are then able to prove a filtered version of Whitehead's theorem :

\begin{DuplicateWhitehead}
Let $f\colon A\to B$ be a filtered map between two fibrant filtered spaces such that $A$ and $B$ are filtered homeomorphic to the realisation of some filtered simplicial sets $X$ and $Y$ respectively. Then $f$ is a filtered homotopy equivalence if and only if the maps of functors $ s\pi_0(f)\colon  s\pi_0(A)\to  s\pi_0(B)$ and $ s\pi_n(f)\colon  s\pi_n(A,\phi)\to s\pi_n(B,f\circ\phi)$ are isomorphisms for all $n\geq 0$ and all pointings of $A$. 
\end{DuplicateWhitehead}

In section \ref{TowardATheoremOfMiller}, we show that the hypothesis of Theorem \ref{Whitehead}, can be strengthened to have a simpler conclusion. If we restrict to the context of PL conically stratified spaces, see \cite[Definition A.5.5]{HigherAlgebra}, we show that it is enough to understand morphisms between filtered spaces only at the level of strata and holinks to detect a homotopy equivalence. The conclusion is identical to the one in \cite[Theorem 6.3]{Miller}, but the class of objects for which the theorems apply are different, although they both include most notions of pseudo-manifolds. It is to be noted that the proofs are completely independent.
\begin{DuplicateAlmostMiller}
Let $f\colon A \to B$ be a filtered map between conically stratified spaces, such that $A$ and $B$ are isomorphic to the filtered realisations of two filtered simplicial sets, $X$ and $Y$ respectively. Then $f$ is a filtered homotopy equivalence if and only if it induces weak equivalences between corresponding strata and holinks.
\end{DuplicateAlmostMiller}

This theorem applies in particular to filtered continuous maps between PL-stratified spaces, which covers many geometric examples such as algebraic varieties.

In section \ref{ApplicationsAndExamples}, we consider the invariants of the homotopy type of stratified spaces which can be extracted from the constructions of this article. Then, we propose a geometric construction which produces fibrant stratified spaces and we use it to produce examples of computations of filtered homotopy groups and of applications of Theorems \ref{Whitehead} and \ref{AlmostMiller}. We start with two explicit geometric examples, subsections \ref{MoebiusCylindre} and \ref{PseudoManifoldExample}, then move on to a family of examples inspired by filtered classifying spaces of diagrams of groups.

Stephen Nand-Lal and Jon Woolf have independently worked on a similar approach to the homotopy theory of stratified spaces, but where the poset determining the filtration is not fixed. In their project the combinatorial model is provided by the category of simplicial sets equipped with the Joyal model structure, see \cite{NandLal}. It would be interesting to understand how this relates to the model structure on slice categories of simplicial sets used in this paper.

Independantly, Peter Haine constructed another model structure on the category $\sS_P$ as a localization of the joyal model structure \cite{Haine}. In his model structure, the cofibrations are the monomorphisms, and filtered homotopy equivalences are included in the class of weak equivalences. In particular, it is a localization of the one we present here.
\section*{Acknowledgements}
The author wishes to thank David Chataur for his careful reading of the manuscript which led to many improvements, and for the discussions on several proofs of this article. The author also wishes to thank Jon Woolf and Stephen Nand-Lal for the fruitful conversations on the subjects of stratified homotopy and pointings. The author is also thankful for the exchanges with Dimitri Ara, Serge Bouc, Greg Friedman, Paul Goerss and Sean Moss during the preparation of this article.

\section{The model category of filtered simplicial sets}\label{FilteredSimplicialSets}
\label{SectionCMFsSetP}

In this section, we define the category of filtered simplicial sets, and use \cite[Théorème 1.3.22]{Cisinski} to prove that there exists a cofibrantly generated model structure on it.

\begin{defin}
We define the category of filtered simplicial sets over the poset $P$, $\sS_P$ to be the category of simplicial sets over $N(P)$. Objects of $\sS_P$ are pairs $(X,\pi)$ where $X$ is a simplicial set, and $\pi\colon X\to N(P)$ is a morphism of simplicial sets, and will be called a filtration on $X$. A morphism in $\sS_P$ from $(X,\pi)$ to $(Y,\nu)$ is a morphism of simplicial sets $f\colon X\to Y$ such that $\nu\circ f= \pi$. We will call such morphism "filtered" and often write $X$ instead of $(X,\pi)$.
\end{defin}

\begin{defin}
Let $\Delta(P)$ be the full subcategory of $\sS_P$ whose objects are of the form 
\begin{equation*}
\Delta^N\xrightarrow{\pi} N(P)
\end{equation*}
for $N\geq 0$.
Let $J=(j_p)_{p\in P}$ be a sequence of integers greater or equal than $-1$ with all values equal to $-1$ outside of a strictly increasing chain $p_0<\dots<p_d\in P$. Then we write $\Delta^J$ to designate the object of $\Delta(P)$ such that $\pi(\Delta^J)=[p_0,\dots,p_0,p_1,\dots,p_d]$, where $p_i$ appears $j_{p_i}+1$ times.

We call the $\Delta^J$s filtered simplices, and we will identify them with their image in $N(P)$. In particular, we will write $\Delta^J=[p_0,\dots,p_0,p_1,\dots,p_d]$. When we need to refer to the non-filtered simplex underlying the definition of $\Delta^J$, we will write $N(J)=\sum_{p\in P}(j_p+1) -1$ (or just $N$ when no confusion is possible), and then 
\begin{equation*}
\Delta^J= \Delta^{N(J)}\to N(P).
\end{equation*}
When usefull, we will also write $\Delta^{J_0}$ for the corresponding non-degenerate simplex in $N(P)$. We will have $\Delta^{J_0}=[p_0,p_1,\dots,p_d]$ where each $p_i$ appears exactly once.
\end{defin}

\begin{prop}\label{PresheavesCategory}
The category of filtered simplicial sets over $P$ is isomorphic to the category of presheaves over $\Delta(P)$, i.e. $\Fun(\Delta(P)^{\text{op}},\Set)$.
\end{prop}

\begin{proof}
It is true for any category of presheaves $\mathcal{C}$, and any object $S$ of $\mathcal{C}$,
that $\mathcal{C}_{\downarrow{} S}$ is a category of presheaves of this form. 
To see this, let $\G$ be a small category, $\mathcal{C}$ the category of presheaves over $\G$, and $S$ an object of $\mathcal{C}$.
Define $\G(S)$ to be the full subcategory of $\mathcal{C}_{\downarrow{} S}$ whose objects are of the form $g\to S$, where $g\in \G$ is identified with $\Hom(-,g)\in \mathcal{C}$. Then, the functor 
\begin{align*}
\mathcal{C}_{\downarrow{} S} & \to \Fun(\G(S)^{\text{op}},\Set)\\
(X\to S) &\mapsto \left\{\begin{array}{rcl}
\G(S)^{\text{op}}&\to& \Set\\
(g\to S)&\mapsto& \Hom_{\mathcal{C}_{\downarrow{} S}}(g\to S,X\to S)
\end{array}\right.
\end{align*}
is an equivalence of categories.
\end{proof}

\begin{defin}
Let $X$ be a filtered simplicial set, we write $X_J=X(\Delta^J)$.
The elements of $X_J$ will be called $J$-simplices of $X$, and we will often identify those elements with the corresponding maps $\Delta^J\to X$.
\end{defin}
As a consequence of this definition and of Proposition \ref{PresheavesCategory}, we have that for any filtered simplicial set, $X$,
\begin{equation*}
X\simeq \colim_{\Delta^J\to X}\Delta^J.
\end{equation*}

\begin{prop}\label{ForgetFree}
We have a pair of adjoint functors $(\Er,\Pr)$
\begin{eqnarray*}
\Er\colon & \sS_P&\to \sS\colon \Pr\\
&(X\to N(P))&\mapsto X\\
&(K\times N(P)\to N(P))&\mapsfrom K
\end{eqnarray*}
\end{prop}

\begin{defin}\label{InnerProductFilteredSimplicialSets}
Let $X\xrightarrow{\pi_X} N(P)$ and $Y\xrightarrow{\pi_Y} N(P)$ be two filtered simplicial sets. We define the filtered product of $X$ and $Y$ as  the following pullback 
\begin{equation*}
\begin{tikzcd}
X\times_{N(P)} Y
\arrow{d}{\pr_X}
\arrow{r}{\pr_Y}
& Y
\arrow{d}{\pi_Y}
\\
X
\arrow{r}{\pi_X}
& N(P)
\end{tikzcd}
\end{equation*}
where the filtration $X\times_{N(P)} Y\to N(P)$ is given by the composition $\pi_X\circ \pr_X=\pi_Y\circ \pr_Y$. Since we will never consider in this article non-filtered product of filtered simplicial sets, we will drop the $N(P)$ from the notation.
This extends to a bi functor
\begin{align*}
-\times - \colon \sS_P\times \sS_P&\to \sS_P\\
(X,Y)&\mapsto X\times Y
\end{align*}
\end{defin}

\begin{defin}\label{TensorProduct}
We define the bi functor 
\begin{align*}
-\otimes-\colon \sS_P\times \sS &\to \sS_P\\
(X,K)&\mapsto X\times \Pr(K)
\end{align*}
where $\times$ is the filtered product of Definition \ref{InnerProductFilteredSimplicialSets}. 
\end{defin}

We now have a functorial cylinder object : For any $X$, $X\otimes \Delta^1$ is a cylinder object of $X$. This naturally leads to the following definition of a filtered homotopy.

\begin{defin}\label{DefHomotopy}
Let $f,g\colon X\to Y$ be two morphisms of filtered simplicial sets. We say that $f$ and $g$ are elementarily filtered homotopic to each other if there exists a filtered map $H\colon X\otimes \Delta^1\to Y$ making the following diagram commutative :

\begin{equation*}
\begin{tikzcd}
X
\arrow[swap]{dr}{\Id\otimes\iota_0}
\arrow[bend left=12]{drr}{f}
&\phantom{X} 
&\phantom{X}
\\
\phantom{X}
& X\otimes \Delta^1
\arrow{r}{H}
& Y
\\
X
\arrow{ur}{\Id\otimes \iota_1}
\arrow[swap, bend right=12]{urr}{g}
&\phantom{X}
&\phantom{X}
\end{tikzcd}
\end{equation*}
Closing this relation by symmetry and transitivity, we get the definition of filtered homotopy between morphisms.
We write $[X,Y]$ for the set of morphisms from $X$ to $Y$ quotiented by this equivalence relation.

If $f\colon X\to Y$ and $g\colon Y\to X$ are two morphisms of filtered simplicial sets such that $g\circ f$ is filtered homotopic to $\Id_X$ and $f\circ g$ is filtered homotopic to $\Id_Y$, we say that $(f,g)$ is a filtered homotopy equivalence (sometimes we will just say that $f$ is a filtered homotopy equivalence), and that $X$ is filtered homotopy equivalent to $Y$. We will say that $X$ and $Y$ are filtered elementarily homotopy equivalent if the homotopies can be chosen to be elementary. When no confusion is possible, we will omit the adjective filtered.
\end{defin}

\begin{defin}
Let $\Delta^J=\Delta^{N}\to N(P)$ be a filtered simplex and $0\leq k\leq N$. We define the $k$-th horn of $\Delta^J$ as $\Lambda^J_k=\Lambda^{N}_k\to N(P)$, where $\Lambda^{N}_k$ is the $k$-th horn of $\Delta^{N}$, and the filtration is given by the composition $\Lambda^N_k\hookrightarrow\Delta^N\to N(P)$. If $\Delta^J=[q_0,\dots,q_N]$, with $q_0\leq q_1\leq \dots\leq q_n\in P$, we will write $\Lambda^J_k=[q_0,\dots,\widetilde{q}_k,\dots,q_N]$.
\end{defin}


\begin{defin}\label{AdmissibleHorn}
A horn inclusion $\Lambda^J_{k}\hookrightarrow \Delta^J$ is said to be admissible 
If $\Delta^J=s_md_k\Delta^J$ for either $m=k-1$ or $m=k$. i.e. precisely when $\Delta^J$, seen as a simplex of $N(P)$, is degenerate with its $k$th vertex repeated.
\end{defin}

\begin{exemple}
To get a better idea of what admissible horn inclusions are, let us look at a few examples. Consider the poset $P=\{p_0<p_1<p_2\}$, we consider the four following horn inclusions 
\begin{align*}
[p_0,p_1,\widetilde{p_2}]\hookrightarrow [p_0,p_1,p_2]\\
[\widetilde{p_0},p_1,p_1]\hookrightarrow [p_0,p_1,p_1]\\
[p_0,p_1,\widetilde{p_1}]\hookrightarrow [p_0,p_1,p_1]\\
[p_0,\widetilde{p_0},p_1]\hookrightarrow [p_0,p_0,p_1]\\
\end{align*}
The first two horn inclusions are not admissible. Indeed, in the first one, there is only $1$ point of color $p_2$, and so the horn is not admissible. In the second example, there is only $1$ point of color $p_0$, so the horn is not admissible either. In the third and fourth example, however, there are two points of color $p_1$ (resp. $p_0$) and so the horns are admissible.
To better understand what differentiates those horns, we drew a picture of all four. See Figure \ref{FigureHorns}. 
\end{exemple}

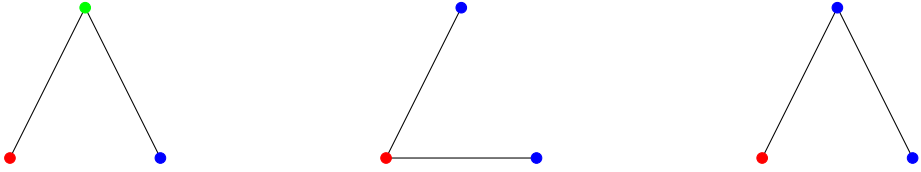
\begin{figure}[h]
\begin{tikzpicture}
\draw[black](-1,-1) -- (0,1);
\draw[black](0,1) -- (1,-1);
\filldraw[green](0,1) circle (2pt);
\filldraw[blue] (1,-1) circle (2pt);
\filldraw[red] (-1,-1) circle (2pt);

\draw[black,cm ={1,0,0,1,(4cm,0cm)}](-1,-1) -- (1,-1);
\draw[black,cm ={1,0,0,1,(4cm,0cm)}](-1,-1) -- (0,1);
\filldraw[blue,cm ={1,0,0,1,(4cm,0cm)}](0,1) circle (2pt);
\filldraw[blue,cm ={1,0,0,1,(4cm,0cm)}] (1,-1) circle (2pt);
\filldraw[red,cm ={1,0,0,1,(4cm,0cm)}] (-1,-1) circle (2pt);

\draw[black,cm ={1,0,0,1,(8cm,0cm)}](-1,-1) -- (0,1);
\draw[black,cm ={1,0,0,1,(8cm,0cm)}](0,1) -- (1,-1);
\filldraw[blue,cm ={1,0,0,1,(8cm,0cm)}](0,1) circle (2pt);
\filldraw[blue,cm ={1,0,0,1,(8cm,0cm)}] (1,-1) circle (2pt);
\filldraw[red,cm ={1,0,0,1,(8cm,0cm)}] (-1,-1) circle (2pt);

\draw[black,cm ={1,0,0,1,(12cm,0cm)}](-1,-1) -- (1,-1);
\draw[black,cm ={1,0,0,1,(12cm,0cm)}](0,1) -- (1,-1);
\filldraw[blue,cm ={1,0,0,1,(12cm,0cm)}](0,1) circle (2pt);
\filldraw[red,cm ={1,0,0,1,(12cm,0cm)}] (1,-1) circle (2pt);
\filldraw[red,cm ={1,0,0,1,(12cm,0cm)}] (-1,-1) circle (2pt);

\end{tikzpicture}
\caption{From left to right, $[p_0,p_1,\widetilde{p_2}]$, $[\widetilde{p_0},p_1,p_1]$, $[p_0,p_1,\widetilde{p_1}]$ and $[p_0,\widetilde{p_0},p_1]$}
\label{FigureHorns}
\end{figure}

\begin{remarque}
Taking a closer look at Figure \ref{FigureHorns}, and going back to Definitions \ref{AdmissibleHorn} and \ref{DefHomotopy}, one notices that admissible horns inclusions are precisely the horn inclusions that are homotopy equivalences. The hypothesis that the $k$th vertex is repeated guarantees that there is some proper face of the simplex $\Delta^J$ which is homotopy equivalent to the whole simplex (notice that in this case, there must be at least two). Furthermore, it imposes that the horn still contains such a face and can still be contracted onto it. In the first example of Figure \ref{FigureHorns}, there is no such face. And in the second, such faces exist but the horn is no longer homotopy equivalent to those. In the third, the horn can still be contracted onto the $1$-simplex with one red vertex and one blue vertex, see Figure \ref{FigureAdmissibleHornHomotopy}. This intuition should serve as a motivation for this choice of admissible horn inclusion and Proposition \ref{AdmissibleHornHomotopyEquivalence} gives a precise statement of this fact. The fact that the classical proofs on anodyne extensions translate directly for admissible horn inclusions should be another reason to be convinced that this choice is relevant. See Proposition \ref{AdmissibleHornAnodyneExt}.
\end{remarque}

\begin{figure}
\begin{tikzpicture}
\filldraw[gray] (0,1)--(1,-1)--(-1,-1)--(0,1);
\draw[black](0,1)--(1,-1)--(-1,-1)--(0,1);
\filldraw[blue](0,1) circle (2pt);
\filldraw[blue] (1,-1) circle (2pt);
\filldraw[red] (-1,-1) circle (2pt);
\draw[->,thick](0.9,-0.8)--(0.1,0.8);
\draw[->,thick,cm ={1,0,0,1,(-0.6cm,0cm)}](0.9,-0.8)--(0.4,0.2);
\draw[->,thick,cm ={1,0,0,1,(-1.2cm,0cm)}](0.9,-0.8)--(0.7,-0.4);

\draw[->,thick] (1.5,0)-- (3.5,0);

\draw[black,cm ={1,0,0,1,(5cm,0cm)}](-1,-1) -- (0,1);
\filldraw[blue,cm ={1,0,0,1,(5cm,0cm)}](0,1) circle (2pt);
\filldraw[red,cm ={1,0,0,1,(5cm,0cm)}] (-1,-1) circle (2pt);

\draw[->,thick,cm ={1,0,0,1,(10cm,0cm)}](-1.5,0)--(-3.5,0);

\draw[black,cm ={1,0,0,1,(10cm,0cm)}](-1,-1) -- (0,1);
\draw[black,cm ={1,0,0,1,(10cm,0cm)}](0,1) -- (1,-1);
\filldraw[blue,cm ={1,0,0,1,(10cm,0cm)}](0,1) circle (2pt);
\filldraw[blue,cm ={1,0,0,1,(10cm,0cm)}] (1,-1) circle (2pt);
\filldraw[red,cm ={1,0,0,1,(10cm,0cm)}] (-1,-1) circle (2pt);
\draw[->,thick,cm ={1,0,0,1,(10cm,0cm)}](0.9,-0.8)--(0.1,0.8);
\end{tikzpicture}
\caption{The filtered homotopy equivalence between a simplex and one of its face and its restriction to an admissible horn}
\label{FigureAdmissibleHornHomotopy}
\end{figure}
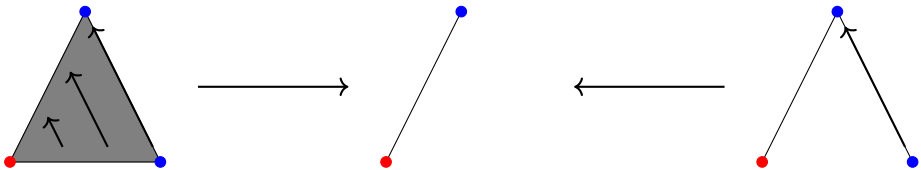

\begin{prop}\label{AdmissibleHornHomotopyEquivalence}
Let $\Lambda^J_k\to \Delta^J$ be any horn inclusion. It is an admissible horn inclusion if and only if it is an elementary filtered homotopy equivalence.
\end{prop}

\begin{proof}
Assume that the horn $\Lambda^J_k\to \Delta^J$ is admissible. Then $\Delta^J=[q_0,\dots,q_{k-1},q_k,q_{k+1},\dots,q_{N}]$ and either $q_k=q_{k-1}$ or $q_k=q_{k+1}$. Write $(\Delta^J)_0=\{e_0,\dots,e_{N}\}$ for the set of vertices of $\Delta^J$. Assume that $q_k=q_{k-1}$ and define a homotopy inverse at the level of vertices
\begin{align*}
(\Delta^J)_0 &\to (\Lambda^J_k)_0\\
e_i &\mapsto \left\{\begin{array}{cl}
e_i &\text{ if $i\not = k-1$}\\
e_k &\text{ if $i=k-1$}
\end{array}\right.
\end{align*}
this map extends to a map of simplicial sets with image included in $d_{k-1}(\Delta^J)\subset \Lambda^J_k$. Furthermore, since $q_{k-1}=q_k$, it is a map of filtered simplicial sets.
The compositions $\Delta^J\to \Lambda^J_k\to \Delta^J$ corresponds to sending $\Delta^J$ to $d_{k-1}\Delta^J\subset \Delta^J$. a homotopy with the identity is given, at the level of vertices, by the following 
\begin{align*}
(\Delta^J\otimes \Delta^1)_0&\to (\Delta^J)_0\\
(e_i,s)&\mapsto \left\{ \begin{array}{cl}
e_i &\text{ if $s=0$ or ($s=1$ and $i\not=k-1$)}\\
e_k &\text{ if $s=1$ and $i=k-1$}
\end{array}\right.
\end{align*}
One checks that this gives a well defined homotopy with the identity. Furthermore, the restriction of the map to $\Lambda^J_k$ gives a homotopy between the composition $\Lambda^J_k\to \Delta^J\to \Lambda^J_k$ and $\Id_{\Lambda^J_k}$. The proof works the same in the case $q_k=q_{k+1}$, one has to replace $k-1$ by $k+1$ everywhere and reverse the role of $s=0$ and $s=1$ in the homotopy.

To show the converse, assume that $j\colon\Lambda^J_k\to \Delta^J$ is any horn inclusion, and that there is a filtered homotopy inverse $r\colon\Delta^J\to \Lambda^J_k$. By symmetry, suppose that the elementary homotopy $H\colon \Lambda^J_k\otimes \Delta^1\to \Lambda^J_k$ is a homotopy from $r\circ j$ to $\Id$, that is $H_{|\Lambda^J_k\otimes \{0\}}= r\circ j$ and $H_{|\Lambda^J_k\otimes \{1\}}= \Id$. This means that if $e_i$ is a vertex of $\Lambda^J_k$, we have $r(e_i)=e_{i'}\leq e_i$, where the order is the usual partial order on the vertices of a simplicial set. Suppose that $q_k\not= q_{k-1},q_{k+1}$, this means that $r(e_k)=e_k$, since $r$ is filtered. Define $l=\min\{i\ |\ r(e_i)\not=e_i\}$. By the previous remark, $l\not= k$. Furthermore, we have that $r(e_l)=e_{l'}<e_l$. But then, since $r$ is a simplicial map, we must have $r(e_{l-1})=e_{l-1}\leq e_{l'}<e_l$. This means that $r(e_l)=e_{l-1}$. In addition, we must have $q_{l-1}=q_l\not = q_k$. Now consider the image by $H$ of the following simplex of $d_{l-1}\Delta^J\otimes \Delta^1\subset \Lambda^J_k\otimes \Delta^1$.
\begin{align*}
H([(e_0,0),\dots,\widehat{(e_{l-1},0)},(e_l,0),(e_l,1),\dots(e_N,1)]) &= [r(e_0),\dots,\widehat{r(e_{l-1})},r(e_l),e_l,\dots,e_N]\\
&= [e_0,\dots,e_{l-1},e_l,\dots,e_N]\\
&=\Delta^J\not\subset\Lambda^J_k
\end{align*}
\end{proof}

\begin{remarque}\label{RemarkElementaryNotNeeded}
Proposition \ref{AdmissibleHornHomotopyEquivalence} still holds if one replaces "elementary filtered homotopy equivalence" by "filtered homotopy equivalence". Indeed, if $i\colon \Lambda^J_k\to\Delta^J$ is a filtered homotopy equivalence, so must be its realization $\RealP{i}\colon \RealP{\Lambda^J_k}\to\RealP{\Delta^J}$ (see Definition \ref{DefinitionFilteredRealisation}, any zig-zag of elementary filtered homotopy equivalences realizes as a filtered homotopy equivalence). By Corollary \ref{DirectWhitehead}, this implies that $\RealP{\Lambda^J_k}$ and $\RealP{\Delta^J}$ have isomorphic filtered homotopy groups. In particular, $i$ must induce a weak equivalence of simplicial sets :
\begin{equation*}
\Map(\RealP{\Delta^J},\RealP{\Lambda^J_k})\to\Map(\RealP{\Delta^J},\RealP{\Delta^J})
\end{equation*}
The codomain of this map is always contractible, whereas one shows that if $e_k$ is the only vertex of $\Delta^J$ of color $p_k$ - that is, if $\Lambda^J_k$ is not admissible - we have the weak equivalence
\begin{equation*}
\Map(\RealP{\Delta^J},\RealP{\Lambda^J_k})\simeq \Map(\RealP{d_k(\Delta^J)},\RealP{\partial(d_k(\Delta^J))})
\end{equation*}
In particular, $\Map(\RealP{\Delta^J},\RealP{\Lambda^J_k})$ is not contractible in this case.
\end{remarque}

Recall the definition of a class of anodyne extension from \cite[Définition 1.3.10.]{Cisinski}. We state it for a class of morphisms in $\sS_P$.

\begin{defin}[Anodyne Extensions]\label{AnodyneExtension}
Let $\Lambda$ be a class of morphisms in $\sS_P$. It is a class of anodyne extensions if it satisfies the following conditions.
\begin{itemize}
\item (An0) There exists a set of morphisms $A$ such that $\Lambda$ is the saturated class generated by $A$.
\item (An1) If $X\hookrightarrow Y$ is a monomorphism, then $X\otimes \Delta^1\cup Y\otimes \{\epsilon\}\hookrightarrow Y\otimes \Delta^1$ is in $\Lambda$. Where $\{\epsilon\}$ can be either vertex of $\Delta^1$.
\item (An2) If $X\hookrightarrow Y$ is in $\Lambda$, then so is $X\otimes \Delta^1\cup Y\otimes \partial(\Delta^1)\hookrightarrow Y\otimes \Delta^1$.
\end{itemize}
\end{defin}

\begin{prop}\label{AdmissibleHornAnodyneExt}
Let $A$ be the set of admissible horn inclusions. 
The saturated class, $\Lambda$, generated by $A$ is a class of anodyne extensions.
\end{prop}

\begin{proof}
The axiom (An0) is verified by definition of $\Lambda$. The Axiom (An1) follows from lemma \ref{An1} and the axiom (An2) is a special case of lemma \ref{An2}
\end{proof}

\begin{lemme}\label{An1}
Let $B$ be the set of all inclusions of the form 
\begin{equation*}
\partial(\Delta^J)\otimes \Delta^1\cup \Delta^J\otimes\{\epsilon\}\hookrightarrow \Delta^J\otimes\Delta^1,
\end{equation*}
where $\Delta^J\in N(P)$, $\epsilon=0,1$ and $\{\epsilon\}$ is the corresponding vertex of $\Delta^1$,
and let $C$ be the class of all inclusions of the form 
\begin{equation*}
X\otimes \Delta^1\cup Y\otimes \{\epsilon\}\hookrightarrow Y\otimes\Delta^1
\end{equation*}
where $X\hookrightarrow Y$ is any inclusion of filtered simplicial set.
Then the saturated classes generated by $A$, $B$ and $C$ coincide.
\end{lemme}

\begin{proof}
In the non-filtered case, this is \cite[Theorem IV.2.1]{Zisman}. The same proof works in the filtered case, as only admissible horn inclusions are used at each step.
\end{proof}

\begin{lemme}\label{An2}
Let $X\hookrightarrow Y$ be a morphism in $\Lambda$ and $Z\hookrightarrow W$ be an inclusion of filtered simplicial sets.
Then the morphism 
\begin{equation*}
X\times W\cup Y\times Z\hookrightarrow Y\times W
\end{equation*}
is in $\Lambda$.
\end{lemme}

\begin{proof}
In the non-filtered case, this is \cite[Proposition IV.2.2]{Zisman}. The same proof works in the filtered case using the fact that $(X\otimes K)\otimes L\simeq X\otimes(K\times L)\simeq (X\otimes L)\otimes K$ for any simplicial sets $K$ and $L$ and any filtered simplicial set $X$.
\end{proof}

Recall the following classical definition.
\begin{defin}
Let $f\colon X\to Y$ and $g\colon W\to Z$ be two morphisms. Then $f$ (respectively $g$) is said to have the left (respectively right) lifting property with respect to $g$ (respectively $f$) if for any commutative diagram of the form 

\begin{equation*}
\begin{tikzcd}
X
\arrow{d}{f}
\arrow{r}
& W
\arrow{d}{g}
\\
Y
\arrow{r}
\arrow[dashrightarrow]{ur}{h}
&Z
\end{tikzcd}
\end{equation*}

there exists a morphism $h$ making the two triangles commute.
We write RLP and LLP for those two properties.
We say that $f$ has the RLP with respect to a class of morphisms, $A$, if it has the RLP with respect to any morphisms in the class, in this case, we will write $f\in l(A)$. Similarly, we write $r(B)$ for the class of morphisms having the RLP with respect to a class of morphisms $B$.
\end{defin}

\begin{defin}\label{DefinClassCisinski}
We define the following classes of morphisms and objects in $\sS_P$.
\begin{itemize}
\item The cofibrations are the monomorphisms.
\item The trivial fibrations are the morphisms with the RLP with respect to cofibrations.
\item The anodyne extensions are the morphisms in $\Lambda$.
\item The naive fibrations are the morphisms having the RLP with respect to anodyne extensions.
\item The fibrant objects are the objects $X$ for which the morphism $X\to N(P)$ is a naive fibration.
\item The weak equivalences are the morphisms $f\colon X\to Y$ such that for all fibrant filtered simplicial set $Z$, the induced application between the sets of homotopy classes of maps $f^*\colon [Y,Z]\to [X,Z]$ is a bijection.
\item The trivial cofibrations are the cofibrations which are also weak equivalences.
\item The fibrations are the morphisms with the RLP with respect to trivial cofibrations.
\end{itemize}
\end{defin}

\begin{theo}\label{ApplyCisinski}
With the classes of cofibrations, fibration and weak equivalences defined above, $\sS_P$ is a cofibrantly generated closed model category.
\end{theo}

\begin{proof}
This is \cite[Théorème 1.3.22]{Cisinski} applied to the category of presheaves $\sS_P$.
\end{proof}

Although we know that the model structure on $\sS_P$ is cofibrantly generated  \cite[Théorème 1.3.22]{Cisinski}, we do not know that the set of admissible horn inclusions is the set of generating trivial cofibrations. In general, we have the following inclusions between classes 
\begin{align*}
&l(r(A))=\{\text{anodyne extensions}\}\subset\{\text{trivial cofibrations}\}\\
&r(A)=\{\text{naive fibrations}\}\supset \{\text{fibrations}\}=r(\{\text{trivial cofibrations}\})
\end{align*}
and those inclusions may be strict (see \cite[Remarque 1.3.46]{Cisinski}). The goal of the next section is to show that in the case of $\sS_P$, they are in fact equalities (Theorem \ref{TheoremCharacterizationsSetP}), which will allow us to further characterize this model structure, such as proving that it is combinatorial and that its weak equivalences are detected by filtered homotopy groups (Theorem \ref{WhatWeHaveSoFar}).

\section{Characterizing the model structure on $\sS_P$}\label{SecondModelStructure}
\label{SectionCharacterizingFibration}
We wish to characterize fibrations in terms of lifting properties with respect to anodyne extensions, in order to get a description of the model category "à la Kan".

To do this, we will follow the strategy of the article \cite{SeanMoss}. In his article, the author constructs the Quillen model structure on the category of simplicial sets, starting from the data of the anodyne extensions, using a fonctorial fibrant replacement: the $\Exi$ functor. We will do the same in a filtered setting, using the (filtered) anodyne extensions of Proposition \ref{AdmissibleHornAnodyneExt}. The data of such a functor is exactly what is needed to characterize fibrations in a Cisinski model structure (see Theorem \ref{TheoAppendixA}). 
The main ideas of the following proof can be arranged as follows:

First, define a filtered subdivision functor $\sd_P\colon \sS_P\to \sS_P$. (Definition \ref{FilteredSubdivision}).
Taking the adjoint, we get the filtered extension functor $\Ex_P\colon \sS_P\to \sS_P$. There is a natural inclusion $X\hookrightarrow \Ex_P(X)$, so by taking the colimit of 
\begin{equation*}
X\hookrightarrow \Ex_P(X)\hookrightarrow\Ex_P^2(X)\hookrightarrow\dots
\end{equation*}
we get $\Exi_P(X)$. This will be our candidate for a fibrant replacement of $X$.

Then, we define a class of morphisms, the class of filtered strong anodyne extensions (FSAE) (Definition \ref{FSAE}), whose retracts will give us the class of anodyne extensions (see Proposition \ref{AdmissibleHornAnodyneExt}). This class is defined in such a way that there are combinatorial characterisations which allow to decide whether or not a given morphism is in it.

Now, using the combinatorial characterisation of FSAE, we show that for any admissible horn inclusion, its subdivision is a FSAE. This statement, in turn, gives us that $\sd_P$ preserves anodyne extensions (Proposition \ref{sdPreserves}), and so $\Ex_P$ preserves (naive) fibrations.


The main step, the fact that $\Exi_P(X)$ is fibrant for all $X$ does not follow from abstract-nonsense and the proof in the non-filtered case can not be adapted directly. This is the step that required working with the rather involved definition of filtered subdivision (Definition \ref{FilteredSubdivision}). We then carefully modify the proof of \cite[Lemma III.4.7]{GoerssJardine} so that it suits the filtered case (Lemme \ref{ExFibrant}).

Equipped with those preliminary results, one can apply Theorem \ref{TheoAppendixA} to show that the class of fibrations and the class of naive fibrations coincide (Theorem \ref{TheoremCharacterizationsSetP}).
Note that the proof of this statement does not rely on the fact that $X\to \Exi_P(X)$ is a weak equivalence. Since we will not need this fact in the rest of the paper, we refer the reader to \cite[Appendix B]{TheseMoi}, for a proof that $\Exi_P$ is indeed a fibrant replacement functor.
%

\subsection{Filtered subdivision and its adjoint}
Recall that given a simplex $\Delta^n$, its subdivision is the simplicial set $\sd(\Delta^n)$ whose simplices are given by increasing chain of non-degenerate simplices of $\Delta^n$ :
\begin{equation*}
\sd(\Delta^n)_k=\{(\sigma_0,\dots,\sigma_k)\ |\ \sigma_0\subseteq \dots\subseteq \sigma_k \in (\Delta^n)_{\nd}\}
\end{equation*}
and faces and degeneracies are given by deletion or repetition of one term of the chain.

\begin{defin}[filtered subdivision]\label{FilteredSubdivision}
Let $\Delta^J= \Delta^N\xrightarrow{\pi} N(P)$ be a filtered simplex. We will define its filtered subdivision, $\sd_P(\Delta^J)$, as a subsimplicial set of $\sd(\Delta^N)\times N(P)$. As a simplicial set $\sd_P(\Delta^J)$ is defined as follows :
\begin{equation*}
\sd_P(\Delta^J)_k=\{((\sigma_0,\dots,\sigma_k),\Delta^K)\ | \ (\sigma_0,\dots,\sigma_k)\in \sd(\Delta^N),\Delta^K=\Delta^k\xrightarrow{\pi_K} N(P), \pi_K(\Delta^k)\subseteq\pi(\sigma_0)\}. 
\end{equation*}
Where the inclusion $\pi_{K}(\Delta^k)\subseteq \pi(\sigma_0)$ means that the subsimplicial set of $N(P)$ generated by $\pi_K(\Delta^K)$ is included in the one generated by $\pi(\sigma_0)$.
In words, the simplices of $\sd_P(\Delta^J)$ are pairs $(\sigma,\Delta^K)$ where $\sigma$ is a simplex of $\sd(\Delta^N)$ and $\Delta^K$ is a simplex of $N(P)$, and such that for every color of $\Delta^K$ there is a point of this color in $\sigma_0$.
The filtration on $\sd_P(\Delta^J)$ is given by the following composition 
\begin{equation*}
 \sd_P(\Delta^J)\hookrightarrow \sd(\Delta^N)\times N(P)\xrightarrow{\pr_{N(P)}} N(P).
\end{equation*}

When manipulating subdivided simplices, it will be convenient to have a more explicit way of writing down simplices. For this reason, we give the following equivalent description of $\sd_P(\Delta^J)$ 
\begin{equation*}
\sd_P(\Delta^J)_k=\{[(\sigma_0,q_0),\dots,(\sigma_k,q_k)]\ |\ \sigma_0\subseteq\dots\subseteq\sigma_k\in (\Delta^J)_{\nd}, q_0\leq \dots\leq q_k\in P, q_i\in \pi(\sigma_0),\forall i\}. 
\end{equation*}
Writing simplices this way, the filtration map then becomes :
\begin{align*}
\sd_P(\Delta^J)&\to N(P)\\
[(\sigma_0,q_0),\dots,(\sigma_k,q_k)]&\mapsto [q_0,\dots,q_k]
\end{align*}
This also allows for a more explicit description of faces and degeneracies. Indeed, $d_i$ corresponds to deleting the term $(\sigma_i,q_i)$ and $s_i$ corresponds to repeating it.
This definition gives rise to a functor 
\begin{align*}
\sd_P\colon\sS_P&\to \sS_P\\
X=\colim_{\Delta^J\in X}\Delta^J&\mapsto \sd_P(X)=\colim_{\Delta^J\in X}\sd_P(\Delta^J)
\end{align*}
Whose effect on the map is given by passing to the colimit.
\end{defin}

\begin{exemple}\label{ExampleNaiveSubdivision}
Let $\Delta^J=[p_0,p_0,p_1]$ be a filtered simplex over $P=\{p_0<p_1\}$. It is of the form $\pi\colon \Delta^{N(J)}\to N(P)$. The "naive" way of defining its subdivision would be to subdivide both $\Delta^{N(J)}$ and $N(P)$ and consider as a filtration the following composition, 
\begin{equation*}
\begin{tikzcd}
\sd(\Delta^{N(J)})
\arrow{r}{\sd(\pi)}
&\sd(N(P))
\arrow{r}{\lv}
&N(P)
\end{tikzcd}
\end{equation*}
where $\lv$ is the classical last-vertex map. Figure \ref{FigureSubdivisions} shows the non-subdivided simplex $\Delta^J$, its naive subdivision, its filtered subdivision, and the filtered subdivision of $[p_0,p_1,p_2]$. The filtration on the simplicial sets are represented by the color of their vertices. A red vertex is sent to $p_0$, a blue one is sent to $p_1$, and a green vertex is sent to $p_2$.
\end{exemple}

\begin{figure}[h]
\begin{tikzpicture}
\draw[black](-1,-1) -- (1,-1);
\draw[black](-1,-1) -- (0,1);
\draw[black](0,1) -- (1,-1);
\filldraw[blue](0,1) circle (2pt);
\filldraw[red] (1,-1) circle (2pt);
\filldraw[red] (-1,-1) circle (2pt);

\draw[black,cm ={1,0,0,1,(4cm,0cm)}](-1,-1) -- (1,-1);
\draw[black,cm ={1,0,0,1,(4cm,0cm)}](-1,-1) -- (0,1);
\draw[black,cm ={1,0,0,1,(4cm,0cm)}](0,1) -- (1,-1);
\draw[black,cm ={1,0,0,1,(4cm,0cm)}](-0.5,0) -- (0,-0.5)--(0.5,0);
\draw[black,cm ={1,0,0,1,(4cm,0cm)}](0,-1)--(0,-0.5)--(0,1);
\draw[black,cm ={1,0,0,1,(4cm,0cm)}](1,-1)--(0,-0.5)--(-1,-1);
\filldraw[blue,cm ={1,0,0,1,(4cm,0cm)}](0,1) circle (2pt);
\filldraw[red,cm ={1,0,0,1,(4cm,0cm)}] (1,-1) circle (2pt);
\filldraw[red,cm ={1,0,0,1,(4cm,0cm)}] (-1,-1) circle (2pt);
\filldraw[blue,cm ={1,0,0,1,(4cm,0cm)}] (-0.5,0) circle (2pt);
\filldraw[blue,cm ={1,0,0,1,(4cm,0cm)}] (0.5,0) circle (2pt);
\filldraw[blue,cm ={1,0,0,1,(4cm,0cm)}] (0,-0.5) circle (2pt);
\filldraw[red,cm ={1,0,0,1,(4cm,0cm)}] (0,-1) circle (2pt);

\draw[black,cm ={1,0,0,1,(8cm,0cm)}](-1,-1) -- (1,-1);
\draw[black,cm ={1,0,0,1,(8cm,0cm)}](-1,-1) -- (0,1);
\draw[black,cm ={1,0,0,1,(8cm,0cm)}](0,1) -- (1,-1);
\draw[black,cm ={1,0,0,1,(8cm,0cm)}](0,-1)--(0,-0.5)--(0,0.5)--(0,1);
\draw[black,cm ={1,0,0,1,(8cm,0cm)}](0.3,0.4)--(0,0.3)--(-0.3,0.4);
\draw[black,cm ={1,0,0,1,(8cm,0cm)}](0.7,-0.4)--(0,-0.3)--(-0.7,-0.4);
\draw[black,cm ={1,0,0,1,(8cm,0cm)}](-1,-1)--(0,-0.3)--(1,-1);
\draw[black,cm ={1,0,0,1,(8cm,0cm)}](0.7,-0.4)--(0,0.3)--(-0.7,-0.4);
\filldraw[blue,cm ={1,0,0,1,(8cm,0cm)}](0,1) circle (2pt);
\filldraw[red,cm ={1,0,0,1,(8cm,0cm)}] (1,-1) circle (2pt);
\filldraw[red,cm ={1,0,0,1,(8cm,0cm)}] (-1,-1) circle (2pt);
\filldraw[red,cm ={1,0,0,1,(8cm,0cm)}] (0,-1) circle (2pt);
\filldraw[red,cm ={1,0,0,1,(8cm,0cm)}](0,-0.3) circle (2pt);
\filldraw[red,cm ={1,0,0,1,(8cm,0cm)}](0.7,-0.4) circle (2pt);
\filldraw[red,cm ={1,0,0,1,(8cm,0cm)}](-0.7,-0.4) circle (2pt);
\filldraw[blue,cm ={1,0,0,1,(8cm,0cm)}](0.3,0.4) circle (2pt);
\filldraw[blue,cm ={1,0,0,1,(8cm,0cm)}](-0.3,0.4) circle (2pt);
\filldraw[blue,cm ={1,0,0,1,(8cm,0cm)}](0,0.3) circle (2pt);

\draw[black,cm ={1,0,0,1,(12cm,0cm)}](-1,-1) -- (1,-1);
\draw[black,cm ={1,0,0,1,(12cm,0cm)}](-1,-1) -- (0,1);
\draw[black,cm ={1,0,0,1,(12cm,0cm)}](0,1) -- (1,-1);
\draw[black,cm ={1,0,0,1,(12cm,0cm)}](-0.66,-0.33)--(-0.33,-0.5) -- (0.33,-0.5)--(0.66,-0.33);
\draw[black,cm ={1,0,0,1,(12cm,0cm)}](-0.33,0.33)--(0,0.17) -- (0.33,0.33);
\draw[black,cm ={1,0,0,1,(12cm,0cm)}](0.33,-1)--(0.33,-0.5) -- (0,0.17);
\draw[black,cm ={1,0,0,1,(12cm,0cm)}](-0.33,-1)--(-0.33,-0.5) -- (0,0.17);
\draw[black,cm ={1,0,0,1,(12cm,0cm)}](-1,-1)--(-0.33,-0.5);
\draw[black,cm ={1,0,0,1,(12cm,0cm)}](1,-1)--(0.33,-0.5);
\draw[black,cm ={1,0,0,1,(12cm,0cm)}](0.66,-0.33)--(0,0.17);
\draw[black,cm ={1,0,0,1,(12cm,0cm)}](-0.66,-0.33)--(0,0.17);
\draw[black,cm ={1,0,0,1,(12cm,0cm)}](-0.33,-1)--(0.33,-0.5);
\draw[black,cm ={1,0,0,1,(12cm,0cm)}](0,1)--(0,0.17);
\filldraw[red,cm ={1,0,0,1,(12cm,0cm)}] (-1,-1) circle (2pt);
\filldraw[red,cm ={1,0,0,1,(12cm,0cm)}] (-0.33,-1) circle (2pt);
\filldraw[red,cm ={1,0,0,1,(12cm,0cm)}] (-0.33,-0.5) circle (2pt);
\filldraw[red,cm ={1,0,0,1,(12cm,0cm)}] (-0.66,-0.33) circle (2pt);
\filldraw[blue,cm ={1,0,0,1,(12cm,0cm)}](1,-1) circle (2pt);
\filldraw[blue,cm ={1,0,0,1,(12cm,0cm)}](0.33,-1) circle (2pt);
\filldraw[blue,cm ={1,0,0,1,(12cm,0cm)}](0.66,-0.33) circle (2pt);
\filldraw[blue,cm ={1,0,0,1,(12cm,0cm)}](0.33,-0.5) circle (2pt);
\filldraw[green,cm ={1,0,0,1,(12cm,0cm)}](0,1) circle (2pt);
\filldraw[green,cm ={1,0,0,1,(12cm,0cm)}](-0.33,0.33) circle (2pt);
\filldraw[green,cm ={1,0,0,1,(12cm,0cm)}](0.33,0.33) circle (2pt);
\filldraw[green,cm ={1,0,0,1,(12cm,0cm)}](0,0.17) circle (2pt);
\end{tikzpicture}
\caption{From left to right, the simplex $\Delta^J=[p_0,p_0,p_1]$, its "naive" subdivision, its filtered subdivision $\sd_P(\Delta^{J})$, and the filtered subdivision of $[p_0,p_1,p_2]$}
\label{FigureSubdivisions}
\end{figure}


As in the non-filtered case, we will also make use of a last-vertex map.

\begin{defin}\label{FilteredLastVertex}
Let $\Delta^J=\Delta^N\xrightarrow{\pi}N(P)$ be a filtered simplex. We define the filtered last-vertex map on $\sd_P(\Delta^J)$ on the vertices
\begin{align*}
\sd_P(\Delta^J)_0&\to \Delta^J_0\\
(\sigma,q)&\mapsto \max\{e\ |\ \text{ $e$ is a vertex of $\sigma$}, \pi(e)=q\},
\end{align*}
where the maximum is taken over the order on the vertices of the simplicial set $\sigma$.
This extends to a map of simplicial set $\lv_P(\Delta^J)\colon\sd_P(\Delta^J)\to \Delta^J$ which will be filtered by definition of the filtration on $\sd_P(\Delta^J)$. Furthermore, this defines a natural transformation $\lv_P\colon \sd_P\to \Id$, via 
\begin{equation*}
lv_P(X)=\colim_{\Delta^J\in X}\lv_P(\Delta^J)\colon \colim_{\Delta^J\in X}\sd_P(\Delta^J)=\sd_P(X)\to X=\colim_{\Delta^J\in X}\Delta^J
\end{equation*}
\end{defin}


\begin{defin}\label{DefinitionEx}
The functor $\sd_P\colon \sS_P\to \sS_P$ has a right adjoint $\Ex_P(X)$, with $\Ex_P(X)_J=\Hom(\sd_P(\Delta^J),X)$. For $X$ a filtered simplicial set, we write $\iota_X\colon X\hookrightarrow \Ex_P(X)$ for the adjoint of $l.v_P\colon \sd_P(X)\to X$.
And we define $\Exi_P(X)$ as the colimit of the diagram 
\begin{equation*}
X\xrightarrow{\iota_X} \Ex_P(X)\xrightarrow{\iota_{\Ex_P(X)}} \Ex_P^2(X)\to \dots
\end{equation*}
\end{defin}

\begin{remarque}
Definitions \ref{FilteredSubdivision} and \ref{FilteredLastVertex} might seem arbitrary complicated. Indeed, given any filtered simplicial set $X\to N(P)$, we could have defined its subdivision as $\sd(X)$ with the filtration given by the composition $\sd(X)\to X\to N(P)$ - equivalently, by $\sd(X)\to\sd(N(P))\to N(P)$) - see Example \ref{ExampleNaiveSubdivision}. However, with this definition, and the associated definition of $\Exi_P$, Theorem \ref{ExFibrant}, stating that $\Exi_P(X)$ is fibrant, would not hold. See remark \ref{WhereWeNeedFilteredSubdivision}.
\end{remarque}

\subsection{Filtered Strong Anodyne Extensions}

We directly translate the formalism of \cite{SeanMoss} into the filtered setting. We refer to the original article for the details.

\begin{defin}[Filtered Anodyne Presentation]\label{FSAE}
Let $f\colon X\to Y$ be an inclusion of filtered simplicial sets. We write $X_{\nd}$ for the non degenerate simplices of $X$. A filtered anodyne presentation for $f$ is the data of 
\begin{itemize}
\item a partition $Y_{\nd}\setminus X_{\nd}=Y_{II}\coprod Y_{I}$,
\item a bijection $\varphi\colon Y_{II}\to Y_{I}$,  
\end{itemize}
such that :
\begin{itemize}
\item  If $\sigma\in Y_{II}$, there exists a unique $k$ such that $\sigma=d_{k}(\varphi(\sigma))$, and the deleted vertex in $\varphi(\sigma)$ shares its color with at least one vertex of $\sigma$.
\item The ancestral preorder is well founded.
\end{itemize}
Here the ancestral preorder is the transitive relation on $Y_{II}\coprod Y_{I}$ generated by $\sigma < \tau$ if either $\sigma$ is a proper face of $\tau$, or if $\tau\in Y_{II}$ and $\sigma\not=\tau$ is a proper face of $\varphi(\tau)$. 
\end{defin}

\begin{defin}
A morphism between filtered simplicial sets is a Filtered Strong Anodyne Extension (FSAE) if it admits a filtered anodyne presentation. 
\end{defin}

\begin{prop}
Any anodyne extension is the retract of a FSAE. 
\end{prop}

\begin{proof}
See Remark 2.2 in \cite{SeanMoss}
\end{proof}

\subsection{Properties of the filtered subdivision}

\begin{prop}\label{sdPreserves}
Let $f\colon X\to Y$ be an anodyne extension. Then $\sd_P(f)$ is an anodyne extension.
\end{prop}

\begin{proof}
Since $\sd_P$ preserves retracts and compositions, it is enough to show that $\sd_P(\Lambda^J_{k}\hookrightarrow \Delta^J)$ is an anodyne extension for any admissible horn inclusion $\Lambda^J_{k}\hookrightarrow \Delta^J$. 
We define an anodyne presentation in the following way (see the proof of \cite[Proposition 2.14]{SeanMoss} for comparison with the non-filtered case). We first define the following sets of non-degenerate simplices in $\sd_P(\Delta^J)_{\nd}$.

\begin{enumerate}[label=(\alph*)]
\item \label{a} $[(\xi_1,\alpha_1),\dots,(\xi_q,\alpha_q),(d_{k}(\Delta^J),\beta_0),\dots,(d_{k}(\Delta^J),\beta_r),(\Delta^J,\gamma_1),\dots,(\Delta^J,\gamma_s)]$, \\ where $\beta_r\not=\gamma_1$, or $s=0$.
\item \label{b} $[(\xi_1,\alpha_1),\dots,(\xi_q,\alpha_q),(d_{k}(\Delta^J),\beta_0),\dots,(d_{k}(\Delta^J),\beta_r),(\Delta^J,\beta_r),(\Delta^J,\gamma_1),\dots,(\Delta^J,\gamma_s)]$.
\item \label{c} $[(\mu_0,\epsilon_0),\dots,(\mu_t,\epsilon_0),(\mu_t,\epsilon_1),\dots,(\mu_t,\epsilon_u),(\sigma_1,\nu_1),\dots,(\sigma_v,\nu_v),(\Delta^J,\gamma_0),\dots,(\Delta^J,\gamma_s)]$, \\ where $\epsilon_u\not =\nu_1$ or ($v=0$ and $\epsilon_u\not =\gamma_0$).
\item \label{d} $[(\mu_0,\epsilon_0),\dots,(\mu_t,\epsilon_0),(\mu_t,\epsilon_1),\dots,(\mu_t,\epsilon_u),(\mu_t,\nu_1),(\sigma_1,\nu_1),\dots \\
\phantom{X} \hspace{200pt} \dots,(\sigma_v,\nu_v),(\Delta^J,\gamma_0),\dots ,(\Delta^J,\gamma_s)]$,\\ where $\nu_1=\gamma_0$ if $v=0$.
\item \label{e} $[(\eta_0,\gamma_0),\dots,(\eta_w,\gamma_0),(\zeta_1,\gamma_0),\dots,(\zeta_x,\gamma_0),(\Delta^J,\gamma_0),\dots,(\Delta^J,\gamma_s)]$,
 where $\zeta_1\not = \eta_w\cup \{e_k\}$.
\item \label{f} $[(\eta_0,\gamma_0),\dots,(\eta_w,\gamma_0),(\eta_w\cup\{e_k\},\gamma_0),(\zeta_1,\gamma_0),\dots,(\zeta_x,\gamma_0),(\Delta^J,\gamma_0),\dots,(\Delta^J,\gamma_s)]$.
\item \label{g} $[(\theta_1,\gamma_0),\dots,(\theta_y,\gamma_0),(\kappa_1,\gamma_0),\dots,(\kappa_z,\gamma_0),(\Delta^J,\gamma_0),\dots,(\Delta^J,\gamma_s)]$, where  $\kappa_1\not = \theta_y\cup \{e_{k'}\}$.
\item \label{h} $[(\theta_1,\gamma_0),\dots,(\theta_y,\gamma_0),(\kappa_1\setminus\{e_{k'}\},\gamma_0),(\kappa_1,\gamma_0),\dots,(\kappa_z,\gamma_0),(\Delta^J,\gamma_0),\dots,(\Delta^J,\gamma_s)]$.
\end{enumerate}
where :
\begin{itemize}
\item $e_k$ is the $k$-th vertex of $\Delta^J$.
\item $e_{k'}$ is a fixed vertex of $\Delta^J$ sharing its color with $e_k$, and different from $e_k$. (At least one exists since $\Lambda^J_k$ is admissible.)
\item $q,r,s,t,u,v,w,x,y,z\geq 0$
\item $\mu_i,\sigma_i,\eta_i,\zeta_i,\theta_i,\kappa_i$ are faces of $\Delta^J$ not equal to $\Delta^J$ or $d_{k}(\Delta^J)$.
\item $(\alpha_i)$ and $(\nu_i)$ are (possibly empty) finite sequence of strictly increasing elements of $P$, where $i$ ranges from $1$ to $q$ and $v$ respectively.
\item $(\beta_i),(\epsilon_i)$ and $(\gamma_i)$ are non-empty finite sequences of strictly increasing elements of $P$, where $i$ ranges from $0$ to $r,u$ and $s$ respectively.
\item $e_k\not \in \eta_i$
\item $e_k\in \zeta_i$
\item $e_{k}\in \theta_i$ but $e_{k'}\not\in \theta_i$
\item $e_k,e_{k'}\in \kappa_i$
\end{itemize}
Then, we take the simplices of \ref{b}, \ref{d}, \ref{f} and \ref{h} to be the simplices of type I and those of \ref{a}, \ref{c}, \ref{e} and \ref{g} to be the simplices of type II, and $\varphi$ to be the morphism suggested by the notations. For example, starting from a simplex of \ref{a} :
\begin{equation*}
[(\xi_1,\alpha_1),\dots,(\xi_q,\alpha_q),(d_{k}(\Delta^J),\beta_0),\dots,(d_{k}(\Delta^J),\beta_r),(\Delta^J,\gamma_1),\dots,(\Delta^J,\gamma_s)]
\end{equation*}
the application $\varphi$ sends it to the simplex of \ref{b}
\begin{equation*}
[(\xi_1,\alpha_1),\dots,(\xi_q,\alpha_q),(d_{k}(\Delta^J),\beta_0),\dots,(d_{k}(\Delta^J),\beta_r),(\Delta^J,\beta_r),(\Delta^J,\gamma_1),\dots,(\Delta^J,\gamma_s)]
\end{equation*}
obtained by adding the vertex $(\Delta^J,\beta_r)$.
 It is easy to see from the definition of $\varphi$ that the vertex added by $\varphi$ shares its color with one already present in the corresponding type II simplex. One shows that every simplex of $\sd_P(\Delta^J)_{\nd}\setminus\sd_P(\Lambda^J_{k})_{\nd}$ appears exactly once in this list and that $\varphi$ is a bijection by a careful enumeration of the non-degenerate simplices of $\sd_P(\Delta^J)$. To see that the ancestral order is well founded, one has to check that there are no infinite strictly decreasing sequence of simplices. But notice that if $\sigma$ is a type I simplex, then any of its direct ancestor has to be of strictly lower dimension (since it has to be a proper face of $\sigma$). So it is enough to show that there are no infinite strictly decreasing sequence of type II simplices of any given dimension. 

Let us examine what happens when we start with a simplex of \ref{a}. 

\begin{equation}\label{SubdivisionSimplexExempleA}
[(\xi_1,\alpha_1),\dots,(\xi_q,\alpha_q),(d_{k}(\Delta^J),\beta_0),\dots,(d_{k}(\Delta^J),\beta_r),(\Delta^J,\gamma_1),\dots,(\Delta^J,\gamma_s)]
\end{equation}
Applying $\varphi$, we get the corresponding simplex of \ref{b}
\begin{equation*}
[(\xi_1,\alpha_1),\dots,(\xi_q,\alpha_q),(d_{k}(\Delta^J),\beta_0),\dots,(d_{k}(\Delta^J),\beta_r),(\Delta^J,\beta_r),(\Delta^J,\gamma_1),\dots,(\Delta^J,\gamma_s)]
\end{equation*}
Now the ancestors of \eqref{SubdivisionSimplexExempleA} of the same dimension are of one of the following forms.
\begin{align*}
&[(\xi_1,\alpha_1),\dots,\widehat{(\xi_i,\alpha_i)},\dots,(\xi_q,\alpha_q),(d_{k}(\Delta^J),\beta_0),\dots\\
&\hspace{180pt} \dots,(d_{k}(\Delta^J),\beta_r),(\Delta^J,\beta_r),(\Delta^J,\gamma_1),\dots,(\Delta^J,\gamma_s)]\\
&\hspace{360pt}1\leq i\leq q\\
&[(\xi_1,\alpha_1),\dots,(\xi_q,\alpha_q),(d_{k}(\Delta^J),\beta_0),\dots,\widehat{(d_{k}(\Delta^J),\beta_i)},\dots\\
&\hspace{180pt}\dots,(d_{k}(\Delta^J),\beta_r),(\Delta^J,\beta_r),(\Delta^J,\gamma_1),\dots,(\Delta^J,\gamma_s)]\\
&\hspace{360pt}0\leq i \leq r\\
&[(\xi_1,\alpha_1),\dots,(\xi_q,\alpha_q),(d_{k}(\Delta^J),\beta_0),\dots,(d_{k}(\Delta^J),\beta_r),(\Delta^J,\beta_r),(\Delta^J,\gamma_1),\dots \\
&\hspace{286pt}\dots,\widehat{(\Delta^J,\gamma_i)},\dots,(\Delta^J,\gamma_s)]\\
&\hspace{360pt}1\leq i\leq s
\end{align*}
In the first and last case, we get a simplex of \ref{b}.
In the second case the following things can happen
\begin{itemize}
\item $i<r$, then we get a simplex of \ref{b}
\item $i=r>0$, we get a simplex of \ref{a}, but notice that $r$ has decreased by $1$.
\item $i=r=0$, then we get a simplex in which $d_{k}(\Delta^J)$ does not appear. It can be in any of the sets \ref{c}, \ref{d}, \ref{e}, \ref{f}, \ref{g} or \ref{h}
\end{itemize}
In the end, what we get is that, among the direct ancestors of \ref{a} of the same dimension, those who are of type II are either in \ref{c}, \ref{e} or \ref{g} or are still in \ref{a} but the quantity $r$ has decreased. Now, doing the same discussion of cases for \ref{c}, \ref{e} and \ref{g} and compiling this information in the form of a directed graph, we get Figure \ref{FigureDirectAncestors}. One recognizes the top part of the graph as the discussion of cases on direct ancestors of \eqref{SubdivisionSimplexExempleA}.
%

\begin{figure}[h]
\begin{tikzpicture}[->,>=stealth',shorten >=0pt,auto,node distance=5cm,semithick]
\node[state](A){\ref{a}};
\node[state](C)[below of =A]{\ref{c}};
\node[state](E)[below left of=C]{\ref{e}};
\node[state](G)[below right of =C]{\ref{g}};

\path (A) edge [loop right] node{$r\searrow$} (A)
	edge node{} (C)
	edge [bend right] node{} (E)
	edge [bend left] node{} (G)

(C) edge [loop right]  node{$u\nearrow,t,\epsilon_0$} (C)
	edge [loop left] node{$t\searrow,\epsilon_0$} (C)
	edge [loop below] node{$\epsilon_0\nearrow$} (C)
	edge [bend right=15] node{} (G)
	edge [bend left=15] node{} (E)
(E) edge [loop left] node{$w\searrow$} (E)
	edge node{} (G)
	edge [bend left=15,red] node{$\gamma_0\nearrow$} (C)
(G) edge [loop right] node{$y\nearrow$} (G)
	edge [bend right=15,red] node[above right]{$\gamma_0\nearrow$} (C); 	
\end{tikzpicture}
\caption{The directed graph of direct ancestors of type II simplices. Following an arrow correspond to passing to an ancestor. The labels indicate which quantities describing the simplices increase ($\nearrow$), decrease ($\searrow$) or are kept constant along the path.}
\label{FigureDirectAncestors}
\end{figure}

Given a type II simplex in one of the four sets \ref{a}, \ref{c}, \ref{e}, \ref{g} we get a direct ancestor by following some arrow on the graph. The labels indicate how quantity varies along each arrow. A quantity next to the symbol $\nearrow$ (resp $\searrow$) means that it is strictly increasing (resp strictly decreasing) when passing to the ancestor. A quantity without any symbol next to it means that it is kept constant when passing to the ancestor. The quantity $\gamma_0$ is kept constant along all arrows, except along the red ones. Notice that if we take out the red arrows and the loops, \ref{g} becomes a sink.
Now since every time we follow the red arrows $\gamma_0$ increases and $\gamma_0$ is bounded, we can safely delete those arrows. But now, since all quantites increasing (or decreasing) when going through a given loop are bounded, every loop can only occur a finite number of times. Taking a closer look at the simplices of \ref{c}, one sees that after a finite number of steps, there are no path left on the graph. This exactly means that given any type II simplex, it only has a finite number of ancestors of the same dimension.
\end{proof}

\subsection{$\Exi_P(X)$ is fibrant}

In this section, we prove the following Lemma :

\begin{lemme}\label{ExFibrant}
Let $X$ be a filtered simplicial set, $\Exi_P(X)$ is fibrant.
\end{lemme}

\begin{remarque}\label{WhereWeNeedFilteredSubdivision}
Here, we want to adapt the proof of \cite[Lemma III.4.7]{GoerssJardine}. Note that this is the part that required us to work with $\sd_P$. Indeed, if one tries to go through this proof using the naive subdivision (see Example \ref{ExampleNaiveSubdivision}), one sees that the dotted arrows in diagram \eqref{GoerssJardineDiagram} may not exist. 
\end{remarque}
For the proof of Lemma \ref{ExFibrant} we will make use of the filtered last-vertex map (Definition \ref{FilteredLastVertex}) and of two other last-vertex maps. This is the content of Definition \ref{LastVertex}.

\begin{defin}\label{LastVertex}
Let $(\sigma,r)$ be a vertex of a subdivision with $\sigma=[e_0,\dots,e_n]$.
Then, the last vertex of $\sigma$ is defined to be $\lv(\sigma)=e_n$.
We can extend this to a \textbf{non-filtered} application 
\begin{eqnarray*}
\lv\colon&\sd_P(X)&\to X\\
&[(\sigma_0,r_0),\dots,(\sigma_m,r_m)]&\mapsto [\lv(\sigma_0),\dots,\lv(\sigma_m)]. 
\end{eqnarray*}
Given $r\in P$, some color appearing in $\sigma$, we also define the last vertex of color $r$ in $\sigma$ to be 
\begin{equation*}
\lv_{r}(\sigma)=\max\{e\ |\ \text{ $e$ is a vertex of $\sigma$}, \pi(e)=r\}.
\end{equation*}
where the maximum is taken over the order on the vertices of the simplicial set $\sigma$.
\end{defin}

\begin{remarque}
Notice that we can express the \textbf{filtered} application $\lv_P$ of Definition \ref{FilteredLastVertex} as  
\begin{equation*}
\lv_P([(\sigma_0,r_0),\dots,(\sigma_m,r_m)])=[\lv_{r_0}(\sigma_0),\dots,\lv_{r_m}(\sigma_m)].
\end{equation*}
\end{remarque}

\begin{proof}[Proof of Lemma \ref{ExFibrant}]
We have to show that given any diagram of the form 

\begin{equation*}
\begin{tikzcd}
\Lambda^J_{k}
\arrow{r}
\arrow{d}
&\Exi_P(X)
\\
\Delta^J
\arrow[dashrightarrow]{ur}
&\phantom{X}
\end{tikzcd}
\end{equation*}

where $\Lambda^J_k$ is admissible, there exists a dotted arrow making the diagram commute. Since $\Lambda^J_{k}$ is compact, one can assume without loss of generality that its image lies in $\Ex_P(X)$. It is then enough to show that a dotted arrow exists for any diagram of the form

\begin{equation}\label{GoerssJardineDiagram}
\begin{tikzcd}
\Lambda^J_{k}
\arrow{r}{\lambda}
\arrow{d}
&\Ex_P(X)
\arrow{d}
\\
\Delta^J
\arrow[dashrightarrow]{r}
&\Ex_P^3(X)
\end{tikzcd}
\end{equation}
By adjunction, diagram (\ref{GoerssJardineDiagram}) is equivalent to the following diagram 
\begin{equation*}
\begin{tikzcd}
\sd_P^2(\Lambda^J_{k})
\arrow{r}{\lv_P^2}
\arrow{d}
&\Lambda^J_{k}
\arrow{d}{\lambda}
\\
\sd_P^2(\Delta^J)
\arrow[dashrightarrow]{r}
& \Ex_P(X)
\end{tikzcd}
\end{equation*}

It is now sufficient to find some dotted arrow making the following diagram commutative :

\begin{equation*}
\begin{tikzcd}
\sd_P^2(\Lambda^J_{k})
\arrow{r}{\lv_P^2}
\arrow{d}
&\Lambda^J_{k}
\arrow{d}
\arrow{dr}{\lambda}
&\phantom{X}
\\
\sd_P^2(\Delta^J)
\arrow[swap,dashrightarrow]{r}{h}
&\Ex_P(\sd_P(\Lambda^J_{k}))
\arrow[swap]{r}{\Ex_P(\lambda)}
&\Ex_P(X)
\end{tikzcd}
\end{equation*}

We define it in the following way. Let $\sigma= [(\sigma_0,q_0),\dots,(\sigma_n,q_n)]$ be a $J'$-simplex of the double filtered subdivision of $\Delta^J$. Keep in mind that each of the $\sigma_i=[(\tau^i_0,r^i_0),\dots,(\tau^i_{m_i},r^i_{m_i})]$ are themselves simplices of the filtered subdivision of $\Delta^J$. We have to construct $h(\sigma)\colon\sd_P(\Delta^{J'})\to \sd_P(\Lambda^J_{k})$.
First, we define an application $f_{\sigma}$ between the vertices of $\Delta^{J'}$ and the vertices of $\Delta^J$. Since we want to be able to refer to each of the vertices of $\Delta^J$ and $\Delta^{J'}$ independantly, we will write those simplices as the ordered set of their vertices : $\Delta^{J'}=\{d_0,\dots,d_n\}$ and $\Delta^J=\{e_0,\dots,e_{N(J)}\}$. Note that some of those vertices might share the same color, in particular, notice that $d_i$ is of color $q_i$, by construction. Let $p$ be the color of $e_k$.

\begin{align*}
f_\sigma\colon \Delta^{J'}_0&\to\Delta^J_0\\
(d_l)&\mapsto\left\{\begin{array}{cl}
\lv_{q_l}(\lv_{q_l}(\sigma_l))& \text{ if $p>q_l$, or $\sigma_l$ doesn't contain any vertex of color $p$,}\\
$\phantom{X}$ & \phantom{\text{ if $p>q_l$,}} \text{ or $\lv(\sigma_l)\not = d_{k}(\Delta^J), \Delta^J$.}
\\
\lv_{q_l}(\lv_p(\sigma_l)) &\text{ if $p< q_l$, and $\sigma_l$ contains a vertex of color $p$}\\
\phantom{X}  & \phantom{\text{ if $p< q_l$,}} \text{ and $\lv(\sigma_l)=d_{k}(\Delta^J)$ or $\Delta^J$.}
\\
e_k &\text{ if $p=q_l$, and $\lv(\sigma_l)=d_{k}(\Delta^J)$ or $\Delta^J$.}  
\end{array}\right.
\end{align*}
It is clear that only one of the three conditions is verified for any given $l$, so $f_{\sigma}$ is well defined. Since the application $\lv_{q_l}$ always returns a vertex of color $q_l$, it is clear that $f_{\sigma}$ preserves the color of vertices. 
Define $h(\sigma)\colon \sd_P(\Delta^{J'})\to\sd_P(\Delta^J)$ as follows. For $(\mu,s)$ a vertex of $\sd_P(\Delta^{J'})$, let $h(\sigma)(\mu,s)=(\nu,s)$, where $\nu$ is the face of $\Delta^J$ containing all the vertices $f_{\sigma}(d)$, $d\in \mu$. Then, define $h(\sigma)$ :
\begin{align*}
h(\sigma)\colon \sd_P(\Delta^{J'})&\to\sd_P(\Delta^J)\\
[(\mu_0,s_0),\dots,(\mu_m,s_m)]&\mapsto [h(\sigma)(\mu_0,s_0),\dots,h(\sigma)(\mu_m,s_m)]
\end{align*}
One checks that $h(d_i(\sigma))=h(\sigma\circ d^i)=h(\sigma)\circ \sd_P(d^i)=d_i(h(\sigma))$, and that the same is true for degeneracies. 

We need to prove that $h(\sigma)\in \Ex_P(\sd_P(\Lambda^J_{k}))$. Or in other words, that we have an inclusion $h(\sigma)(\sd_P(\Delta^{J'}))\subseteq \sd_P(\Lambda^J_{k})$. Suppose that $h(\sigma)(\sd_P(\Delta^{J'}))\not\subset (\sd_P(\Lambda^J_{k}))$. This means in particular that every vertex of $d_{k}(\Delta^J)$ must be in the image of $f_{\sigma}$. Recall that $\sigma=[(\sigma_0,q_0),\dots,(\sigma_n,q_n)]$, with $\sigma_l= [(\tau^l_0,r^l_0),\dots,(\tau^l_{m_l},r^l_{m_l})]$. So in particular, for every vertex of $d_{k}(\Delta^J)$ there must be at least one $\tau^l_i$ containing it. But then, $\tau^{n}_{m_n}$ must contain $d_{k}(\Delta^J)$ since it contains all the $\tau^l_i$ as faces.
Let $i'=\max\{i\ |\ d_{k}(\Delta^J)\not \subseteq \tau^{n}_i\}$. Then $\tau^n_{i'}\subseteq d_{k'}(\Delta^J)$ for some $k'\not = k$. Now we know that all $\tau^l_i$ are either containing $d_{k}(\Delta^J)$ or contained in $d_{k'}(\Delta^J)$. But since $e_{k'}$ is in the image by our hypothesis, there must be an $l'$ such that $f_{\sigma}(d_{l'})=e_{k'}$. We now know that $\lv(\sigma_{l})$ contains $d_{k}(\Delta^J)$ for all $l\geq l'$. Let $p'$ be the color of $e_{k'}$. We have :
\begin{itemize}
\item if $p'<p$, since $\Lambda^J_k$ is admissible, $e_k$ must share its color with either $e_{k+1}$ or $e_{k-1}$. By symmetry, suppose that $e_k$ and $e_{k+1}$ have the same color. There must be an $l$ such that $f_{\sigma}(d_l)=e_{k+1}$. Then necessarily $l>l'$, but then $\lv(\sigma_l)$ contains $d_{k}(\Delta^J)$ and so $f_{\sigma}(d_l)=e_k$ so there is a contradiction.
\item if $p'=p$. Then $f_\sigma(d_{l'})=e_k\not = e_{k'}$, which is a contradiction.
\item if $p'>p$. we know that $\lv(\sigma_{l'})$ must contain $d_{k}(\Delta^J)$. 
We have $e_{k'}=\lv_{p'}(\lv_p(\sigma_{l'}))$. In particular, $\lv_p(\sigma_{l'})$ contains $d_{k}(\Delta^J)$.
But since $p'>p$, we know that for any $l$, we have $d_k(\Delta^J)\subseteq \lv_p(\sigma_{l'})\subseteq \lv_{p'}(\sigma_l)$.
Now, $\Delta^J$ must contain at least two points of color $p'$, because each $\tau^{n}_{i}$ must contain at least a vertex of color $p$ and $\tau^{n}_{i'}$ does not contain $e_{k'}$.
But then, the other point of color $p'$ can not be in the image of $f_{\sigma}$ because $\lv_{p'}(\sigma_l)$ contains $d_{k}(\Delta^J)$ for all $l$ and $\lv_{p'}(d_{k}(\Delta^J))=e_{k'}$.
\end{itemize}


Now to see that the square is commutative, it is enough to notice that when $\sigma\in \sd_P^2(\Lambda^J_{k})$, we have $\lv(\sigma_l)\not= d_{k}(\Delta^J),\Delta^J$ for all $l$. In this case $f_{\sigma}(d_l)=\lv_P(\lv_P(\sigma_l,q_l))$, which makes it easy to compare the two compositions.
\end{proof}

\subsection{A full description of the combinatorial model category $\sS_P$}

\begin{theo}\label{TheoremCharacterizationsSetP}
The category of filtered simplicial sets, together with the following class of maps, is a combinatorial model category.
\begin{itemize}
\item The \textbf{cofibrations} are the monomorphisms,
\item the \textbf{trivial fibrations} are the maps with the right lifting property against cofibrations,
\item the \textbf{fibrations} are the maps with the right lifting property against all admissible horn inclusions,
\item the class of \textbf{trivial cofibrations} is the saturated class generated by the admissible horn inclusions,
\item A map $f\colon X\to Y$ is a \textbf{weak equivalence} if and only if $\Exi_P(f)\colon\Exi_P(X)\to\Exi_P(Y)$ is a filtered homotopy equivalence.
\end{itemize}
Furthermore, those class of morphisms coincide with those of Definition \ref{DefinClassCisinski} and Theorem \ref{ApplyCisinski}
\end{theo}

\begin{proof}
It suffices to prove that the class coincide with those of Theorem \ref{ApplyCisinski}, since we now the latter define a combinatorial model structure. The characterization of the class of weak equivalences follows from the fact that $X\to \Exi_P(X)$ is a weak equivalence for all $X$, a proof of which can be found in \cite[Appendix B]{TheseMoi}. To prove the characterization for the first four classes of maps, we will apply Theorem \ref{TheoAppendixA} to the presheaf category $\sS_P$ together with $\sd_P$ and $\Ex_P$. We have already shown the following :
\begin{itemize}
\item $\sd_P$ preserves monomorphisms (by construction) and anodyne extensions  (Lemma \ref{sdPreserves})
\item For any filtered simplicial set $X$, $\Exi_P(X)$ is fibrant (Lemma \ref{ExFibrant})
\end{itemize}
It remains to be proven that for any filtered simplex $\Delta^J$, the map $\lv_P\colon \sd_P(\Delta^J)\to \Delta^J$ is an absolute weak equivalence (see \cite[Definition 1.3.55]{Cisinski}) and surjective. The latter follows from the definition of $\lv_P$. Let $\Delta^{J_0}=[p_0,\dots,p_d]$ be the non-degenerated filtered simplex associated to $\Delta^J$. Consider the following map 
\begin{align*}
\Delta^{J_0}&\to \sd_P(\Delta^{J})\\
[p_0,\dots,p_n]&\mapsto [(\Delta^J,p_0),\dots,(\Delta^J,p_n)]
\end{align*}
It is a section of the composition $\sd_P(\Delta^J)\to \Delta^J\to \Delta^{J_0}$. The following map gives a homotopy between the composition $\sd_P(\Delta^J)\to\Delta^{J_0}\to\sd_P(\Delta^J)$ and the identity of $\sd_P(\Delta^J)$
\begin{align*}
H\colon \Delta^1\otimes \sd_P(\Delta^{J})&\to \sd_P(\Delta^J)\\
[(0,(\sigma_0,q_0)),\dots,(0,(\sigma_k,q_k)),(1,(\sigma_{k+1},q_{k+1})),\dots]&\mapsto [(\sigma_0,q_0),\dots,(\sigma_k,q_k),(\Delta^J,q_{k+1}),\dots]
\end{align*}
In particular, the map $\Delta^{J_0}\to \sd_P(\Delta^J)$ is a strong deformation retract (see \cite[Définition 1.3.25]{Cisinski}), so it is an anodyne extension \cite[Lemme 1.3.38]{Cisinski}. By \cite[Proposition 1.3.56]{Cisinski}, it is an absolute weak equivalence. Now consider the following section to the map $\Delta^J\to \Delta^{J_0}$, where we write $\Delta^J=[e^{p_0}_0,\dots,e^{p_0}_{j_{p_0}},e^{p_1}_0,\dots,e^{p_d}_{j_{p_d}}]$
\begin{align*}
f\colon \Delta^{J_0}&\to\Delta^{J}\\
[p_0,\dots,p_d]&\mapsto [e^{p_0}_0,\dots,e^{p_d}_0]
\end{align*}
Just as earlier, this section is a strong deformation retract, and so it is an anodyne extension and an absolute weak equivalence.  Then, consider the following commutative diagram
\begin{equation*}
\begin{tikzcd}
\sd_P(\Delta^{J_0})
\arrow{r}{\sd_P(f)}
\arrow[swap]{d}{\lv_P(\Delta^{J_0})}
&\sd_P(\Delta^J)
\arrow{d}{\lv_P(\Delta^J)}
\\
\Delta^{J_0}
\arrow{r}{f}
&\Delta^{J}
\end{tikzcd}
\end{equation*}
Since absolute weak equivalences are stable under composition \cite[Proposition 1.3.57]{Cisinski}, $f\circ\lv_P(\Delta^{J_0})=\lv_P(\Delta^J)\circ \sd_P(f)$ is an absolute weak equivalence. Since $\sd_P$ preserve anodyne extensions, $\sd_P(f)$ is also an anodyne extension, and an absolute weak equivalence. but then, by \cite[Proposition 1.3.57]{Cisinski}, $\lv_P(\Delta^J)$ is also an absolute weak equivalence.
\end{proof}

\section{Simplicial structure}\label{SimplicialStructure}

In this section we define a simplicial structure on the category of filtered simplicial sets and show that the model category we defined in previous sections is actually a simplicial model category. We then use a Quillen adjunction with a category of diagrams of simplicial sets to define the filtered homotopy groups of a fibrant filtered simplicial set. Applying those constructions we prove Theorem \ref{WhiteheadSimplicial} which characterizes weak equivalences between fibrant objects.

Recall definition \ref{TensorProduct} of the filtered simplicial set $X\otimes K$. 

\begin{defin}
Let $K$ be a simplicial set and $Y$ be a filtered simplicial set. We define $Y^K$ to be the following filtered simplicial set
\begin{align*}
Y^K\colon \Delta(P)^{\op}&\to \Set\\
\Delta^J&\mapsto \Hom_{\sS_P}(\Delta^J\otimes K,Y)
\end{align*}
\end{defin}

Proposition \ref{PowerK} then follows from the definition.
\begin{prop}\label{PowerK}
Let $K$ be a simplicial set. Then the functor $-\otimes K\colon \sS_P\to \sS_P$ is left adjoint to the functor $(-)^K\colon \sS_P\to \sS_P$.
\end{prop}

\begin{prop}\label{SimplicialCategory}
There is a simplicial category structure on $\sS_P$ where the simplicial set of maps between $X$ and $Y$ is given by 
\begin{equation*}
\Map(X,Y)_n= \Hom_{\sS_P}(X\otimes \Delta^n,Y)
\end{equation*}
\end{prop}

\begin{proof}
By \cite[Lemma II.2.3]{GoerssJardine}, we only have to check the following three conditions 
\begin{itemize}
\item For $K$ any simplicial set, $-\otimes K$ is left adjoint to $(-)^K$. This is Proposition \ref{PowerK}.
\item For any filtered simplicial set $X$ the functor $X\otimes -$ commutes with colimits, and satisfy $X\otimes \Delta^0\simeq X$. The former is true by compatibility of the functor $\Pr$ with inner product and colimits. The latter is true because $\Pr(\Delta^0)\simeq N(P)$ and the projection $X\times N(P)\to X$ is always an isomorphism. Note that here and everywhere else in this article, the product of two filtered simplicial sets, denoted $"\times"$, means the inner product in the category $\sS_P$, which is a fibered product of simplicial sets over $N(P)$. (See Definition \ref{InnerProductFilteredSimplicialSets}).
\item For any simplicial sets $K,L$ and filtered simplicial set $X$ there is an isomorphism between $(X\otimes K)\otimes L$ and $X\otimes(K\times L)$ which is natural in $K,L$ and $X$.
We have the following sequence of natural isomorphism
\begin{align*}
(X\otimes K)\otimes L &= (X\times \Pr(K))\times \Pr(L)\\
&\simeq X\times (\Pr(K)\times \Pr(L))\\
&\simeq X\times \Pr(K\times L)\\
&= X\otimes(K\times L) 
\end{align*}
\end{itemize}
\end{proof}

\begin{theo}\label{SimplicialModelCategory}
With the functor $\Map$, $\sS_P$ is a simplicial model category.
\end{theo}

\begin{proof}
By \cite[Corollary II.3.12]{GoerssJardine}, we only have to check that the following is true. 
\begin{itemize}
\item Let $f\colon X\to Y$ be a cofibration between filtered simplicial sets and $n\geq 0$, then the map 
\begin{equation*}
X\otimes (\Delta^n)\cup_{X\otimes \partial(\Delta^n)} Y\otimes \partial{\Delta^n}\to Y\otimes \Delta^n
\end{equation*}
is a cofibration that is trivial if $f$ is.
\item Let $f\colon X\to Y$ be a cofibration between filtered simplicial set, then the maps
\begin{equation*}
X\otimes \Delta^1\cup_{X\otimes \{\epsilon\}} Y\otimes \{\epsilon\}\to Y\otimes \Delta^1
\end{equation*}
are trivial cofibration, where $\epsilon$ is one of the two vertices of $\Delta^1$.
\end{itemize}
The first statement, in the case where $f$ is not trivial comes from the fact that cofibrations are monomorphisms and that all operations involved preserve monomorphisms. In the case that $f$ is trivial, this is a particular case of Lemma \ref{An2} applied to $Z\hookrightarrow W= \Pr(\partial(\Delta^n))\hookrightarrow \Pr(\Delta^n)$.
The second statement is the axiom (An1) of a class of anodyne extensions. See Definition \ref{AnodyneExtension}.
\end{proof}

\begin{defin}\label{FilteredDiagram}
We define the filtered diagram functor 
\begin{align*}
D\colon \sS_P&\to \Fun(\Delta(P)^{op},\sS)\\
X&\mapsto (\Delta^J\mapsto \Map(\Delta^J,X))
\end{align*}
\end{defin}

Now by \cite[Sections 11.6 and 11.7]{Hirschhorn} the category of diagrams $\Fun(\Delta(P)^{\op},\sS)$ admits a model structure, often called the projective model structure, in which fibrations and weak equivalences are determined levelwise. Fixing this model structure on $\Fun(\Delta(P)^{\op},\sS)$, we can state the following result.
 
\begin{prop}\label{DiagramFibration}
Let $f\colon X\to Y$ be a fibration between filtered simplicial sets. Then $D(f)$ is a fibration. In addition, $D(f)$ is trivial if and only if $f$ is trivial.
\end{prop}

To prove Proposition \ref{DiagramFibration}, we will use the following lemma.

\begin{lemme}\label{LemmeDiagramme}
Let $f\colon X\to Y$ be a fibration. It is a trivial fibration if and only if, for all filtered simplices $\Delta^J$, the applications induced by $f$, $\Map(\partial(\Delta^J),X)\to \Map(\partial(\Delta^J),Y)$ and $\Map(\Delta^J,X)\to \Map(\Delta^J,Y)$ are weak equivalences in the Kan-Quillen model structure.
\end{lemme}

\begin{proof}
The direct implication comes directly from the fact that $\sS_P$ is a simplicial model category.
Let us prove the reverse implication. Let $\Delta^J$ be a filtered simplex. Consider the following diagram 
\begin{equation*}
\begin{tikzcd}
\Map(\Delta^J,X)
\arrow{dr}{(4)}
\arrow[bend left = 6]{drr}
\arrow[bend right = 20,swap]{ddr}{(2)}
&\phantom{X}
&\phantom{X}
\\
\phantom{X}
&\Map(\Delta^J,Y)\times_{\Map(\partial(\Delta^J),Y)} \Map(\partial(\Delta^J),X)
\arrow{r}
\arrow{d}{(3)}
& \Map(\partial(\Delta^J),X)
\arrow{d}{(1)}
\\
\phantom{X}
&\Map(\Delta^J,Y)
\arrow{r}
&\Map(\partial(\Delta^J),Y)
\end{tikzcd}
\end{equation*}
The simplicial model structure gives us that the morphisms labeled (1) and (2) are fibrations. They are in addition weak equivalences by hypothesis, so they are trivial fibrations. But then (3) is a trivial fibration because it is the pullback of a trivial fibration, and then (4) is a weak equivalence by the 2 out of 3 axiom. It is also a fibration because $\sS_P$ is a simplicial model category.
In particular, (4) is a surjection.
Now consider the following diagram
\begin{equation*}
\begin{tikzcd}
\partial(\Delta^J)
\arrow{r}
\arrow{d}
& X
\arrow{d}{f}
\\
\Delta^J
\arrow{r}
\arrow[dashrightarrow]{ur}
&Y
\end{tikzcd}
\end{equation*}
The solid arrows diagram corresponds to one element in $\Map(\Delta^J,Y)\times_{\Map(\partial(\Delta^J),Y)} \Map(\partial(\Delta^J),X)$. The fact that (4) is a surjection guarantees that a dotted arrow exist. We showed that $f$ has the RLP against all generating cofibrations, so it must be a trivial fibration.
\end{proof}

\begin{proof}[Proof of Proposition \ref{DiagramFibration}]
The fact that $D(f)$ preserve fibrations and trivial fibrations comes from the simplicial model category structure. Suppose $D(f)$ is a trivial fibration. Then, we have that all the application $\Map(\Delta^J,X)\to \Map(\Delta^J,Y)$ are weak equivalences.
We now have to show that all the applications $\Map(\partial(\Delta^J),X)\to \Map(\partial(\Delta^J),Y)$ are also weak equivalences to use Lemma \ref{LemmeDiagramme}.
To see this, recall that $\partial(\Delta^J)\simeq \colim_{\Delta^{J'}\subset \partial(\Delta^J)}\Delta^{J'}$. This is true for any simplicial set. Then, we have that
\begin{align*}
\Map(\partial(\Delta^J),X) & \simeq \Map\left(\colim_{\Delta^{J'}\subset \partial(\Delta^J)}(\Delta^{J'}),X\right)\\
&\simeq \lim_{\Delta^{J'}\subset \partial(\Delta^J)}\Map(\Delta^{J'},X)
\end{align*}
But since the limit is taken over fibrations, we have a weak equivalence $\lim \Map(\Delta^{J'},X)\sim \holim \Map(\Delta^{J'},X)$. Doing the same for $Y$, and using the fact that we have weak equivalences levelwise, we get 
\begin{align*}
\Map(\partial(\Delta^J),X) & \simeq \lim \Map(\Delta^{J'},X)\\
&\sim \holim \Map(\Delta^{J'},X) \\
&\sim \holim \Map(\Delta^{J'},Y) \\
& \simeq \lim \Map(\Delta^{J'},Y) \\
& \simeq \Map(\partial(\Delta^J),Y)
\end{align*}
We now have the required hypotheses to apply the lemma \ref{LemmeDiagramme}.
\end{proof}

\begin{prop}\label{FibrantWeakEquivalence}
Let $f\colon X\to Y$ be a morphism between fibrant objects, $f$ is a weak equivalence if and only if $D(f)$ is. 
\end{prop}

\begin{proof}
Proposition \ref{DiagramFibration} implies that $D$ is a right Quillen functor. As such, it preserves weak equivalences between fibrant objects.
To show the reverse implication, let $f\colon X\to Y$ be a morphism between fibrant objects such that $D(f)$ is a weak equivalence. Then, let $f=pi$ be a factorisation of $f$ into a trivial cofibration followed by a fibration. Applying $D$ to the factorization, we get the following diagram
\begin{equation*}
\begin{tikzcd}
D(X)
\arrow{r}{D(f)}
\arrow[swap]{d}{D(i)}
&D(Y)
\\
D(Z)
\arrow[swap]{ur}{D(p)}
\end{tikzcd}
\end{equation*}
We know that $Z$ is a fibrant object, by construction. But then $D(i)$ is the image of a weak equivalence between fibrant objects, so it is still a weak equivalence. By the 2 out of 3 property, $D(p)$ is a weak equivalence. Hence by Proposition \ref{DiagramFibration}, $p$ is a trivial fibration. But now, $f$ is the composition of a trivial fibration and a trivial cofibration so it is a weak equivalence. 
\end{proof}

The last thing needed to define filtered homotopy groups is a suitable notion of pointings.

\begin{defin}[Pointing]\label{PointingSimplicialSets}
Let $X$ be a filtered simplicial set. Let $V\subset N(P)$ be any inclusion of filtered simplicial sets. A pointing of $X$ over $V$ is any map of filtered simplicial sets $\phi\colon V\to X$.
Given such a pointing and a subset $V'\subset V$ we get an  induced pointing over $V'$, $\phi_{|V'}\colon V'\to X$. A pointing of $X$ is the choice of some $V\subseteq N(P)$ and some morphism $\phi\colon V\to X$, the filtered simplicial set $V$ will often be left implicit.

Let $X$ be a filtered simplicial set, and $\phi\colon V\to X$ a pointing of $X$. For all $\Delta^J\subseteq V$, we get a pointing of $\Map(\Delta^J,X)$ by taking $\phi_{|\Delta^J}$ as a base-point. We write $(\Map(\Delta^J,X),\phi)$ for the corresponding pointed simplicial set with the convention that $(\Map(\Delta^J,X),\phi)=\emptyset$ if $\Delta^J\not \subset V$.\end{defin}

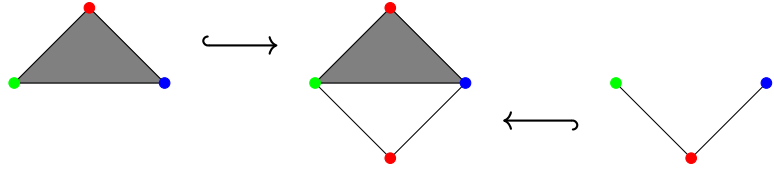
\begin{figure}[h]
\begin{tikzpicture}
\filldraw[gray] (0,0)--(1,1)--(2,0)--(0,0);
\draw[black] (0,0)--(1,1);
\draw[black](1,1)--(2,0);
\draw[black](2,0)--(0,0);
\filldraw[red] (1,1) circle (2pt);
\filldraw[blue] (2,0) circle (2pt);
\filldraw[green] (0,0) circle (2pt);
\draw[thick, right hook->] (2.5,0.5)--(3.5,0.5);

\filldraw[gray,cm ={1,0,0,1,(4cm,0cm)}] (0,0)--(1,1)--(2,0)--(0,0);
\draw[black,cm ={1,0,0,1,(4cm,0cm)}] (0,0)--(1,1);
\draw[black,cm ={1,0,0,1,(4cm,0cm)}](1,1)--(2,0);
\draw[black,cm ={1,0,0,1,(4cm,0cm)}](2,0)--(0,0);
\draw[black,cm ={1,0,0,1,(4cm,0cm)}] (0,0)--(1,-1);
\draw[black,cm ={1,0,0,1,(4cm,0cm)}](1,-1)--(2,0);
\filldraw[red,cm ={1,0,0,1,(4cm,0cm)}](1,-1) circle (2pt);
\filldraw[red,cm ={1,0,0,1,(4cm,0cm)}] (1,1) circle (2pt);
\filldraw[blue,cm ={1,0,0,1,(4cm,0cm)}] (2,0) circle (2pt);
\filldraw[green,cm ={1,0,0,1,(4cm,0cm)}] (0,0) circle (2pt);

\draw[black,cm ={1,0,0,1,(8cm,0cm)}] (0,0)--(1,-1);
\draw[black,cm ={1,0,0,1,(8cm,0cm)}](1,-1)--(2,0);
\filldraw[red,cm ={1,0,0,1,(8cm,0cm)}](1,-1) circle (2pt);
\filldraw[blue,cm ={1,0,0,1,(8cm,0cm)}] (2,0) circle (2pt);
\filldraw[green,cm ={1,0,0,1,(8cm,0cm)}] (0,0) circle (2pt);
\draw[thick, right hook->,cm ={1,0,0,1,(8cm,0cm)}] (-0.5,-0.5)--(-1.5,-0.5);
\end{tikzpicture}
\caption{Two pointings of a filtered simplicial set}
\label{FigurePointing}
\end{figure}

\begin{exemple}\label{ExamplePointing}
Let us look at examples of pointings. Figure \ref{FigurePointing} represents two different pointings of some fibrant filtered simplicial set, $X$. According to Remark \ref{PointingUpToHomotopy}, those are the only two pointing needed to apply Theorem \ref{WhiteheadSimplicial}.
\end{exemple}

\begin{remarque}
There are multiple ways of defining pointing for filtered simplicial sets that will give equivalent notion of filtered homotopy groups. Looking at remark \ref{PointingUpToHomotopy}, one sees that it would have been enough to define pointings over non-degenerate simplices of $N(P)$. Furthermore, given some $V\subseteq N(P)$ and some filtered simplicial set $X$, there may not exist a pointing of $X$ over $V$. Nonetheless, allowing for a more general notion of pointing allows one to consider less pointings in total. In Example \ref{ExamplePointing}, only $2$ pointing are needed whereas $3$ would be needed if one were to consider pointing over non-degenerate simplices of $N(P)$. Another question is whether the empty pointing should be allowed or not. One can choose to forbid it, for consistency with the non-filtered case, or to allow it to avoid future distinctions of cases. In any case, Definition \ref{PointingSimplicialSets} is made to handle partial pointings, of which the empty pointing is only an example. Note that when $\phi$ is a pointing of $X$ defined over a strict sub-simplicial set $V\subsetneq N(P)$, the filtered homotopy groups, $s\pi_n(X,\phi)$, of the next definitions are only \textit{partially} defined. That is, they are defined as functors from some subcategory $\Delta(V)^{\op}\subseteq\Delta(P)^{\op}$ to $\Set$ (or $\Grp$). To avoid unnecessary confusion, we will not make this distinction of cases in the future, and for $n\geq 1$, one can safely extend those functors by $0$ where they are not defined.
\end{remarque}

\begin{defin}[Filtered Homotopy Groups]\label{FilteredHomotopyGroupsSimplicial}
Let $X$ be a fibrant filtered simplicial set, and $\phi\colon V\to X$ be a pointing of $X$. We define the functor of filtered connected components of $X$ to be the functor
\begin{align*}
s\pi_0(X) \colon \Delta(P)^{\op}&\to \Set\\
\Delta^J&\mapsto \pi_0(\Map(\Delta^J,X)).
\end{align*} 
We can recover a functor which takes values in pointed sets, $ s\pi_0(X,\phi)$, by taking $s\pi_0(X,\phi)(\Delta^J)=\pi_0(\Map(\Delta^J,X),\phi)$.
Let $n$ be a positive integer.
We define the $n$-th filtered homotopy group of $X$ with base point $\phi$ to be the following functor :
\begin{align*}
 s\pi_n(X,\phi)\colon \Delta(P)^{\op}& \to \Grp\\
\Delta^J&\mapsto \pi_n(\Map(\Delta^J,X),\phi)
\end{align*}
\end{defin}

\begin{remarque}\label{RemarqueFilteredHomotopyGroupsSimplicialSets}
In Definition \ref{FilteredHomotopyGroupsSimplicial}, the functor $s\pi_0(X)$ can be seen as a filtered simplicial set since it is a functor from $\Delta(P)^{\op}$ to $\Set$ (see Proposition \ref{PresheavesCategory}). We will use this observation in order to "draw" the $0$-th filtered homotopy groups of filtered simplicial sets (or spaces), as a compact way of conveying the information it contains, see Figure \ref{FigureExamplePi0}. This point of view is also relevant when thinking about pointings, since they can be thought of as maps from pieces of $s\pi_0(X)$ into $X$, see remark \ref{PointingUpToHomotopy}.
One could also chose to consider the $n$-th filtered homotopy groups as filtered simplicial sets (or groups), but this point of view does not add much clarity. We will only use it in the statement $e)$ of Corollary \ref{CorollaireAlmostMiller}. Here, the $1$-skeleton of the $n$-th filtered homotopy groups is to be understood as the restriction of the functor $s\pi_n(A,\phi)$ to the full subcategory of $\Delta(P)^{op}$ containing the filtered simplices of dimension $\leq 1$.

The structure of the filtered homotopy groups is quite rich but in part redondant. We will see (Lemma \ref{WEonNonDegenerate}) that it is sufficient to compute the functor on non-degenerate simplices of $\Delta(P)$ (i.e. on simplices of the form $\Delta^{J_0}=[p_0,\dots,p_d]$ with $p_0<\dots<p_d$). What remains, is the data of (the homotopy groups of) the strata and generalized holinks (see Definition \ref{GeneralizedHolink}), together with the information of how they map into each other. For example, $s\pi_n(X,\phi)([p])=\pi_n(X_p,\phi([p]))$, where $X_p$ is the $p$-stratum of $X$, and $s\pi_n(X,\phi)([p,q])=\pi_n(\Hol_{p,q}(X),\phi([p,q]))$, where $\Hol_{p,q}(X)$ is the simplicial analogue of the homotopy links of the $p$-stratum into the $q$-stratum of $X$.



Furthermore, let $P=\{p\}$ be the poset with only one element, $K$ be a Kan complex, and $k\colon \Delta^0\to K$ a vertex of $K$. We can see $K$ as simplicial set trivially filtered over $P$, and compute $s\pi_n(K,k)$. There is only one non-degenerate simplex in $\Delta(\{p\})$, $[p]$, and $s\pi_n(K,k)([p])=\pi_n(K,k)$. In particular, we recover the usual definition of homotopy groups when working over a trivial poset.
\end{remarque}

We can now state a characterization of homotopy equivalences between fibrant filtered simplicial sets using the filtered homotopy groups.

\begin{theo}\label{WhiteheadSimplicial}
Let $f\colon X\to Y$ be a filtered map between fibrant filtered simplicial sets. It is a filtered homotopy equivalence if and only if the induced maps $ s\pi_n(f)\colon  s\pi_n(X,\phi)\to  s\pi_n(Y,f\circ\phi)$ and $ s\pi_0(f)\colon  s\pi_0(X)\to  s\pi_0(Y)$ are isomorphisms of functors for all $n\geq 1$ and all pointing $\phi\colon V\to X$.
\end{theo}

\begin{proof}
Let $f\colon X\to Y$ be a morphism between fibrant filtered simplicial sets.
The following affirmation are equivalent by previous results
\begin{itemize}
\item $f$ is a homotopy equivalence
\item $f$ is a weak equivalence
\item $D(f)$ is a weak equivalence
\item $D(f)(\Delta^J)$ is a weak equivalence for all $\Delta^J\subseteq N(P)$.
\end{itemize}
Fix $\Delta^J\subseteq N(P)$. We know that $D(f)(\Delta^J)\colon \Map(\Delta^J,X)\to \Map(\Delta^J,Y)$ is a weak equivalence if and only if it induces isomorphisms 
\begin{equation*}
 \pi_n(D(f)(\Delta^J))\colon \pi_n(\Map(\Delta^J,X),*)\to \pi_n(\Map(\Delta^J,Y),D(f)(\Delta^J)(*)),
\end{equation*}
for all $n\geq 1$ and all $*\in \Map(\Delta^J,X)_0$, and a bijection  
\begin{equation*}
 \pi_0(D(f)(\Delta^J))\colon\pi_0(\Map(\Delta^J,X))\to\pi_0(\Map(\Delta^J,Y))
\end{equation*}
 But each vertex of $\Map(\Delta^J,X)$ corresponds to a pointing of $X$, so we get the desired equivalence.
\end{proof}

We will also need the following proposition, telling us that the filtered homotopy groups are homotopy invariants.

\begin{prop}\label{HomotopyPi}
Let $f,g\colon X\to Y$ be two homotopic morphisms between two fibrant filtered simplicial sets. Then $ s\pi_n(f)= s\pi_n(g)$ for all $n$.
\end{prop}

\begin{proof}
Let $H\colon X\otimes \Delta^{1}\to Y$ be a homotopy between $f$ and $g$. Then, applying $\Map(\Delta^J,-)$ to the diagram giving the homotopy between $f$ and $g$, we get the right part of the following diagram

\begin{equation*}
\begin{tikzcd}
\phantom{X}
&\Map(\Delta^{J},X)
\arrow[swap]{dl}{\iota_0}
\arrow{dr}{\Map(\Delta^J,\iota_0)}
\arrow[bend left=10]{drrr}{\Map(\Delta^J,f)}
&\phantom{X}
&\phantom{X}
&\phantom{X}
\\
\Map(\Delta^J,X)\times \Delta^1
\arrow{rr}{h}
&\phantom{X}
&\Map(\Delta^J,X\otimes \Delta^1)
\arrow{rr}{\Map(\Delta^J,H)}
&\phantom{X}
&\Map(\Delta^J,Y)
\\
\phantom{X}
&\Map(\Delta^J,X)
\arrow{ul}{\iota_1}
\arrow[swap]{ur}{\Map(\Delta^J,\iota_1)}
\arrow[swap, bend right=10]{urrr}{\Map(\Delta^J,g)}
\end{tikzcd}
\end{equation*}
Now the morphism $h$ can be defined on $n$-simplices by 
\begin{align*}
\Map(\Delta^J,X)_n\times \Delta^1_n &\to \Map(\Delta^J,X\otimes\Delta^1)_n\\
(\sigma\colon \Delta^J\otimes\Delta^n\to X,\tau\colon \Delta^n\to \Delta^1)&\mapsto (\sigma\otimes(\tau\circ \pr_{\Delta^n})\colon \Delta^J\otimes \Delta^n\to X\otimes \Delta^1)
\end{align*}
Where $\pr_{\Delta^n}\colon \Delta^J\otimes \Delta^n\to \Delta^n$ is the projection. The composition $\Map(\Delta^J,H)\circ h$ gives a homotopy between $\Map(\Delta^J,f)$ and $\Map(\Delta^J,g)$, so they must induce the same morphism on homotopy groups.
\end{proof}

\begin{prop}\label{HomotopyPointing}
Let $\phi,\psi\colon V\to X$ be two pointings of $X$, and $\theta\colon V\otimes \Delta^1\to X$ a homotopy between them. And let $f\colon X\to Y$ be an application between fibrant filtered simplicial sets. For all $n\geq 0$, there is a commuting square where the vertical morphisms are isomorphisms of functors
\begin{equation*}
\begin{tikzcd}
 s\pi_n(X,\phi)
\arrow{r}{ s\pi_n(f)}
\arrow{d}{ s\pi_n(X,\theta)}
& s\pi_n(Y,f\circ\phi)
\arrow{d}{ s\pi_n(Y,f\circ\theta)}
\\
 s\pi_n(X,\psi)
\arrow{r}{ s\pi_n(f)}
& s\pi_n(Y,f\circ\psi)
\end{tikzcd}
\end{equation*}
\end{prop}

\begin{proof} 
Taking $\Delta^J\subseteq V$ and applying $\Map(\Delta^J,-)$ to the data of Proposition \ref{HomotopyPointing}, we get the following diagram of pointed simplicial sets :
\begin{equation*}
\begin{tikzcd}[column sep=huge]
(\Map(\Delta^J,X),\phi)
\arrow{r}{\Map(\Delta^J,f)}
&(\Map(\Delta^J,Y),f\circ\phi)
\\
(\Map(\Delta^J\otimes\Delta^1,X),\theta)
\arrow{u}{\Map(\iota_0,X)}
\arrow[swap]{d}{\Map(\iota_1,X)}
\arrow{r}{\Map(\Delta^J\otimes\Delta^1,f)}
&(\Map(\Delta^J\otimes\Delta^1,Y),f\circ\theta)
\arrow[swap]{u}{\Map(\iota_0,Y)}
\arrow{d}{\Map(\iota_1,Y)}
\\
(\Map(\Delta^J,X),\psi)
\arrow{r}{\Map(\Delta^J,f)}
&\Map(\Delta^J,Y),f\circ\psi)
\end{tikzcd}
\end{equation*}
Now, vertical morphisms are all trivial fibrations. By passing to homotopy groups, we get the isomorphisms $ \pi_n(\Map(\iota_0,-))$ and $ \pi_n(\Map(\iota_1,-))$. We get the desired isomorphism by taking 
\begin{equation*}
s\pi_n(-,\theta)(\Delta^J)= \pi_n(\Map(\iota_1,-))\circ (\pi_n(\Map(\iota_0,-)))^{-1}.
\end{equation*} 
\end{proof}

\begin{remarque}\label{PointingUpToHomotopy}
Proposition \ref{HomotopyPointing} allows us to apply Theorem \ref{WhiteheadSimplicial} in a much more efficient way. Indeed, we do not have to compute the value of $ s\pi_n(f)\colon  s\pi_n(X,\phi)\to  s\pi_n(Y,f\circ\phi)$ for all pointing $\phi$, but only for one point in each homotopy class of points of $X$. More precisely, for all maximal simplices $\sigma\in  s\pi_0(X)$ choose a pointing $\phi_{\sigma}\colon \Delta^J\to X$ such that $[\phi_{\sigma}]=\sigma$.
Then, under the hypothesis of Theorem \ref{WhiteheadSimplicial}, $f$ is a homotopy equivalence if and only if the maps of filtered simplicial sets $ s\pi_n(f)\colon  s\pi_n(X,\phi_{\sigma})\to s\pi_n(Y,f\circ\phi_{\sigma})$ are isomorphisms for all $n$ and all maximal simplices $\sigma\in  s\pi_0(X)$. Furthermore, if $ s\pi_0(X)\subset N(P)$, which should be interpreted as $X$ being connected in a filtered sense, and if there exists some map $\phi_X\colon  s\pi_0(X)\to X$, then $f$ is a homotopy equivalence if and only if the maps of filtered simplicial sets $ s\pi_n(f)\colon  s\pi_n(X,\phi_X)\to  s\pi_n(Y,f\circ\phi_X)$ are isomorphisms for all $n$. 
%
\end{remarque}

\begin{remarque}
Taking $\phi=\psi$ Proposition \ref{HomotopyPointing}, tells us that $ s\pi_1(X,\phi)$ acts on $ s\pi_n(X,\phi)$.
\end{remarque}

Let us summarize what we have proved so far :

\begin{theo}\label{WhatWeHaveSoFar}
The category of filtered simplicial sets over a poset $P$ is a simplicial combinatorial model category where,
\begin{itemize}
\item the cofibrations are the monomorphisms,
\item the fibrations are the maps that satisfy the right lifting property with respect to admissible horn inclusions
\item the weak equivalences are the maps $f\colon X\to Y$ inducing isomorphisms between the filtered homotopy groups of fibrant replacements of $X$ and $Y$.
\end{itemize}
\end{theo}

\section{Filtered topological spaces}\label{FilteredTopologicalSpaces}

In this section, we define a category of topological spaces filtered over $P$. We then use Theorem \ref{WhiteheadSimplicial} to prove a filtered version of Whitehead's Theorem : Theorem \ref{Whitehead}.

\subsection{The simplicial category of filtered topological spaces}

\begin{defin}[Filtered topological spaces]
Let $P$ be a partially ordered set. We will consider it as a topological space with the Alexandrov topology.
We define the category of filtered topogical spaces as follows 
\begin{itemize}
\item The object are the pairs $(A,\varphi)$ where $A$ is a topological space, and $\varphi\colon A\to P$ is a continuous map.
\item The maps from $(A,\varphi_A)$ to $(B,\varphi_B)$ are continuous maps $f\colon A\to B$ such that the following diagram commutes 
\begin{equation*}
\begin{tikzcd}
A
\arrow{rr}{f}
\arrow[swap]{dr}{\varphi_A}
&\phantom{X}
&B
\arrow{dl}{\varphi_B}
\\
\phantom{X}
&P
&\phantom{X}
\end{tikzcd}
\end{equation*}
\item The composition and identity are those of topological spaces and continous maps.
\end{itemize}
\end{defin}

Recall the definition of the realisation for simplicial sets 

\begin{defin}
Let $K$ be a simplicial set, we define the geometric realisation of $K$ as the quotient 
\begin{equation*}
\Real{K}=\left(\coprod_n K_n\times \Real{\Delta^n}\right)/\sim
\end{equation*}
where the relation is generated by $(\sigma,S^{i}(t))\sim(s_i(\sigma),t)$ and $(\sigma,D^{i}(t))\sim(d_i(\sigma),t)$.
\end{defin}

We define the application $\varphi_P\colon \Real{N(P)}\to P$ 
by $\varphi_P([p_0,\dots,p_n],t)=p_n$ when $t$ is in the interior of $[p_0,\dots,p_n]$, and $\varphi_P(\{p\})=p$.

\begin{defin}[Filtered geometric realization]\label{DefinitionFilteredRealisation}
Let $X\xrightarrow{\pi} N(P)$ be a filtered simplicial set. We define its filtered geometric realisation over $P$ to be the filtered topological space $\Real{X}\xrightarrow{\varphi_P\circ\Real{\pi}} P$.
This definition extends to a functor 
\begin{equation*}
\Real{-}_P\colon \sS_P\to \Top_P
\end{equation*}
\end{defin}

\begin{remarque}
Notice that one can extend the definition of $\Er$ and $\Pr$ from Proposition \ref{ForgetFree}, to the category of filtered topological spaces. Doing so, one gets a pair of adjoint functors 
\begin{align*}
\Er\colon \Top_P&\to \Top\colon \Pr\\
(A\to P)&\mapsto A\\
(M\times P\to P)&\mapsfrom M
\end{align*}
One also checks that $\Real{\Er(X)}_P\simeq \Er(\Real{X}_P)$, that is why we use the same notation for the two functors.
\end{remarque}

\begin{defin}[Simplicial structure on $\Top_P$]
We define the following functors :
\begin{align*}
-\otimes-\colon \Top_P\times \sS &\to \Top_P\\
(A\to P,K)&\mapsto A\times \Real{K}\to P\\
\phantom{X}&\phantom{X}\\
\Map\colon\Top_P^{\op}\times\Top_P&\to \sS\\
(A,B)&\mapsto \left\{\begin{array}{rcl}
\Map(A,B)\colon\Delta^{\op} &\to &\Set\\
\Delta^n&\mapsto\ &\Hom_{\Top_P}(A\otimes \Delta^n,B)
\end{array}\right.
\end{align*}
In addition, let $K$ be a simplicial set, and $p\in P$. Let $K_p$ be the filtered simplicial set $K\to P$ whose image in $P$ is exactly the vertex $p$. Let $B$ be a filtered topological space, we define $B^{K}$ as follows.
As a topological space :
\begin{equation*}
B^K= \coprod_{p\in P}\Hom_{\Top_P}(\Real{K_p}_P,B)\subseteq \Hom_{\Top}(\Real{K},B).
\end{equation*}
It inherits the topology of the inner $\Hom$ of topological spaces. It is also naturally filtered with the filtration sending any morphism in $\Hom_{\Top_p}(\Real{K_p}_P,B)$ to $p$.
This defines a functor 

\begin{equation*}
(-)^-\colon \sS^{\op}\times \Top_P\to \Top_P
\end{equation*}
\end{defin}

\begin{prop}
With the functor defined above, the category $\Top_P$ of filtered topological spaces over the poset $P$ is a simplicial category.
\end{prop}

\begin{proof}
The proof goes as for Proposition \ref{SimplicialCategory} using \cite[Lemma II.2.3]{GoerssJardine}. The only non-obvious fact is that given a simplicial set $K$ the functor $(-)^K$ is right adjoint to the functor $-\otimes K$.
Let $A$ and $B$ be two filtered simplicial set. The bijections are given by 
\begin{align*}
\Hom(A\otimes K,B)&\simeq \Hom(A,B^K)\\
(f\colon A\times \Real{K}\to B)&\mapsto \left\{
\begin{array}{rcl}
A & \to &B^K\\
a &\mapsto & (k\mapsto f(a,k))
\end{array}\right.\\
\left.
\begin{array}{rcl}
A\times \Real{K}&\to& B\\
(a,k)&\mapsto& g(a)(k)
\end{array}\right\}
&\mapsfrom (g\colon A\to B^K)
\end{align*}
\end{proof}

There is now a canonical way to define filtered homotopies between filtered spaces.

\begin{defin}
Let $f,g\colon A\to B$ be two filtered maps. We say that $f$ and $g$ are filtered homotopic (written $f\sim g$) if there exists a filtered map $H\colon A\otimes \Delta^1\to B$ such that $H\circ (\Id_A\otimes\iota_0)=f$ and $H\circ(\Id_B\otimes \iota_1)=g$, where $\iota_{\epsilon}$ is the map sending $\Delta^0$ to the corresponding vertex of $\Delta^1$.
If $f\colon A\to B$ is a filtered map, we say that $f$ is a filtered homotopy equivalence, and that $A$ and $B$ are filtere homotopy equivalent if there exists $g\colon B\to A$ such that $f\circ g\sim \Id_B$ and $g\circ f\sim \Id_A$. 
\end{defin}

We will also need the following functor

\begin{defin}[Filtered singular simplices]
Let $A$ be a filtered topological space over $P$. We define the filtered simplicial set of its filtered singular simplices $\Sing_P(A)$ by
\begin{equation*}
\Sing_P(A)_J=\Hom_{\Top_P}(\Real{\Delta^J}_P,A),
\end{equation*}
where faces and degeneracies are taken as in the non filtered case.
\end{defin}

\begin{prop}\label{SimplicialAdjunction}
There is an adjunction of simplicial categories. 
\begin{equation*}
\Real{-}_P\colon \sS_P\leftrightarrow\Top_P\colon \Sing_P
\end{equation*}
\end{prop}

\begin{proof}
Let $X$ be a filtered simplical set and $A$ a filtered topological space. The bijections are as follows
\begin{align*}
\Hom_{\Top_P}(\Real{X}_P,A)&\simeq \Hom_{\sS_P}(X,\Sing_P(A))\\
(f\colon \Real{X}\to A)&\mapsto \left\{ \begin{array}{rcl}
X & \to & \Sing_P(A)\\
\sigma\colon \Delta^J\to X & \mapsto & f\circ\Real{\sigma}_P\colon \Real{\Delta^J}_P\to A
\end{array}\right.\\
\left.\begin{array}{rcl}
\widehat{g}\colon\Real{X}&\to &A\\
(\sigma,t)&\mapsto &g(\sigma)(t)
\end{array}
\right\}
&\mapsfrom (g\colon X\to \Sing_P(A))
\end{align*}
Since those applications are the same as in the non filtered case, it is immediate that they form a bijection. What needs to be checked is that they are well defined. It is clear that $f\circ\Real{\sigma}_P$ is a filtered application since it is the composite of two filtered applications. Now, to see that $\widehat{g}$ is filtered, consider the following diagram.
\begin{equation*}
\begin{tikzcd}
\Real{X}
\arrow{r}{\widehat{g}}
&A
\\
\Real{\Delta^J}
\arrow{u}{\Real{\sigma}}
\arrow[swap]{ur}{g(\sigma)}
&\phantom{X}
\end{tikzcd}
\end{equation*}
we know that $\Real{\sigma}$ and $g(\sigma)$ are filtered, so the restriction of $\widehat{g}$ to the image of $\Real{\sigma}$ is necessarily filtered. But then, $\widehat{g}$ is filtered because it is on every simplex of the realisation.

We then get the simplicial adjunction by noticing the following 
\begin{align*}
\Map(\Real{X}_P,A)_n &\simeq \Hom_{\Top_P}(\Real{X}_P\times \Real{\Delta^n},A)\\
&\simeq \Hom_{\Top_P}(\Real{X\otimes\Delta^n}_P,A)\\
&\simeq \Hom_{\sS_P}(X\otimes\Delta^n,\Sing_P(A))\\
&\simeq \Map(X,\Sing_P(A))_n
\end{align*}
\end{proof}

\subsection{Fibrant objects}\label{FibrantObjects}

Contrary to the classical case, it is not always true that $\Sing_P(A)$ is fibrant, see Example \ref{ExampleNonFibrantSpaces}. The object of this subsection is to show that there is a large class of filtered spaces, $A$, such that $\Sing_P(A)$ is fibrant, and that useful examples are in this class. To do so we will need the definition of conically stratified spaces from \cite[Definition A.5.5] {HigherAlgebra}.

\begin{defin}
Let $A$ be a filtered space, and $a\in A$ be a point. Write $p=\pi(a)$. We say that $A$ is conically stratified at the point $a$ if there exists $B$ a filtered space over $P_{>p}$, and a space $C$ such that there is a filtered open embedding $C\times c(B)\hookrightarrow A$ whose image contains $a$.
Here, $P_{>p}$ is the subposet of $P$ whose elements are all strictly greater than $p$, $c(B)$ is the cone on $B$, and the filtration on $C\times c(B)$ is given by $\varphi(c,0)=p$ and $\varphi(c,(b,t))=\varphi(b)$ for $t>0$. The filtered space $A$ is said to be conically stratified if it is conically stratified at every point.
\end{defin}

\begin{remarque}
In particular (PL)-pseudomanifolds, in the original sense of Goresky and MacPherson \cite{IntersectionHomologyI}, are conically stratified.
\end{remarque}

In this subsection, we prove the following :

\begin{prop}\label{ConStratFibrant}
Any conically stratified space is fibrant.
\end{prop}

Where fibrant is meant in the following sense.

\begin{defin}
A filtered topological space $A$ is said to be fibrant if $\Sing_P(A)$ is fibrant.
Equivalently, by adjunction, $A$ is fibrant if and only if for any diagram of the form 
\begin{equation*}
\begin{tikzcd}
\Real{\Lambda^J_{k}}_P
\arrow{r}
\arrow{d}
&A
\\
\Real{\Delta^J}_P
\arrow[dashrightarrow]{ur}
\end{tikzcd}
\end{equation*}
where the morphism $\Real{\Lambda^J_{k}}_P\to \Real{\Delta^J}_P$ is the realisation of an admissible horn inclusion, a dotted arrow exists, making the diagram commute.
\end{defin}

\begin{exemple}\label{ExampleNonFibrantSpaces}
Let us look at a few example of filtered spaces. Let $P=\{p_0<p_1\}$ be the linear poset with $2$ elements, and consider $\Delta^J=[p_0,p_0,p_1]$ and $\Lambda^J_1=[p_0,\widetilde{p_0},p_1]$. Then $\RealP{\Lambda^J_1}$ is not a fibrant filtered space, since there exist no filtered lift in the following diagram 
\begin{equation*}
\begin{tikzcd}
\Real{\Lambda^{J}_{1}}_P
\arrow{d}
\arrow{r}{\Id}
&\Real{\Lambda^{J}_{1}}_P
\\
\Real{\Delta^{J}}_P
\end{tikzcd}
\end{equation*}
But $\RealP{\Delta^J}$ will be a fibrant filtered space, as a consequence of Proposition \ref{ConStratFibrant}. We provide a picture of the two spaces in this example in Figure \ref{PictureSpaces}.
\end{exemple}

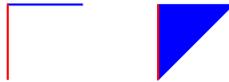
\begin{figure}[h]
\begin{tikzpicture}
\draw[blue, thick,opacity=0.5](0,1)--(1,1);
\draw[red, thick](0,0)--(0,1.014);

\filldraw[blue,blue,opacity=0.5](2,1)--(3,1)--(2,0)--(2,1);
\draw[red,thick](2,0)--(2,1.005);
\end{tikzpicture}
\caption{The non-fibrant space $\Real{[p_0,\widetilde{p_0},p_1}_P$ and the fibrant space $\Real{[p_0,p_0,p_1]}_P$.}
\label{PictureSpaces}
\end{figure}

\begin{proof}[Proof of Proposition \ref{ConStratFibrant}]
By \cite[Theorem A.6.4]{HigherAlgebra}, if $A$ is a conically stratified space, $\Er(\Sing_P(A))$ is a quasi-category. In particular, for any diagram of the form 
\begin{equation*}
\begin{tikzcd}
\Real{\Lambda^J_{k}}_P
\arrow{r}{\alpha}
\arrow{d}
&A
\\
\Real{\Delta^J}_P
\arrow[swap, dashrightarrow]{ur}{f}
\end{tikzcd}
\end{equation*}
There exists a lift as long as $N(J)\geq 2$ and $0<k<N(J)$. We prove that a lift exists in the remaining cases, when $\Lambda^J_k\to \Delta^J$ is admissible.

If $N(J)=1$, then since $\Lambda^J_k\to \Delta^J$ is admissible, we must have $\Delta^J=[p,p]$ for some $p\in P$. Then, $\Lambda^J_k=[p]$ and there is a (unique) section $s\colon \RealP{\Delta^J}\to \RealP{\Lambda^J_k}$. Taking $f=\alpha\circ s$ gives the desired lift. 

If $N(J)\geq 2$ and $k=0$, since $\Lambda^J_k\to \Delta^J$ is admissible, the $1$-st vertex of $\Delta^J$ must share its color with the $0$-th. In particular, there is a linear filtered homeomorphism $\phi\colon \RealP{\Delta^J}\to \RealP{\Delta^J}$ exchanging the $0$-th and $1$-st vertex of $\Delta^J$ and leaving the other vertices fixed. This homeomorphism sends $\Lambda^J_0$ to $\Lambda^J_1$. But now, $\Lambda^J_1$ is an inner horn, and so there exists a lift $f_1$ in the following diagram 
\begin{equation*}
\begin{tikzcd}
\Real{\Lambda^J_{1}}_P
\arrow{r}{\alpha\circ \phi^{-1}}
\arrow{d}
&A
\\
\Real{\Delta^J}_P
\arrow[swap, dashrightarrow]{ur}{f_1}
\end{tikzcd}
\end{equation*}
Taking $f=f_1\circ\phi$ gives the desired lift.

The case $k=N(J)$ is identical to the case $k=0$.
\end{proof}

\subsection{Homotopy groups}

As in the simplicial case, we can define a diagram functor 

\begin{defin}
We define the following functor
\begin{align*}
D^{\Top}\colon \Delta(P)^{\op}\times \Top_P&\to \sS\\
(\Delta^J,A)&\mapsto\Map(\Real{\Delta^J}_P,A)
\end{align*}
\end{defin}

\begin{remarque}\label{IsomorphismOfDiagram}
The simplicial adjunction of Proposition \ref{SimplicialAdjunction} gives us that the functor $D^{\Top}$ and $D(\Sing_P(-))$ are isomorphic. For this reason, the superscript $\Top$ will be omitted.
\end{remarque}

In order to define filtered homotopy groups, we need to define pointings of a filtered topological space. We will proceed in the same way as we did for filtered simplicial sets, see Definition \ref{PointingSimplicialSets}.

\begin{defin}[Pointing]\label{PointingTopologicalSpaces}
Let $A$ be a filtered topological space. Let $V\subset N(P)$ be any inclusion of filtered simplicial set. A pointing of $A$ over $V$ is a filtered map $\phi\colon \Real{V}_P\to A$.
Given such a pointing and a subset $V'\subset V$ we get an  induced pointing over $V'$, $\phi_{|V'}\colon \Real{V'}_P\to A$. A pointing of $A$ is the choice of some $V\subseteq N(P)$ and some morphism $\phi\colon \Real{V}_P\to A$, the filtered simplicial set $V$ will often be left implicit.

Let $A$ be a filtered topological space, and $\phi\colon \Real{V}_P\to A$ a pointing of $A$. For all $\Delta^J\subseteq V$, we get a pointing of $\Map(\Real{\Delta^J}_P,A)$ by taking $\phi_{|\Delta^J}$ as a base-point. We write $(\Map(\Real{\Delta^J}_P,A),\phi)$ for the corresponding pointed simplicial set with the convention that $(\Map(\Real{\Delta^J}_P,A),\phi)=\emptyset$ if $\Delta^J\not \subset V$.
\end{defin}

\begin{defin}[Filtered Homotopy Groups]\label{TopFilteredHomotopyGroup}
Let $A$ be a filtered space, and $\phi\colon \Real{V}_P\to A$ be a pointing of $A$. We define the functor of filtered connected components of $A$ to be the functor
\begin{align*}
 s\pi_0(A) \colon \Delta(P)^{\op}&\to \Set\\
\Delta^J&\mapsto  \pi_0(\Map(\Real{\Delta^J}_P,A)).
\end{align*} 
We can recover a functor which takes values in pointed sets, $ s\pi_0(A,\phi)$ by taking $s\pi_0(A,\phi)(\Delta^J)=\pi_0(\Map(\Real{\Delta^J}_P,A),\phi)$.
Let $n$ be a positive integer.
We define the $n$-th homotopy group of $A$ with base point $\phi$ to be the functor :
\begin{align*}
 s\pi_n(A,\phi)\colon \Delta(P)^{\op}& \to \Grp\\
\Delta^J&\mapsto  \pi_n(\Map(\Real{\Delta^J}_P,A),\phi)
\end{align*}
\end{defin}

\begin{exemple}
Consider the two stratified spaces of Figure \ref{FigureExamplePi0}, stratified over $P=\{p_0<p_1\}$. The first, $A$, is the pinched torus, and the second, $B$ is the gluing of two spheres along their north and south poles.  Their $s\pi_0$ is represented in Figure \ref{FigureExamplePi0}. We see that $s\pi_0(A)$ has one vertex corresponding to the regular stratum and one vertex corresponding to the singular strata, but two distinct one simplices between those vertices, corresponding to the two paths (up to filtered homotopy) going from the singular stratum to the regular stratum. Indeed, a path going from the singular part into the regular part leaves the singular part going either to the right or to the left. On the contrary, $s\pi_0(B)$ has two vertices for singular strata and two vertices for regular strata, but a vertex corresponding to a singular stratum is connected by a single $1$ simplex to a vertex corresponding to a regular stratum. Indeed, when one chooses one singular stratum and one regular stratum in $B$ there is only one path going from the former to the latter up to filtered homotopy.

In this simple example, higher filtered homotopy groups can also be computed. Choose pointings for $A$ and $B$, $\phi\colon \RealP{[p_0,p_1]}\to A$ and $\psi \colon \RealP{[p_0,p_1]}\to B$.
To compute $s\pi_n(A,\phi)$ (resp. $s\pi_n(B,\psi)$) we need to compute $\pi_n(\Map(\Delta^J,A),\phi)$  (resp. $\pi_n(\Map(\Delta^J,B),\psi)$).
It is enough to compute those for $\Delta^J=[p_0],[p_1]$ and $[p_0,p_1]$ (see Remark \ref{SkeletonHomotopyGroups}).
For $A$ and $n=1$, we have $(\Map(\RealP{[p_0]},A),\phi)\sim (*,*)$ and so $s\pi_1(A,\phi)([p_0])=0$,
 and $(\Map(\RealP{[p_0,p_1]},A),\phi)\sim (S^1\coprod S^1,*)$ and $(\Map(\RealP{[p_1]}),\phi) \sim (S^1,*)$. Furthermore, one can check that the map 
 \begin{equation*}
(\Map(\RealP{[p_0,p_1]},A),\phi) \to (\Map(\RealP{[p_1]},A),\phi)
\end{equation*}
 induces an iso on the first homotopy group. This implies that $s\pi_1(A,\phi)$ (restricted to the non-degenerate simplices of $\Delta(P)$), is given by the following diagram 
\begin{equation*}
\begin{tikzcd}
\phantom{X}
&\mathbb{Z}
\arrow{dr}{\Id}
\arrow{dl}
&\phantom{X}
\\
0
&\phantom{X}
&\mathbb{Z}
\end{tikzcd}
\end{equation*}
With $s\pi_1(A,\phi)([p_0])$ on the left, $s\pi_1(A,\phi)([p_1])$ on the right, and $s\pi_1(A,\phi)([p_0,p_1])$ on top. Furthermore, we have that $s\pi_n(A,\phi)$ is trivial for all $n>1$, that is, $s\pi_n(A,\phi)$ is the constant functor equal to $0$. Computing the higher filtered homotopy groups of $B$ gives the same result, which means that $A$ and $B$ are distinguished only by their $s\pi_0$.
\end{exemple}
\begin{figure}[h]
\includegraphics[width= 200pt]{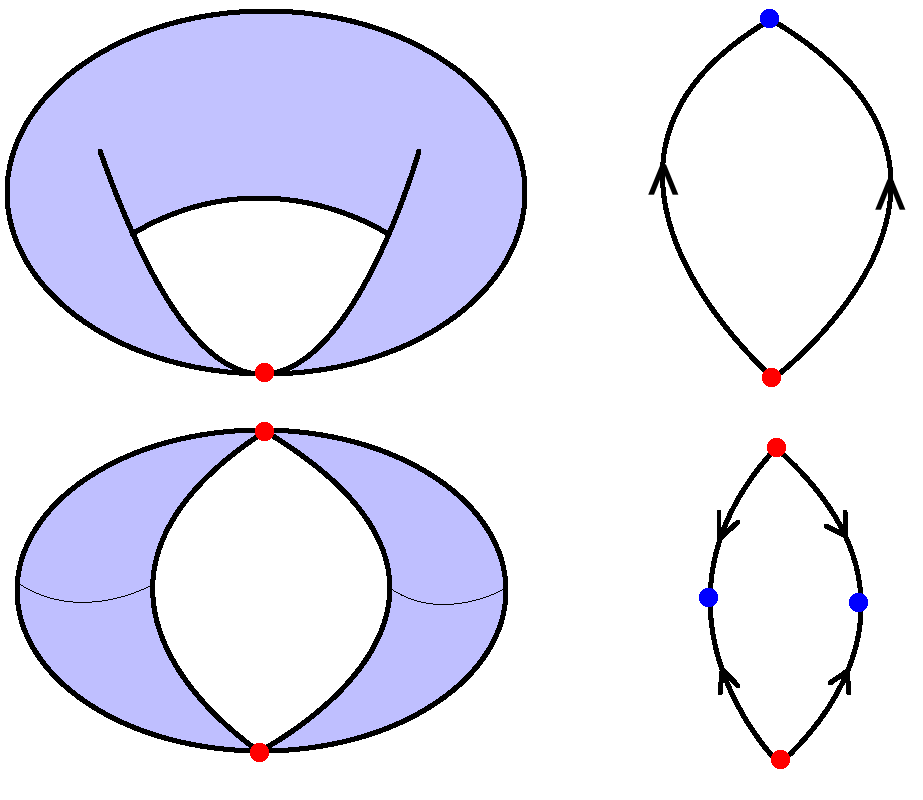}
\caption{The stratified spaces $A$ and $B$ and their associated filtered simplicial set of connected components, $s\pi_0(A)$ and $s\pi_0(B)$.}
\label{FigureExamplePi0}
\end{figure}

\begin{remarque}\label{IsomorphismOfPi}
Let $(A,\phi)$ be a pointed fibrant filtered space, we get from remark \ref{IsomorphismOfDiagram} that $ s\pi_n(A,\phi)\simeq  s\pi_n(\Sing_P(A),\phi')$ where $\phi'$ is given by the adjunction.
\end{remarque}

\begin{prop}\label{HomotopyGroupInvariant}
Let $f,g\colon A\to B$ be two homotopic maps between filtered spaces, then, for all $n\geq 0$, $ s\pi_n(f)= s\pi_n(g)$.
\end{prop}

\begin{proof}
Let $H\colon A\otimes\Delta^1\to B$ be a homotopy between $f$ and $g$. Then, applying $\Sing_P$ to the diagram giving the homotopy between $f$ and $g$, we get the right part of the following diagram
\begin{equation*}
\begin{tikzcd}
\phantom{X}
&\Sing_P(A)
\arrow[swap]{dl}{\iota_0}
\arrow{dr}{Sing(\iota_0)}
\arrow[bend left=10]{drrr}{Sing(f)}
&\phantom{X}
&\phantom{X}
&\phantom{X}
\\
\Sing_P(A)\otimes \Delta^1
\arrow{rr}{h}
&\phantom{X}
&\Sing_P(A\otimes \Delta^1)
\arrow{rr}{\Sing_P(H)}
&\phantom{X}
&\Sing_P(B)
\\
\phantom{X}
&\Sing_P(A)
\arrow{ul}{\iota_1}
\arrow[swap]{ur}{\Sing_P(\iota_1)}
\arrow[swap, bend right=10]{urrr}{\Sing_P(g)}
\end{tikzcd}
\end{equation*}
Now, the morphism $h$ can be defined as follows

\begin{align*}
\Sing_P(A)\otimes\Delta^1&\to \Sing_P(A\otimes\Delta^1)\\
\left(\sigma\colon\Real{\Delta^J}_P\to A,\tau\colon \Delta^J\to \Pr(\Delta^1)\right)&\mapsto \left(\sigma\times \Real{\tau}_P\colon \Real{\Delta^J}_P\to A\otimes\Delta^1\right).
\end{align*}
Now, $\Sing_P(H)\circ h$ gives us a homotopy between $\Sing_P(f)$ and $\Sing_P(g)$. Applying $D$, and using the same argument as in the proof of Proposition \ref{HomotopyPi}, we get a homotopy between $D(\Sing_P(f))$ and $D(\Sing_P(g)$, but now by Remark \ref{IsomorphismOfDiagram} this implies that $f$ and $g$ induce the same map between filtered homotopy groups.
\end{proof}

\begin{corollaire}\label{DirectWhitehead}
Let $f\colon A\to B$ be a filtered homotopy equivalence between two filtered spaces. Then, the maps of functors $ s\pi_0(f)$ and  $ s\pi_n(f)$, $n\geq 1$ are isomorphisms for all pointings of $A$.
\end{corollaire}

\begin{proof}
Let $g$ be the homotopy inverse of $f$. Then there exists filtered homotopies $f\circ g\sim \Id_{B}$ and $g\circ f\sim \Id_{A}$. By Proposition \ref{HomotopyGroupInvariant}, this implies that $ s\pi_n(f)\circ s\pi_n(g)= s\pi_n(\Id_B)=\Id$, and that $ s\pi_n(g)\circ  s\pi_n(f)= s\pi_n(\Id_A)=\Id$ for all $n\geq 0$. In particular, the $ s\pi_n(f)$ are isomorphisms with inverse $ s\pi_n(g)$.
\end{proof}

We now are ready to state and prove the main theorem

\begin{theo}\label{Whitehead}
Let $f\colon A\to B$ be a filtered map between two fibrant filtered spaces such that $A$ and $B$ are filtered homeomorphic to the realisation of some filtered simplicial sets $X$ and $Y$. Then $f$ is a filtered homotopy equivalence if and only if the maps of functors $ s\pi_0(f)\colon  s\pi_0(A)\to  s\pi_0(B)$ and $ s\pi_n(f)\colon  s\pi_n(A,\phi)\to s\pi_n(B,f\circ\phi)$ are isomorphisms for all $n\geq 0$ and all pointings of $A$. 
\end{theo}

\begin{proof}
The direct implication is precisely Corollary \ref{DirectWhitehead}.
For the reverse implication, the following diagram contains the idea behind the proof
\begin{equation*}
\begin{tikzcd}[column sep=large]
\phantom{X}
&\Real{\Sing_P(B)}_P
\arrow{r}{\Real{\widetilde{g}}_P}
&\Real{\Sing_P(A)}_P
\arrow{d}{\ev_A}
\arrow{r}{\Real{\Sing_P(f)}_P}
&\Real{\Sing_P(B)}
\arrow{d}{\ev_B}
\\
A
\arrow{r}{h}
\arrow{d}{\iota_A}
&B
\arrow{r}{g}
\arrow{u}{\iota_B}
&A
\arrow{r}{f}
&B\\
\Real{\Sing_P(A)}_P
\arrow{r}{\Real{\widetilde{h}}_P}
&\Real{\Sing_P(B)}_P
\arrow{r}{\Real{\Sing_P(g)}_P}
\arrow{u}{\ev_B}
&\Real{\Sing_P(A)}_P
\arrow{u}{\ev_A}
\end{tikzcd}
\end{equation*}
Let $f\colon A\to B$ be as in the theorem and such that $f$ induces isomorphism on every homotopy groups. Using $\ev$, the co-unit of the adjunction between $\Sing_P$ and $\Real{-}_P$, we get the rightmost square of our diagram.
Now, by Remark \ref{IsomorphismOfPi}, $\Sing_P(f)$ induces isomorphisms on all homotopy groups with pointings coming from pointing of $A$. But, by adjunction, any pointing of $\Sing_P(A)$ comes from a pointing of $A$. So, since $A$ and $B$ are fibrant by hypothesis, we can apply Theorem \ref{WhiteheadSimplicial},
to get $\widetilde{g}\colon \Sing_P(B)\to \Sing_P(A)$, a homotopy inverse to $\Sing_P(f)$. Now since $B\simeq \Real{Y}_P$, we can define 
\begin{align*}
\iota_B\colon B&\to \Real{\Sing_P(B)}_P\\
(\sigma\colon \Delta^J\to Y,t)&\mapsto (\Real{\sigma}_P\colon \Real{\Delta^J}_P\to \Real{Y}_P,t)
\end{align*}
It is in fact the realisation of the unit of the adjunction between $\Sing_P$ and $\Real{-}_P$. Defining \\ $g=\ev_A\circ\Real{\widetilde{g}}_P\circ\iota_B$, we get the middle upper commutative square.
We claim that $g$ is a right homotopy inverse to $f$. To see that, let $\widetilde{H}\colon \Sing_P(B)\otimes\Delta^1\to \Sing_P(B)$ be a homotopy between
$\Sing_P(f)\circ\widetilde{g}$ and $\Id_{\Sing_P(B)}$. Then, $H:=\ev_B\circ\Real{\widetilde{H}}_P\circ(\iota_B\otimes\Delta^1)
\colon B\otimes \Delta^1\to B$ gives us a homotopy between $f\circ g$ and $\Id_B$.
But then, using Proposition \ref{HomotopyGroupInvariant}, We know that $ s\pi_n(f)\circ s\pi_n(g)= s\pi_n(f\circ g)=\Id_{ s\pi_n(B)}$. And since $ s\pi_n(f)$ is an isomorphism by hypothesis, we get that $ s\pi_n(g)$ is also an isomorphism. But then, we can apply the same construction to $g$, to get the last two bottom squares. We get a map $h\colon A\to B$ and a homotopy $H'\colon A\otimes\Delta^1\to A$ between $g\circ h$ and $\Id_{A}$. But now we have the following homotopies :
\begin{align*}
g\circ f&\sim_{g\circ f\circ H'} g\circ f\circ g\circ h\\
&\sim_{g\circ H\circ h} g\circ h\\
&\sim_{H'} \Id_B
\end{align*} 
And so $g$ is a two-sided homotopy inverse to $f$.
\end{proof}

\begin{remarque}
In Theorem \ref{Whitehead}, the hypothesis that $A$ and $B$ are fibrant filtered spaces is quite strong. By definition, it is equivalent to the hypothesis that $\Sing_P(A)$ and $\Sing_P(B)$ are fibrant, which was necessary to apply Theorem \ref{WhiteheadSimplicial}. By Proposition \ref{ConStratFibrant}, we know that it is true for any conically stratified space. Notice that the proof of Proposition \ref{ConStratFibrant} relied solely on the fact that the underlying simplicial set is a quasi-category, which is proved in \cite[Theorem A.6.4]{HigherAlgebra}. Similarily, by \cite[Proposition 8.1.2.6]{NandLal} any homotopically stratified metric space (in the sense of \cite{Quinn}) with a finite number of strata is fibrant, and so Theorem \ref{Whitehead} also applies to those objects. One can also show that $\Sing_P(Q)$ is a quasi-category when $Q$ is a poset with the Alexandrov topology with some continuous (hence order preserving) map $Q\to P$, which implies that $Q$ is fibrant.
\end{remarque}

\section{Toward a theorem of Miller}\label{TowardATheoremOfMiller}

In \cite[Theorem 6.3]{Miller}, the author shows that a strongly stratified map between two homotopically stratified spaces is a stratified homotopy equivalence if and only if it induces homotopy equivalences between strata and holinks. In this section, we derive from Theorem \ref{Whitehead} a comparable result. We get Theorem \ref{AlmostMiller} that states that a filtered map between conically stratified PL spaces is a filtered homotopy equivalence if and only if it induces weak equivalences between strata and holinks.

To be able to state this result properly, we first need to make a few definitions. Recall that if $M$ and $N$ are two topological spaces, there is the natural compact open topology on $\Hom_{\Top}(M,N)$ we will write the corresponding topological space as $\Cont(M,N)$.

\begin{defin}
Let $A$ and $B$ be two filtered topological spaces. We define the topological space $\Cont_P(A,B)$ as the subspace of $\Cont(\Er(A),\Er(B))$, given by the inclusion $\Hom_{\Top_P}(A,B)\subseteq \Hom_{\Top}(\Er(A),\Er(B))$.
This defines a functor 
\begin{equation*}
\Cont_P\colon \Top_P^{\op}\times \Top_P\to \Top
\end{equation*}
\end{defin}

\begin{defin}[Strata]
Let $(A,\varphi\colon A\to P)$ be a filtered space and $p\in P$ be an element of the base poset. Then $A_p$, the $p$-th strate of $A$ is defined as $A_p:=\varphi^{-1}(\{p\})\subseteq A$. Equivalently, $A_p=\Cont_P(\Real{\  [p]\ }_P,A)$.
\end{defin}

\begin{defin}[Holink]
Let $p\leq q\in P$, and $A$ be a filtered space. We define $\Hol_{p,q}(A)$, the holink of the $p$-th and $q$-th strata of $A$ as
\begin{equation*}
\Hol_{p,q}(A)=\Cont(\Real{\  [p,q]\ }_P,A).
\end{equation*}
\end{defin}

We can now state the main result of this section

\begin{theo}\label{AlmostMiller}
Let $f\colon A \to B$ be a filtered map between conically stratified spaces, such that $A$ and $B$ are isomorphic to the filtered realisation of two filtered simplicial sets $X$ and $Y$ respectively. Then $f$ is a filtered homotopy equivalence if and only if it induces weak equivalences between corresponding strata and holinks.
\end{theo}

\begin{remarque}
What Theorem \ref{AlmostMiller} tells us is that, when working with conically stratified spaces, it is enough to understand the strata and the interaction between pairs of strata through the holinks to understand the homotopy types. Comparing it with Theorem \ref{Whitehead} we see that it is a much easier condition to check. Conversely, one has that any morphism between conically stratified spaces inducing weak equivalences on strata and holinks must induce isomorphisms on all homotopy groups by Theorem \ref{Whitehead} which means that such a morphism preserves the interactions between any number of strata. 
\end{remarque}

A few lemmas and definitions are needed to go from Theorem \ref{Whitehead} to Theorem \ref{AlmostMiller}. 

\begin{defin}
Let $(A,\varphi\colon A\to P)$ be a filtered topological space and $M$ be a topological space. We define $A\otimes M$ as the filtered space $(A\times M,\varphi\circ\proj_{A}\colon A\times M\to P)$ where $\proj_{A}$ is the projection on $A$.
This defines a functor 
\begin{equation*}
-\otimes -\colon \Top_P\times \Top\to \Top_P
\end{equation*}
\end{defin}



\begin{lemme}\label{HolSing}
Let $A$ be a filtered topological space, and $p\leq q\in P$. We have natural isomorphisms $\Sing(\Hol_{p,q}(A))\simeq \Map(\Real{ \ [p,q] \ }_P,A)$, and $\Sing(A_p)\simeq \Map(\Real{ \ [p]\ }_P,A)$.
\end{lemme}

\begin{proof}
Let $n\geq 0$, we have the following chain of bijections.
\begin{align*}
\Sing(\Hol_{p,q}(A))_n&\simeq \Hom_{\Top}(\Real{\Delta^n},\Hol_{p,q}(A))\\
&\simeq \Hom_{\Top}(\Real{\Delta^n},\Cont_P(\Real{ \ [p,q] \ }_P,A))\\
&\simeq \Hom_{\Top_P}(\Real{ \ [p,q] \ }_P\otimes \Real{\Delta^n},A)\\
&\simeq \Map(\Real{ \ [p,q] \ }_P,A)_n
\end{align*}
Where the third bijection comes from the following observations. Since the space underlying $\RealP{[p,q]}$ is homeomorphic to the interval $[0,1]$, there is a natural bijection 
\begin{equation}\label{EquationHolinkSing}
\Hom_{\Top}(\Real{\Delta^n},\Cont(\Real{ \ [p,q] \ },A)\simeq \Hom_{\Top}(\Real{ \ [p,q] \ }\times \Real{\Delta^n},A)
\end{equation}
Now notice that $\Hom_{\Top}(\Real{\Delta^n},\Cont_P(\Real{ \ [p,q] \ }_P,A))$ is a subset of the left hand side of 
\eqref{EquationHolinkSing} and that $\Hom_{\Top_P}(\Real{ \ [p,q] \ }_P\otimes \Real{\Delta^n},A)$ is a subset of the right hand side. Clearly, a map in the left hand side taking values in $\C^0_P(\RealP{[p,q]},A)$ is sent to a filtered map, and conversely, any filtered map in the right hand side comes from such a map. Restricting \eqref{EquationHolinkSing} thus gives the desired bijection. It is then easy to check that the collection of bijections for $n\geq 0$ respect faces and degeneracies, which gives the desired result. The same proof gives us the second isomorphism.
\end{proof}

\begin{lemme}\label{WEonNonDegenerate}
Let $f\colon X\to Y$ be a map between fibrant filtered simplicial sets.
Then $f$ induces weak equivalences $\Map(\Delta^J,X)\to \Map(\Delta^J,Y)$ for all filtered simplices $\Delta^J\to N(P)$
 if and only if it induces weak equivalences for all non-degenerate simplices of $N(P)$, i.e. simplices of the form $\Delta^{J_0}=[p_0,p_1,\dots,p_d]$ with $p_0<p_1<\dots<p_d\in P$. And the same holds for map between filtered topological spaces.
\end{lemme}

\begin{proof}
For any $\Delta^J=[q_0,\dots,q_N]$, $q_0\leq\dots\leq q_N$, there is a corresponding non-degenerate simplex of $N(P)$, $\Delta^{J_0}=[p_0,\dots,p_d]$, $p_0<\dots<p_d$, with $\{q_0,\dots,q_N\}=\{p_0,\dots,p_d\}$. Any inclusion $\Delta^{J_0}\subset \Delta^J$ is a trivial cofibration (the only map going in the other direction is a homotopy inverse). But now applying $\Map$, we get the following diagram
\begin{equation*}
\begin{tikzcd}[column sep=large]
\Map(\Delta^J,X)
\arrow{d}
\arrow{r}{\Map(\Delta^J,f)}
&\Map(\Delta^{J},Y)
\arrow{d}
\\
\Map(\Delta^{J_0},X)
\arrow{r}{\Map(\Delta^{J_0},f)}
&\Map(\Delta^{J_0},Y)
\end{tikzcd} 
\end{equation*}
Since the vertical morphisms are weak equivalences, by the two out of three axiom, we get the desired equivalence. The same proof also gives the desired result for filtered spaces.
\end{proof}

\begin{remarque}\label{SkeletonHomotopyGroups}
Let $A$ be a topological space filtered over $P$ (or a fibrant filtered simplicial set), and $\phi$ be a pointing of $A$. Lemma \ref{WEonNonDegenerate} implies that in order to compute $s\pi_n(A,\phi)$, it is enough to compute $s\pi_n(A,\phi)(\Delta^{J_0})$, for all non-degenerate simplices $\Delta^{J_0}\in N(P)_{\nd}$.
Indeed, Lemma \ref{WEonNonDegenerate} implies that if $\Delta^J=s(\Delta^{J_0})$ where $s$ is some degeneracy, then $s\pi_n(A,\phi)(\Delta^J)\simeq s\pi_n(A,\phi)(\Delta^{J_0})$. 
In particular, this means that if the length of chains in $P$ is bounded by some $d$, then the "dimension" of $s\pi_n(A,\phi)$ will be bounded by $d$. In the particular case of $P=\{*\}$ a filtered space over $P$ is just a space, and we get that $s\pi_n(A,\phi)$ is of dimension $0$, which means that we recover the genuine homotopy groups of $A$.  
\end{remarque}

\begin{defin}
Let $\Delta^{J_0}=[p_0,\dots,p_n]$ be a non-degenerate simplex of $N(P)$.
We define the outer skeleton of $\Delta^{J_0}$,  $\OSk(\Delta^{J_0})$, as the filtered simplicial set whose $0$-simplices are the $p_i$ for $0 \leq i\leq n$, whose non degenerate $1$-simplices are the $[p_i,p_{i+1}]$, and with no non-degenerate simplices of dimension greater than $2$.
\end{defin}

\begin{lemme}\label{JoyalPath}
Let $X$ be a fibrant filtered simplicial set such that $\Er(X)$ is fibrant for the Joyal model structure on $\sS$. Then, the inclusion $\OSk(\Delta^{J_0})\to (\Delta^{J_0})$ induces a weak equivalence $\Map(\Delta^{J_0},X)\to \Map(\OSk(\Delta^{J_0}),X)$.
\end{lemme}

\begin{proof}
By Theorem \ref{SimplicialModelCategory}, we already know that it is a fibration, so we need to show that it is a trivial one. Consider the following lifting problem
\begin{equation*}
\begin{tikzcd}
\partial{\Delta^n}
\arrow{r}{\alpha}
\arrow{d}
&\Map(\Delta^{J_0},X)
\arrow{d}
\\
\Delta^n
\arrow[dashrightarrow]{ur}
\arrow[swap]{r}{\beta}
&\Map(\OSk(\Delta^{J_0},X))
\end{tikzcd} 
\end{equation*}
The existence of the dotted arrow is equivalent, by the adjunction given by the simplicial category structure, to the existence of the following lift
\begin{equation}\label{LiftJoyal}
\begin{tikzcd}
\Delta^{J_0}\otimes \partial(\Delta^n)\cup_{\OSk(\Delta^{J_0})\otimes \partial(\Delta^n)} \OSk(\Delta^{J_0})\otimes \Delta^n
\arrow{r}{\alpha^{\#}\cup \beta^{\#}}
\arrow{d}
&X
\\
\Delta^{J_0}\otimes \Delta^n
\arrow[dashrightarrow]{ur}
\end{tikzcd}
\end{equation}
But since $\Er(X)$ is fibrant in the Joyal model structure, it is enough to check that the morphism on the left is a trivial cofibration in the Joyal model structure.
To see this, consider the following pushout diagram
\begin{equation*}
\begin{tikzcd}
\OSk(\Delta^{J_0})\otimes \partial(\Delta^n)
\arrow{r}{(1)}
\arrow{d}
&\Delta^{J_0}\otimes \partial(\Delta^n)
\arrow{d}
\arrow[bend left = 24]{ddr}
&\phantom{X}
\\
\OSk(\Delta^{J_0})\otimes \Delta^n
\arrow{r}{(2)}
\arrow[bend right = 12]{drr}{(3)}
& \Delta^{J_0}\otimes \partial(\Delta^n)\cup_{\OSk(\Delta^{J_0})\otimes \partial(\Delta^n)} \OSk(\Delta^{J_0})\otimes \Delta^n
\arrow{dr}{(4)}
&\phantom{X}
\\
\phantom{X}
&\phantom{X}
&\Delta^{J_0}\otimes \Delta^n
\end{tikzcd}
\end{equation*}
Now, apply $\Er$ to $(1)$, we get the morphism
\begin{equation*}
 \Sp\left(\Er(\Delta^{J_0})\right)\times \partial(\Delta^n)\times N(P)\to \Er\left(\Delta^{J_0}\right)\times \partial(\Delta^n)\times N(P)
\end{equation*}

Notice that $\Er(\OSk(\Delta^{J_0})$ is isomorphic to $\Sp(\Er(\Delta^{J_0}))$ which is the spine of the simplex $\Er(\Delta^{J_0})$. By \cite[Proposition 2.13]{Joyal}, the inclusion of the spine into the simplex is a trivial cofibration in the Joyal model structure, and as a consequence of \cite[Théorème 6.12]{Joyal} the same is true after taking a product with any simplicial set. For this reason $\Er(1)$ is a trivial cofibration in the Joyal model structure, and the same is true for $\Er(3)$ by the same argument.
Now, since pushout preserve trivial cofibration, $\Er(2)$ is also a trivial cofibration. But then by the two out of three axiom, $\Er(4)$ is a weak equivalence in the Joyal model structure. Since it is also a cofibration, it is a trivial cofibration. Going back to diagram \eqref{LiftJoyal}, and applying $\Er$, there must exists some non-filtered lift since $\Er(X)$ is fibrant in the sense of Joyal. But by the commutativity of the diagram, such a lift must be compatible with filtrations.
\end{proof}

We can now prove Theorem \ref{AlmostMiller}.
\begin{proof}[Proof of Theorem \ref{AlmostMiller}]
Let $f\colon A\to B$ be a map satisfying the hypothesis of Theorem \ref{AlmostMiller}. First suppose that $f$ is a homotopy equivalence. We get immediatly that $f$ induces homotopy equivalences on the strata.
Now, let $p\leq q \in P$, consider the map 
\begin{equation*}
\Sing(\Hol_{p,q}(f))\colon \Sing(\Hol_{p,q}(A))\to\Sing(\Hol_{p,q}(B)).
\end{equation*}
By lemma \ref{HolSing}, it is the map
\begin{equation*}
\Map(\Real{ \ [p,q] \ }_P,f)\colon \Map(\Real{ \ [p,q] \ }_P,A)\to \Map(\Real{ \ [p,q] \ }_P,B)
\end{equation*}
But this is a weak equivalence by Theorem \ref{Whitehead}, and so $f$ induces weak equivalences between the holinks.

Now, suppose $f$ induces weak equivalences on all strata and holinks. Let $p\leq q\in P$, applying $\Sing$ and using Lemma \ref{HolSing} we get the following weak equivalences
\begin{equation*}
\Map(\Real{\  [p]\ }_P,f)\colon \Map(\Real{ \ [p]\ }_P,A)\to \Map(\Real{ \ [p]\ }_P,B) 
\end{equation*}
and
\begin{equation*}
\Map(\Real{ \ [p,q] \ }_P,f)\colon \Map(\Real{ \ [p,q] \ }_P,A)\to\Map(\Real{ \ [p,q] \ }_P,B)
\end{equation*}
Now, let $\Delta^{J_0}=[p_0,\dots,p_n]$ be a non-degenerate simplex of $N(P)$. We know that $\OSk(\Delta^{J_0})$ is the colimit of its non-degenerate simplices, but those are of the form $[p_i]$ and $[p_i,p_{i+1}]$. Using the same argument as in the proof of Proposition \ref{DiagramFibration}, we get that $f$ induces weak equivalences 
\begin{equation*}
\Map(\Real{\OSk(\Delta^{J_0})}_P,f)\colon \Map(\Real{\OSk(\Delta^{J_0})}_P,A)\to \Map(\Real{\OSk(\Delta^{J_0})}_P,B).
\end{equation*}
By Remark \ref{IsomorphismOfDiagram}, $f$ induces weak equivalences 
\begin{equation*}
\Map(\OSk(\Delta^{J_0}),\Sing_P(f))\colon \Map(\OSk(\Delta^{J_0}),\Sing_P(A))\to \Map(\OSk(\Delta^{J_0}),\Sing_P(B))
\end{equation*}
Now consider the following diagram.
\begin{equation*}
\begin{tikzcd}[column sep= 8em]
\Map(\Delta^{J_0},\Sing_P(A))
\arrow{d}
\arrow{r}{\Map(\Delta^{J_0},\Sing(f))}
&\Map(\Delta^{J_0},\Sing_P(B))
\arrow{d}
\\
\Map(\OSk(\Delta^{J_0}),\Sing_P(A))
\arrow{r}{\Map(\OSk(\Delta^{J_0}),\Sing(f))}
&\Map(\OSk(\Delta^{J_0}),\Sing_P(B))
\end{tikzcd}
\end{equation*}
By \cite[Theorem A.6.4]{HigherAlgebra} together with Proposition \ref{ConStratFibrant}, the hypothesis of Lemma \ref{JoyalPath} are verified, and so the vertical morphisms are weak equivalences. By the two out of three property, the top morphism must be a weak equivalence. But then, by 
Lemma \ref{WEonNonDegenerate} $f$ induces weak equivalences :
\begin{equation*}
\Map(\Delta^{J},\Sing_P(f))\colon \Map(\Delta^{J},\Sing_P(A))\to \Map(\Delta^{J},\Sing_P(B))
\end{equation*}
For all filtered simplices $\Delta^J$. In particular, fixing some pointing $\phi\colon V\to A$ and some $n\geq 0$,  we get that $f$ induces isomorphism on homotopy groups
\begin{equation*}
 s\pi_n(f)\colon  s\pi_n(\Map(\Delta^{J},A),\phi)\to s\pi_n(\Map(\Delta^{J},B),f\circ \phi)
\end{equation*}
And so, applying Theorem \ref{Whitehead}, we get that $f$ is a homotopy equivalence.
\end{proof}

We can rephrase Theorem \ref{AlmostMiller} as follows :

\begin{corollaire}\label{CorollaireAlmostMiller}
Under the hypotheses of Theorem \ref{AlmostMiller}, the following assertions are equivalent :
\begin{enumerate}[label=\alph*)]
\item the map $f$ is a filtered homotopy equivalence,
\item the map $f$ induces weak equivalences between strata and holink,
\item for all simplices $\Delta^J\in \Delta(P)$, $f$ induces weak equivalences 
\begin{equation}\label{EquationCorollaireWhitehead}
\Map(\RealP{\Delta^J},A)\to\Map(\RealP{\Delta^J},B)
\end{equation}
\item for all simplices $\Delta^J\in\Delta(P)$ of dimension $\leq 1$, $f$ induces weak equivalences \eqref{EquationCorollaireWhitehead},
\item For all $n\geq 0$, and for all pointing of $A$, $\phi$, $f$ induces an isomorphism on the $1$-skeletons, see Remark \ref{RemarqueFilteredHomotopyGroupsSimplicialSets},

\begin{equation}\label{EquationCorollaireWhitehead2}
s\pi_n(f)\colon s\pi_n(A,\phi)\to s\pi_n(B,f\circ\phi)
\end{equation}
\item For all $n\geq 0$, and for all pointing of $A$, $\phi$, $f$ induces an isomorphism \eqref{EquationCorollaireWhitehead2}.
\end{enumerate}
\end{corollaire}

\begin{proof}
$a\Leftrightarrow b$ is the content of Theorem \ref{AlmostMiller}, and $a\Leftrightarrow f$ is the content of Theorem \ref{Whitehead}. By definition, of the filtered homotopy groups one has $c\Leftrightarrow f$, $e\Leftrightarrow d \Rightarrow b$ and $c\Rightarrow d$, which concludes the proof.
\end{proof}

\begin{remarque}
Clearly, if one wants to check that a given map is a filtered homotopy equivalence, the condition $b$ - or equivalently $d$ - are the most straightforward to check. On the other hand, to prove that no such filtered homotopy equivalence can exist, conditions $c$ and $f$ can be useful. $c$ amounts to checking the homotopy type of generalized holinks, and $f$ corresponds to checking that the group map between homotopy groups of strata and holinks coincide in the filtered spaces that are compared. For examples of use of condition $d$ see Remark \ref{RemarkElementaryNotNeeded}, and for condition $e$ see subsections \ref{SubsectionMoebius} and \ref{PseudoManifoldExample}.
\end{remarque}
\section{Applications and examples}\label{ApplicationsAndExamples}
In this section, we characterize some invariants of the filtered homotopy type, and we give a few examples of filtered topological spaces whose homotopy type can be understood.

\subsection{Invariants of the filtered diagram}
Some invariants of the filtered homotopy type can be expressed directly from the data of the filtered diagram functor (see Definition \ref{FilteredDiagram}).
We will try to describe a few of those invariants, and see how they can be understood topologically.

Let $A$ be a fibrant filtered space. There is an obvious invariant that should not be forgotten which is the homotopy type of its underlying space. Indeed, any filtered homotopy equivalence is in particular a homotopy equivalence of the underlying space. Then the next invariant to consider is the homotopy type of $D(A)$. The fact that this is a filtered homotopy invariant is the reason why the filtered homotopy groups are homotopy invariant.

Now, taking any homotopy invariant of the diagram $D(A)$ gives us a filtered homotopy invariant of $A$, for example the homotopy type of one of its pieces.

\begin{defin}\label{GeneralizedHolink}
Let $A$ be a fibrant filtered space, and let $\Delta^{J_0}\in N(P)_{\nd}$ be a non degenerate filtered simplex. Then we call generalized ($J_0$-)holink of A the topological space 
\begin{equation*}
\Hol_{J_0}(A)=\Cont_P(\Real{\Delta^{J_0}}_P,A)
\end{equation*}
\end{defin}

\begin{remarque}
In the case where $\Delta^{J_0}$ is a point (resp. a $1$-simplex), we get the corresponding strata (resp. holink) of A.
\end{remarque}

\begin{remarque}
Using the same proof as in Lemma \ref{HolSing}, we get that
\begin{equation*}
\Sing(\Hol_{J_0}(A))\simeq \Map(\Real{\Delta^{J_0}}_P,A).
\end{equation*}
This explains why we were able to define filtered homotopy groups for all filtered spaces, not just fibrant ones (see Definition \ref{TopFilteredHomotopyGroup}). Indeed, this implies that $\Map(\RealP{\Delta^{J_0}},A)$ is always a fibrant simplicial set, and there is no ambiguity when defining its homotopy groups. (As opposed to passing to a fibrant replacement, which could mean either : replace $A$ with something fibrant, or replace $\Map(\RealP{\Delta^{J_0}},A)$ with something fibrant).
\end{remarque}

This discussion leads directly to the following proposition.

\begin{prop}\label{StrataInvariants}
The homotopy types of strata, holinks, and generalized holinks are homotopy invariants of fibrant spaces.
\end{prop}

Holinks are not easy to manipulate in general, but in the case of an isolated singularity, one can replace them by the link

\begin{prop}\label{HolingIsolatedSing}
Let $\{a\}\subseteq A$ be a conically stratified space over $P=\{p_0<p_1\}$, and let $L$ be the link of $\{a\}$. There exists a weak equivalence.
\begin{equation*}
L\simeq \Hol(A)
\end{equation*}
In particular, let $\{b\}\subseteq B$ be another such space with link $M$, and $f\colon A\to B$ be a filtered homotopy equivalence. Then $L$ and $M$ are weakly homotopy equivalent.
\end{prop}

\begin{proof}
Let $U\subset A$ be a neighborhood of $a$ such that $U\simeq c(L)$. By \cite[Proposition A.7.9]{HigherAlgebra}, there exists a weak equivalence $\Hol_{p_0,p_1}(c(L))\simeq\Hol_{p_0,p_1}(U)\sim \Hol_{p_0,p_1}(A)$, where \begin{equation*}
c(L)=L\times \Real{\ [p_0,p_1]\ }_{P}/(L\times \Real{\ [p_0]\ }_{P})
\end{equation*}
is the filtered cone on $L$. 
Now, consider the map 
\begin{align*}
L&\to \Hol_{p_0,p_1}(c(L))\\
l&\mapsto \left(\begin{array}{rcl}
\Real{\ [p_0,p_1]\ }_{P}&\to& c(L)\\
t&\mapsto &(l,t)
\end{array}\right)
\end{align*}
it is a homotopy equivalence whose inverse is given by the map
\begin{align*}
\Hol_{p_0,p_1}(c(L))&\to L\\
\gamma&\mapsto \pr_L(\gamma(1/2)).
\end{align*}
Where $\pr_L\colon c(L)\setminus \{0\}\to L$ is the projection from the cylinder to $L$.
To prove the second statement, just notice that the homotopy type of the link is well defined. Then, by Proposition \ref{ConStratFibrant}, $A$ and $B$ are fibrant. So it follows from proposition \ref{StrataInvariants}.
\end{proof}

\begin{remarque}\label{HolinkBoundary}
The first part of Proposition \ref{HolingIsolatedSing} can be generalized to any kind of singular stratum, provided we add some structure.
Consider the following construction. Let $S\subset A$ be a conically stratified set such that there exists some space $L$ and a $G$-fiber bundle $\bar{c}(L)\hookrightarrow N\xrightarrow{f} S$ whose total space $N$ is homeomorphic to a closed tubular neighborhood of $S$, where $G$ is the topological group of homeomorphisms of $L$ acting levelwise on $\bar{c}(L)$. Then, the map $f$ restricts to a fiber bundle $L\hookrightarrow \partial(N)\xrightarrow{f} S$. Let $x$ be a point in $\partial(N)$, and let $U$ be a trivializing neighborhood of $f$ containing $f(x)$, with the structure map $\varphi_U\colon f^{-1}(U)\to U\times c(L)$. We have $\varphi_U(x)=(f(x),(l,1))$ for some $l\in L$. We can construct a filtered map $\gamma_x\colon \Real{\ [p_0,p_1]\ }\to N$, given by $\gamma_x(t)=\varphi_U^{-1}((f(x),(l,t))$. Since we considered a $G$-fiber bundle with $G$ the group of homeomorphism of $L$, and not all homeomorphism of $\bar{c}(L)$, the map
\begin{align*}
\partial{N}&\to \Hol(N)\\
x&\mapsto \gamma_x
\end{align*} 
will be well-defined. One can check that this will be a homotopy equivalence, by the same kind of argument as in Proposition \ref{HolingIsolatedSing}. Furthermore, applying \cite[Proposition A.7.9]{HigherAlgebra} to a decomposition of $N$ into trivializing open subsets, we get a weak equivalence $\Hol(N)\sim \Hol(A)$. So, in this particular case, the holink can be replaced with the boundary of the tubular neighborhood to compute homotopy groups.
\end{remarque}

Even though the required property in Remark \ref{HolinkBoundary} can seem very specific, we will look at examples satisfying it by construction through the remainder of this paper, and we will use this replacement to compute homotopy groups.

\subsection{Other invariants}

Some other notable filtered homotopy invariants are the Intersection (co)-homology groups. See \cite[Corollary 4.1.11, Definition 2.9.10] {Friedman}. 

From \cite[Example 1.5]{MemoireDavid}, we know that intersection homology and intersection cohomology factor through the filtered simplicial set of filtered simplices $\Sing_P$, at least when $P$ is linear. More precisely, we have the following

\begin{prop}\label{IntersectionHomologyFactors}
Let $P=\{0,\dots,n\}$ be a linear poset, $\bar{p}$ be a perversity and $k$ be an integer. There exists a functor $G$ making the following diagram commute

\begin{equation*}
\begin{tikzcd}
\sS_P
\arrow[dashrightarrow]{r}{G}
&\text{Ab}
\\
\Top_P
\arrow{u}{\Sing_P}
\arrow[swap]{ur}{I^{\bar{p}}H_k}
\end{tikzcd}
\end{equation*}
And the same is true for intersection cohomology.
\end{prop}

In the case of spaces with isolated singularities, it is known from \cite[Section 3.2]{MemoireDavid} that intersection cohomology factors through the category of filtered diagrams, however it is not known if it does in the general case. More precisely :

\begin{problem}
Under the conditions of Proposition \ref{IntersectionHomologyFactors}, does there exists a functor $H$ such that the following diagram of functors commutes?
\begin{equation*}
\begin{tikzcd}
\Fun(\Delta(P)^{\op},\sS)
\arrow[dashrightarrow]{dr}{H}
&\phantom{X}
\\
\sS_P
\arrow{u}{D}
&\text{Ab}
\\
\Top_P
\arrow{u}{\Sing_P}
\arrow{ur}{I_{\bar{p}}H^k}
&\phantom{X}
\end{tikzcd}
\end{equation*}
\end{problem}

The intersection homotopy groups defined in \cite{Gajer} can also be seen to be invariants of the filtered homotopy type. Whether or not they factor through $\Sing_P$ is unknown by the author.

\subsection{Constructing filtered spaces}\label{ConstructingFilteredSpaces}

In this section, we show how to use fiber bundles to construct filtered spaces whose filtered homotopy groups can be computed in terms of the usual homotopy groups.

Let $G$ be some discrete group, and let $Q\to M$ be a principal $G$-fiber bundle. Then taking any topological space $F$ with some $G$-action, one get a $G$ fiber bundle on $M$ of the form
\begin{equation*}
F\hookrightarrow Q\times_{G}F=E\xrightarrow{f} M
\end{equation*}
Now, any action on $F$ extends to an action on $\bar{c}(F)$, the closed cone on $F$, via $g.(x,t)=(g.x,t)$ for $g\in G, x\in F$ and $t\in [0,1]$. This gives us the following $G$-fiber bundle
\begin{equation*}
\bar{c}(F)\hookrightarrow Q\times_{G}\bar{c}(F)=N\to M
\end{equation*}
But now, there is a canonical section $M\hookrightarrow N$ coming from the homeomorphism $M\simeq Q/G$, and there is an inclusion $E\hookrightarrow N$ coming from the inclusion $F\hookrightarrow\bar{c}(F)$ which sends $x\in F$ to $(x,1)\in \bar{c}(F)$. Taking any application $g\colon E\to E'$, and taking the following pushout
\begin{equation*}
\begin{tikzcd}
E
\arrow[hookrightarrow]{d}
\arrow{r}
&E'
\arrow{d}
\\
N
\arrow{r}
&A
\end{tikzcd}
\end{equation*}
One gets a space $A$ filtered over $P=\{p_0<p_1\}$ whose singular stratum is $M$, and which admits a tubular neighborhood homeomorphic to $N$ (consider the total space of the fiber bundle with fiber the cone on $F$ truncated at height $1/2$). This space is constructed in such a way that its filtered homotopy groups can be computed easily, provided those of the non-filtered spaces involved are known.

\begin{prop}\label{HomotopyGroupFiberBundle}
Let $A$ be the filtered space obtained from the previous construction. Then $A$ is fibrant and $ s\pi_n(A)$ is given by the following diagram, for all (filtered) spaces suitably pointed.
\begin{equation*}
\begin{tikzcd}
\phantom{X}
& \pi_n(E)
\arrow{dr}{d_0= \pi_n(f)}
\arrow[swap]{dl}{d_1= \pi_n(g)}
&\phantom{X}
\\
 \pi_n(M)
&\phantom{X}
& \pi_n(E')
\end{tikzcd}
\end{equation*}
where the top row is $s\pi_n(A)([p_0,p_1])$, the bottom left is $s\pi_n(A)([p_0])$ and the bottom right is $s\pi_n(A)([p_1])$.
\end{prop}

\begin{proof}
The filtered space $A$ is conically stratified, as seen by looking at trivializing open subsets for the fiber bundle $N$.
By Remark \ref{SkeletonHomotopyGroups}, we can identify $s\pi_n(A)$ with the sub-diagram containing the $s\pi_n(A)(\Delta^{J_0})$ with $\Delta^{J_0}$ a non-degenerate simplex of $N(P)$.
 In particular, we only have to compute $\pi_n(\Map(\Real{\ [p_0,p_1]\ }_P,A)$ and $ \pi_n(\Map(\Real{\ [p_i]\ }_P,A)$, $i=0,1$. By Lemma \ref{HolSing} and Remark \ref{HolinkBoundary}, the former is isomorphic to $\pi_n(E)$, and by lemma \ref{HolSing}, $\Map(\Real{\ [p_i]\ }_P,A)$  is isomorphic to $\Sing(A_{p_i})$, where $A_{p_0}=M$ is the regular part, and $A_{p_1}=A\setminus M$ is the regular part. By construction, the map $E\to A_{p_1}$ is homotopically equivalent to $E\to E'$, and so we get the desired diagram.
\end{proof}

\begin{remarque}\label{RemarkFiberBundle}
The space $N$ is homotopically equivalent to $M$, by construction, but not in a filtered way. This provides a generic way of constructing filtered spaces which are homotopically equivalent but not filtered homotopically equivalent. Some examples are given in the following sections.
\end{remarque}

\subsection{The filtered Moebius strip and the cylinder}\label{MoebiusCylindre}
\label{SubsectionMoebius}
Let us give an example of application of the previous construction.
Let $F_i\hookrightarrow S^1\times_{\mathbb{Z}/2\mathbb{Z}}F_i\xrightarrow{f_i} S^1$, $i=1,2$ be two $\mathbb{Z}/2\mathbb{Z}$ fiber bundles over $S^1$,
with $F_1=\{0\}$ and $f_1$ the identity map, and $f_2$ the twofold covering of $S^1$ by $S^1$. That is $F_2=\{0,1\}$ and $\mathbb{Z}/2\mathbb{Z}$ acting by permutation.
Applying the previous construction to both spaces, one get two filtered spaces, $A_1$ and $A_2$. 
The first is filtered homeomorphic to a cylinder with one of its boundary as its singular stratum : $S^1\times\{0\}\hookrightarrow S^1\times [0,1] $. The second, is filtered homeomorphic to a Moebius strip with an embeded circle as its singular stratum $S^1\times \{0\}/{\sim}\hookrightarrow S^1\times [-1,1]/{\sim}$, where $(x,t)\sim (-x,-t)$. But now, applying Proposition \ref{HomotopyGroupFiberBundle}, one gets $ s\pi_1(A_1)$ (on the left) and $ s\pi_1(A_2)$ (on the right).
\begin{equation*}
\begin{tikzcd}
\phantom{X}
&\Z
\arrow{dr}{\Id}
\arrow[swap]{dl}{\Id}
&\phantom{X}
&\phantom{X}
&\Z
\arrow[swap]{dl}{\times 2}
\arrow{dr}{\Id}
&\phantom{X}
\\
\Z
&\phantom{X}
&\Z
&\Z
&\phantom{X}
&\Z
\end{tikzcd}
\end{equation*}
Both spaces are homotopy equivalent to their singular strata, by Remark \ref{RemarkFiberBundle}, but since their first homotopy group are not isomorphic, they can not be filtered homotopically equivalent by Theorem \ref{Whitehead}.

\subsection{From fiber bundles to pseudo-manifolds}\label{PseudoManifoldExample}
One can build a much richer example, involving compact orientable pseudo-manifold, using ideas from the previous one. 
Let $S^1\hookrightarrow S^1\times S^1\xrightarrow{f_1} S^1$ be a trivial fiber bundle, and let 
$S^1\coprod S^1\hookrightarrow S^1\times_{\Z/2\Z} S^1\xrightarrow{f_2} S^1$ be the $\Z/2\Z$-fiber bundle where $\Z/2\Z$ acts on the fiber $S^1\coprod S^1$ by permutation of the copy of $S^1$. Both total spaces of those fiber bundles are homeomorphic to the torus $T$. Now, writing $T\hookrightarrow FT$ for the inclusion of the torus as the boundary of the solid torus, and applying the construction of subsection \ref{ConstructingFilteredSpaces}, one gets two spaces, $A_1$ and $A_2$, as the following pushouts 
\begin{equation}\label{PushoutExamplePseudoManifold}
\begin{tikzcd}
T
\arrow{r}
\arrow[hookrightarrow]{d}
& FT
\arrow{d}
&T
\arrow{r}
\arrow[hookrightarrow]{d}
& FT
\arrow{d}
\\
N_1
\arrow{r}
&A_1
&N_2
\arrow{r}
&A_2
\end{tikzcd}
\end{equation} 
The first filtered space, $A_1$ is filtered homeomorphic to the product of a circle with a sphere with one singular point : $S^1\times (\{x_0\}\subset S^2)$, and $A_2$ is filtered homeomorphic to the total space of a $\Z/2\Z$ fiber bundle over $S^1$, whose fiber are given by a wedge of two spheres with the action of $\mathbb{Z}/2\mathbb{Z}$ permuting the spheres. Both are compact pseudo-manifolds, and they are orientable because their regular stratum are homeomorphic to an open solid torus, which is orientable. Computing intersection homology with perversity $0$ and singular homology applying Mayer-Vietoris Theorem to the decomposition given by \eqref{PushoutExamplePseudoManifold}, one gets that 
\begin{equation*}
I^0H_{k}(A_1,\Z)\simeq H_k(A_1,\Z)\simeq I^0H_k(A_2,\Z)\simeq H_k(A_2,\Z)\simeq \left\{\begin{array}{rl}
\Z& \text{ if $0\leq k\leq 3$}\\
0& \text{ if $k<0$ or $k>3$}
\end{array}\right.
\end{equation*}

So the two filtered spaces are not distinguished by those invariants. But using Proposition \ref{HomotopyGroupFiberBundle}, one computes $ s\pi_1(A_1)$ and $ s\pi_1(A_2)$ (on the left and right respectively).
\begin{equation*}
\begin{tikzcd}
\phantom{X}
&\Z\oplus \Z
\arrow[swap]{dl}{\Id\oplus 0}
\arrow{dr}{\Id\oplus 0}
&\phantom{X}
&\phantom{X}
&\Z\oplus\Z
\arrow[swap]{dl}{\times 2\oplus 0}
\arrow{dr}{\Id\oplus 0}
&\phantom{X}
\\
\Z
&\phantom{X}
&\Z
&\Z
&\phantom{X}
&\Z
\end{tikzcd}
\end{equation*}
So, from Theorem \ref{Whitehead}, one deduces that there exists no filtered homotopy equivalence between $A_1$ and $A_2$.

\subsection{Diagram of groups}
Let $G$ be a group and $f\colon H\to G$ be the inclusion of any subgroup. Then, from $f$, one can construct a $G$ fiber bundle of the form,
\begin{equation*}
G/H\hookrightarrow EG\times_{G}G/H\to BG
\end{equation*}
Where $BG$ is a classifying space for $G$ and $EG\to BG$ is its associated universal cover. The total space $EG\times_{G}G/H$ is a classifying space for $H$ since it is the quotient of a contractible space by a free action of $H$. Choosing any homomorphism $g\colon H\to G'$, and applying the construction of subsection \ref{ConstructingFilteredSpaces}, we get the filtered space $A(f,g)$. Applying Proposition \ref{HomotopyGroupFiberBundle}, we get that the first filtered homotopy group of $A(f,g)$ is of the form
\begin{equation*}
\begin{tikzcd}
\phantom{X}
&H
\arrow{dr}{g}
\arrow[swap]{dl}{f}
&\phantom{X}
\\
G
&\phantom{X}
&G'
\end{tikzcd}
\end{equation*}
and that all higher filtered homotopy groups are trivial.
\begin{prop}\label{WhiteheadDiagrams}
Let $G_i,H_i,G'_i$ be groups, $f_i\colon H_i\to G_i$ be inclusions and $g_i\colon H_i\to G'_i$ be homomorphisms for $i=1,2$. Then, if $A(f_1,g_1)$ and $A(f_2,g_2)$ are filtered homotopy equivalent, there exists three isomorphisms $\varphi_H\colon H_1\to H_2$, $\varphi_G\colon G_1\to G_2$, and $\varphi_{G'}\colon G'_1\to G'_2$ such that the following diagram commutes
\begin{equation*}
\begin{tikzcd}
G_1
\arrow{d}{\varphi_G}
&H_1
\arrow{d}{\varphi_H}
\arrow{r}{g_1}
\arrow[swap]{l}{f_1}
& G'_1
\arrow{d}{\varphi_{G'}}
\\
G_2
&H_2
\arrow{r}{g_2}
\arrow{l}[swap]{f_2}
& G'_2
\end{tikzcd}
\end{equation*}
Conversely, if such a diagram exists, and if in addition $A(f_1,g_1)$ and $A(f_2,g_2)$ can be triangulated in a way compatible with their stratification, there exists a homotopy equivalence between $A(f_1,g_1)$ and $A(f_2,g_2)$.
\end{prop}

\begin{proof}
By theorem \ref{Whitehead} any homotopy equivalence between $A(f_1,g_1)$ and $A(f_2,g_2)$ will induce such an isomorphism of diagrams. Conversely, suppose three such isomorphisms are given, since the construction of a classifying space can be done functorially, there are maps $B(\varphi_G)\colon BG_1\to BG_2$, $EG_1\times_{G_1}G_1/H_1\to EG_1\times_{G_2}G_2/H_2$ and $B(\varphi_{G'})\colon BG_1'\to BG_2'$. Furthermore, the two first maps induce a map between the tubular neighborhood obtained from the construction of subsection \ref{ConstructingFilteredSpaces}. Taking the pushout of those maps, we get $h\colon A(f_1,g_1)\to A(f_2,g_2)$, which induce isomorphisms on homotopy groups by construction. But then, by theorem \ref{Whitehead}, $h$ is a homotopy equivalence.
\end{proof}

\begin{exemple}
Looking back at the example of subsection \ref{MoebiusCylindre}, one sees that it was in fact the result of one such construction. Indeed choosing $f_1=\Id\colon \Z\to \Z$ and $f_2=\times 2\colon \Z\to \Z$, (and $g_i=\Id$ for $i=1,2$) one recovers the stratified spaces from subsection \ref{MoebiusCylindre}. But then the observation that they can not be filtered homotopic is just a particular case of Proposition \ref{WhiteheadDiagrams}.
\end{exemple}

\begin{remarque}
Though it is tempting to say that Proposition \ref{WhiteheadDiagrams} implies that there exists a well defined classifying space for any diagram of groups of a suitable form, it should be noted that this statement is not implied by Proposition \ref{WhiteheadDiagrams}. Indeed, Proposition \ref{WhiteheadDiagrams} gives us information only for spaces obtained from the construction of this subsection. It is not known whether a space with the same filtered homotopy groups as $A(f,g)$ will be of the filtered homotopy type of $A(f,g)$.
\end{remarque}

Proposition \ref{WhiteheadDiagrams} allows us to produce many examples of filtered spaces having homotopy equivalent strata but that are not filtered homotopy equivalent to each others. All one needs is a pair of inclusions of subgroups $H_1\hookrightarrow G_1$ and $H_2\hookrightarrow G_2$ such that $H_1$ is isomorphic to $H_2$ and $G_1$ is isomorphic to $G_2$ but there exists no commutative diagram of the form
\begin{equation*}
\begin{tikzcd}
H_1
\arrow[hookrightarrow]{d}
\arrow{r}{\varphi_H}
&H_2
\arrow[hookrightarrow]{d}
\\
G_1
\arrow{r}{\varphi_G}
&G_2
\end{tikzcd}
\end{equation*}
where $\varphi_H$ and $\varphi_G$ are both isomorphism. For the reader who may wonder how easy it is to produce such examples, a family of such pairs is provided in the next example. 

\begin{exemple}
Let $q$ be a prime number, and $n\geq 3$ an integer. Let $G$ be the group of upper unitriangular matrices of size $n$ with coefficients in $\Z/q\Z$. Then, let $H_1$ and $H_2$ be the subgroups generated by $M_1$ and $M_2$ respectively where
\begin{equation*}
M_1= 
\begin{bmatrix}
1 		&0		& \dots 	& 1\\
0 		&1		& \dots 	& 0\\
\vdots	&\vdots	& \ddots	& \vdots\\
0		&0		& \dots		& 1
\end{bmatrix}
\text{ and } M_2=
\begin{bmatrix}
1 		&0		& \dots 	& 0\\
0 		&1		& \dots 	& 1\\
\vdots	&\vdots	& \ddots	& \vdots\\
0		&0		& \dots		& 1
\end{bmatrix}.
\end{equation*}
We have $H_1\simeq H_2\simeq \Z/q\Z$. But, notice that $H_1$ is included in the center of $G$, whereas $H_2$ is not. In particular, any automorphism of $G$ will send $H_1$ to a subgroup of the center of $G$, but not $H_2$, and so there exists no isomorphisms $\varphi_G\colon G\to G$ and $\varphi_H\colon H_1\to H_2$ such that the following diagram commutes.
\begin{equation*}
\begin{tikzcd}
H_1
\arrow{r}{\varphi_H}
\arrow[hookrightarrow]{d}
&H_2
\arrow[hookrightarrow]{d}
\\
G
\arrow{r}{\varphi_G}
&G
\end{tikzcd}
\end{equation*}
And so, by proposition \ref{WhiteheadDiagrams}, the filtered spaces corresponding to those inclusions are not filtered homotopy equivalent. 
\end{exemple}
\appendix

\section{Characterizing fibrations in a presheaf category}
\label{AppendixA}
In this section, we prove that under certain conditions, one can characterize the class of fibrations in a Cisinski model structure \cite[Théorème 1.3.22]{Cisinski} (Theorem \ref{TheoAppendixA}). We will follow a proof of Cisinski for the case of simplicial sets (see \cite[Proposition 2.1.41]{Cisinski}), and extend it to any category of presheaf over an Eilenberg-Zilber category, equipped with a pair of adjoint functors satisfying some properties similar to the pair of functor $(\sd,\Ex)$. We refer the reader to \cite[Definition 1.3.1]{CisinskiHigherCategories} for a definition of Eilenberg-Zilber categories. We will only need the following facts about those categories :
\begin{itemize}
\item The category $\Delta$ is an Eilenberg-Zilber category \cite[Example 1.3.2]{CisinskiHigherCategories}
\item If $A$ is an Eilmenberg-Zilber category, and $X$ is a presheaf over $A$, $A/X$ is an Eilenberg-Zilber category \cite[Example 1.3.3]{CisinskiHigherCategories}. In particular, $\Delta(P)$ is an Eilenberg-Zilber category,
\end{itemize}
together with the following lemma :
\begin{lemme}\label{LemmeEilenberg}
Let $F,G\colon A\to \mathcal{C}$ be two functors from an Eilenberg-Zilber category to a model category. Write $F_{!}$ and $G_{!}$ their Kan extensions. Assume that $F_{!}$ and $G_{!}$ map monomorphisms to cofibrations. Then, if a natural transformation $u\colon F\to G$ induces weak equivalences $u_a\colon F(a)\to G(a)$ for all representable presheaves $a$, it induces weak equivalences $u_X\colon F_{!}(X)\to G_{!}(X)$ for all presheaves.
\end{lemme}

\begin{proof}
Consider the class of objects of $\widehat{A}$ such that $u_X$ is a weak equivalence. By assumption, it contains all representable presheaves. By \cite[Corollary 2.3.16, 2.3.18 and 2.3.29]{CisinskiHigherCategories}, this class is saturated by monomorphisms \cite[Definition 1.3.9]{CisinskiHigherCategories}, thus by \cite[Corollary 1.3.10]{CisinskiHigherCategories}, it must contain all presheaves.
\end{proof}

\begin{theo}\label{TheoAppendixA}
Let $A$ be a small Eilenberg-Zilber category, and consider $\widehat{A}$ as a model category, where the model structure comes from some homotopical data $(I,\Lambda)$. Assume that there exists a set of cofibration $\An$, such that $\Lambda=l(r(\An))$, and such that for all $j\colon K\to L\in \An$, $K$ and $L$ are compact objects of $\widehat{A}$. Assume in addition that we are given:
\begin{itemize}
\item a functor $A\to \widehat{A}$ whose Kan extension is written $S\colon \widehat{A}\to\widehat{A}$,
\item a right adjoint to $S$, $E\colon \widehat{A}\to\widehat{A}$,
\item a natural transformation $\alpha\colon S\to \Id$.
\end{itemize}
Write $\beta\colon \Id\to E$ for the natural transformation adjoint to $\alpha$, and write $\Ei$ for the following colimit:
\begin{equation*}
\begin{tikzcd}
\Id
\arrow{r}{\beta}
&E
\arrow{r}{\beta_{E(-)}}
&E^2
\arrow{r}{\beta_{E^2(-)}}
&\dots
\arrow{r}
&E^n
\arrow{r}
&E^{n+1}
\arrow{r}
&\dots
\end{tikzcd}
\end{equation*}
If in addition 
\begin{itemize}
\item the functor $S$ preserves monomorphisms and anodyne extensions,
\item for all representable presheaf $a$, $\alpha_a\colon S(a)\to a$ is an absolute weak equivalence (see \cite[Definition 1.3.55]{Cisinski}) and an epimorphism,
\item for all presheaf $X$, $\Ei(X)$ is fibrant,
\end{itemize}
then, the class of naive fibrations and the class of fibrations coincide. In this case, $\An$ is a set of generating trivial cofibrations, and we have equality between the following classes
\begin{equation*}
l(r(\An))=\Lambda=l(\{\text{naive fibrations}\})=l(\{\text{fibrations}\})=\{\text{trivial cofibrations}\}
\end{equation*}
\end{theo}

\begin{lemme}For any presheaf of $\widehat{A}$, $X$, $\alpha_X\colon S(X)\to X$ is a weak equivalence. Furthermore, the functor $S$ preserves trivial cofibrations.
\end{lemme}

\begin{proof}
The first assertion follows from the previous lemma by taking $F_{!}=S$ and $G_{!}=\Id$. Let $f\colon X\to Y$ be a trivial cofibration. Consider the following commutative diagram :
\begin{equation*}
\begin{tikzcd}
S(X)
\arrow{r}{S(f)}
\arrow[swap]{d}{\alpha_X}
&S(Y)
\arrow{d}{\alpha_Y}
\\
X
\arrow{r}{f}
&Y
\end{tikzcd}
\end{equation*}
Since $f$ and the vertical maps are weak equivalences, so is $S(f)$ by two out of three. Since $S(f)$ preserves cofibrations by assumption, $S(f)$ preserves trivial cofibrations.
\end{proof}

\begin{defin}
Let $Y$ be some fixed presheaf of $\widehat{A}$. Define $S^Y\colon \widehat{A}/Y\to\widehat{A}/Y$ as
\begin{equation*}
S^Y(X,f\colon X\to Y)=(S(X),f\circ\alpha_X\colon S(X)\to Y),
\end{equation*}
and the natural transformation $\alpha^Y\colon S^Y\to \Id$ as $\alpha^Y_{(X,f\colon X\to Y)}=\alpha_X$. Similarly, define $E^Y$ as
\begin{equation*}
E^Y(X,f\colon X\to Y)=(Y\times_{E(Y)}E(X),\pr_Y\colon Y\times_{E(Y)}E(X)\to Y),
\end{equation*}
and the natural transformation $\beta^Y\colon \Id\to E^Y$ as $\beta^Y_{(X,f)}=(f,\beta_X)\colon (X,f)\to (Y\times_{E(Y)}E(X),\pr_Y)$.
\end{defin}

\begin{lemme}
Let $Y$ be some presheaf of $\widehat{A}$. $E^Y$ is right adjoint to $S^Y$. Furthermore, if $f\colon X\to Y$ is a naive fibration, then $\beta^Y_{(X,f)}\colon X\to Y\times_{E(Y)}E(X)$ is an anodyne extension in $\widehat{A}$.
\end{lemme}

\begin{proof}
The first part of the lemma follows from the definitions. For the second part, let $\C=\widehat{A}/Y=\widehat{A/Y}$ be the category of presheaves over $Y$, and consider on $\C$ the Cisinski model structure induced by the one on $\widehat{A}$. In particular, the class of anodynes extension of $\C$ is $\Lambda/Y$. By assumption, $\alpha^Y_{(a,f\colon a\to Y)}$ is a weak equivalence for all representable presheaf $a\in \widehat{A}$, thus by lemma \ref{LemmeEilenberg}, $\alpha^Y_{(X,f\colon X\to Y)}$ is a weak equivalence for all objects of $\C$. Let $g\colon (X,f)\to (X',f')$ be some map in $\C$, consider the following commutative diagram.
\begin{equation*}
\begin{tikzcd}
S^Y(X,f)
\arrow{r}{S^Y(g)}
\arrow[swap]{d}{\alpha^Y_{(X,f)}}
&S^Y(X',f')
\arrow{d}{\alpha^Y_{(X',f')}}
\\
(X,f)
\arrow{r}{g}
&(X',f')
\end{tikzcd}
\end{equation*}
Since the vertical maps are weak equivalences in $\C$, $S^Y(g)$ is a weak equivalence if and only if $g$ is. Since $S^Y$ preserves cofibration, this implies that $S^Y$ is a left Quillen functor. Let $f\colon X\to Y$ be some naive fibration in $\widehat{A}$, this means that $(X,f)$ is a fibrant object in $\C$. Now by \cite[Corollary 1.4.4 (b)]{Hovey}, this implies that $\beta^Y_{(X,f)}\colon (X,f)\to E^Y(X,f)$ is a weak equivalence.  Furthermore, since for all representable presheaf $S(a)\to a$ is an epimorphism, $\beta_Y\colon Y\to E(Y)$ must be a monomorphism, which implies that $\beta^Y_{(X,f)}\colon (X,f)\to E^Y(X,f)$ is a monomorphism, so it is a trivial cofibration. Now since $S$ preserve anodyne extensions, $E$ preserves naive fibrations, which means $E(f)\colon E(X)\to E(Y)$ is a naive fibration, and so is $Y\times_{E(Y)}E(X)\to Y$. In particular, $\beta^Y_{(X,f)}\colon (X,f)\to E^Y(X,f)$ is a trivial cofibration between fibrant objects. By \cite[Lemme 1.3.39]{Cisinski}, it is an anodyne extension in $\C$. But then, the underlying map $\beta^Y\colon X\to Y\times_{E(Y)}E(X)$ is an anodyne extension of $\widehat{A}$.
\end{proof}

\begin{lemme}\label{LemmeBetaAnodyneExtension}
If $f\colon X\to Y$ is a naive fibration, $X\to Y\times_{\Ei(Y)}\Ei(X)$ is an anodyne extension.
\end{lemme}

\begin{proof}
The map $X\to Y\times_{\Ei(Y)}\Ei(X)$ is the transfinite composition
\begin{equation}\label{EquationTransfiniteComposition}
X\to Y\times_{E(Y)}E(X)\to Y\times_{E^2(Y)}E^2(X)\to \dots \to Y\times_{E^n(Y)}E^n(X)\to \dots
\end{equation}
Using the identification
\begin{equation*}
Y\times _{E(Y)}E\left(Y\times_{E^n(Y)}E^n(X)\right)\simeq Y\times_{E^{n+1}(Y)}E^{n+1}(X)
\end{equation*}
and the fact that $E^n$ preserves naive fibrations, we get by the previous lemma that (\ref{EquationTransfiniteComposition}) is a transfinite composition of anodyne extensions.
\end{proof}

\begin{lemme}\label{LemmeEiPreservesFibrations}
If $f\colon X\to Y$ is a naive fibration, $\Ei(f)\colon \Ei(X)\to\Ei(Y)$ is a fibration.
\end{lemme}

\begin{proof}
By assumption, $\Ei(X)$ and $\Ei(Y)$ are fibrant, so by \cite[Proposition 1.3.36]{Cisinski}, it is enough to show that $\Ei(f)$ is a naive fibration. Consider the following lifting problem with $j\colon K\to L\in \An$
\begin{equation*}
\begin{tikzcd}
K
\arrow[swap]{d}{j}
\arrow{r}{k}
&\Ei(X)
\arrow{d}{\Ei(f)}
\\
L
\arrow{r}{l}
&\Ei(Y)
\end{tikzcd}
\end{equation*}
Since $K$ and $L$ are assumed to be compact, there must exist some integer $n$ such that the previous lifting problem factors as follows.
\begin{equation*}
\begin{tikzcd}
K
\arrow[swap]{d}{j}
\arrow{r}{k}
&E^n(X)
\arrow{d}{E^n(f)}
\\
L
\arrow{r}{l}
&E^n(Y)
\end{tikzcd}
\end{equation*}
Since $E$ preserves naive fibration, so does $E^n$ and a lift must exist. 
\end{proof}

\begin{proof}[Proof of Theorem \ref{TheoAppendixA}]
By \cite[Proposition 1.3.47]{Cisinski}, it is enough to show that any naive fibration $f\colon X\to Y$ admits a factorisation of the form $f=qj$ where $j$ is an anodyne extension and $q$ is a fibration. Consider the following factorisation.
\begin{equation*}
\begin{tikzcd}
X
\arrow{r}{j}
\arrow[swap]{dr}{f}
&Y\times_{\Ei(Y)}\Ei(X)
\arrow{d}{q}
\arrow{r}
&\Ei(X)
\arrow{d}{\Ei(f)}
\\
&Y
\arrow{r}
&\Ei(Y)
\end{tikzcd}
\end{equation*}
By lemma \ref{LemmeEiPreservesFibrations}, $q$ is a fibration since it is the pullback of a fibration. By lemma \ref{LemmeBetaAnodyneExtension}, $j$ is an anodyne extension.
\end{proof}
\bibliographystyle{alpha}
\bibliography{biblio}

\end{document}